\RequirePackage{fix-cm}
\documentclass{svjour3}                     % onecolumn (standard format)
\smartqed  % flush right qed marks, e.g. at end of proof
\usepackage{graphicx}
\usepackage{comment}
\usepackage{bbm}
\usepackage{svg}
\usepackage{blindtext}
\usepackage[inline]{enumitem}
\usepackage{tikz, tikz-3dplot}
\usepackage{tikz-qtree}
\usetikzlibrary{positioning,calc}
\usepackage{etoolbox}
\usepackage{xcolor}
\usepackage{algorithm}
\usepackage{algpseudocode}
\usepackage{amsmath,amsfonts,bm,mathtools}
\usepackage{pgfplotstable}
\usepackage{pgfplots}
\usepackage{ulem}
\usepackage{relsize}
\usepackage{caption}
\usepackage{subcaption}
\usepackage{amssymb}
\usepackage{appendix}
\usepackage{multicol}
\usepackage{multirow}

\let\Halmos\qed

\let\ml=\mathlarger

\DeclareMathOperator{\MC}{MC}
\DeclareMathOperator{\OA}{OA}
\DeclareMathOperator{\MCOA}{MC-OA}
\DeclareMathOperator{\OAMC}{OA-MC}
\DeclareMathOperator{\OARLT}{OA-RLT}
\DeclareMathOperator{\OAhull}{OA-Hull}
\DeclareMathOperator{\proj}{proj}
\DeclareMathOperator{\epi}{epi}
\DeclareMathOperator{\RLT}{RLT}
\DeclareMathOperator{\conv}{conv}
\DeclareMathOperator{\conc}{conc}
\DeclareMathOperator{\interior}{int}
\DeclareMathOperator{\graph}{gr}
\DeclareMathOperator{\cone}{cone}
\DeclareMathOperator{\aff}{aff}
\DeclareMathOperator{\cl}{cl}
\DeclareMathOperator{\vertex}{vert}
\DeclareMathOperator{\dims}{dim}
\DeclareMathOperator{\hypo}{hyp}
\DeclareMathOperator{\sepP}{SEP-P}
\DeclareMathOperator{\sepQ}{SEP-Q}
\DeclareMathOperator{\range}{range}
\DeclareMathOperator{\dom}{dom}
\DeclareMathOperator{\pr}{Pr}
\DeclareMathOperator{\cp}{\mathcal{CP}}

\DeclareMathOperator{\E}{\mathrm{E}}
\DeclareMathOperator{\R}{\mathbb{R}}
\DeclareMathOperator{\Z}{\mathbb{Z}}
\DeclareMathOperator{\supp}{\mathrm{supp}}
\DeclareMathOperator{\rec}{rec}
\DeclareMathOperator{\ri}{ri}

\DeclareMathOperator{\QP}{QP}
\DeclareMathOperator{\BQP}{BQP}
\DeclareMathOperator{\MP}{MP}
\DeclareMathOperator{\psd}{\mathbb{S}_+^{n+1}}

\DeclareMathOperator{\st}{s.t.}
\DeclareMathOperator{\Tr}{Tr}
\DeclareMathOperator{\SM}{\mathbb{S}}

\DeclareMathOperator{\Diag}{Diag}

\newcommand{\revision}{\textcolor{black}}

\def \X {\mathcal{X}}

\definecolor{links}{RGB}{204,36,29}
\usepackage[colorlinks=true,breaklinks=true,bookmarks=true,urlcolor=links,citecolor=links,linkcolor=links,bookmarksopen=false,draft=false]{hyperref}
\def\EMAIL#1{\href{mailto:#1}{#1}}% When hyperref is used, otherwise outcomment
\def\URL#1{\href{#1}{#1}}         % When hyperref is used, otherwise outcomment

\begin{document}

%\title{On relaxing composite functions using inner function structure: Polynomially equivalent formulations\thanks{This work was supported in part by NSF CMMI Grant 1727989}}
\title{Convexification techniques for fractional programs}

    \titlerunning{Convexification techniques for fractional programs}


\author{  Taotao He \and Siyue Liu   \and Mohit Tawarmalani
}


% \institute{
% Taotao He \at
%               Antai College of Economics and Management, Shanghai Jiao Tong University \\
%               \EMAIL{hetaotao@sjtu.edu.cn}           %  \\
% %             \emph{Present address:} of F. Author  %  if needed
%            \and
%             \at
%               \\
% }

\institute{Taotao He \at
              Antai College of Economics and Management, Shanghai Jiao Tong University \\
              \EMAIL{hetaotao@sjtu.edu.cn}, \URL{https://taotaoohe.github.io/}           %  \\
           \and
           Siyue Liu \at 
           Tepper School of Business, 
           Carnegie Mellon University\\
           \EMAIL{siyueliu@andrew.cmu.edu}
           \and
           Mohit Tawarmalani \at
              Mitchell E. Daniels, Jr. School of Business, Purdue University\\
							\EMAIL{mtawarma@purdue.edu}, \URL{http://web.ics.purdue.edu/$\sim$mtawarma/}
}

\date{Received: date / Accepted: date}

\maketitle

\begin{abstract}
This paper develops a correspondence relating convex hulls of fractional functions with those of polynomial functions over the same domain. Using this result, we develop a number of new reformulations and relaxations for fractional programming problems. First, we relate $0\mathord{-}1$ problems involving a ratio of affine functions with the boolean quadric polytope, and use inequalities for the latter to develop tighter formulations for the former. Second, we derive a new formulation to optimize a ratio of quadratic functions over a polytope using copositive programming. Third, we show that univariate fractional functions can be convexified using moment hulls. Fourth, we develop a new hierarchy of relaxations that converges finitely to the simultaneous convex hull of a collection of ratios of affine functions of $0\mathord{-}1$ variables. Finally, we demonstrate theoretically and computationally that our techniques close a significant gap relative to state-of-the-art relaxations, require much less computational effort, and can solve larger problem instances.

\end{abstract}

%------- shortcuts-------------
\def \s {\rho}
\def \t {\sigma}
\def \S {\mathcal{S}}
\def \F {\mathcal{F}}
\def \G {\mathcal{G}}

\section{Introduction}
Fractional terms arise in many combinatorial and nonconvex optimization problems \cite{frenk2005fractional,stancu2012fractional,borrero2017fractional} to model choice behavior, financial and performance ratios, variational principles, engineering models, and geometric characteristics. For example, Luce-type choice models~\cite{luce2012individual} are used in feature selection~\cite{nguyen2009optimizing}, assortment optimization~\cite{mendez2014branch}, and facility location problems~\cite{barros1998discrete}; financial and performance ratios are used to measure return on investment~\cite{konno1989bond,cheridito2013reward}, \revision{fairness~\cite{xinying2023guide}}, 
quality of regression models~\cite{gomez2021mixed}, efficiency of queuing systems~\cite{berman1985optimal}, performance of restricted policies~\cite{azar2003optimal}, quality of set-covers~\cite{arora1977set}, consistent biclustering~\cite{busygin2005feature}, and other graph properties such as density~\cite{lanciano2023survey} or connectedness; engineering models capture physical processes such as Underwood equations in distillation configuration design~\cite{underwood1949fractional} and exergy models in chemical processes~\cite{gooty2023exergy}; geometrical measures are used to measure flatness of shapes and condition numbers of matrices~\cite{chen2011minimizing}.

We are interested in optimization problems of the following form:
\begin{alignat*}{2}
    &\max_{x\in \X} &\quad& \sum_{i=1}^m \frac{f_i(x)}{g_i(x)},
\end{alignat*}
where $f_i:\R^n\rightarrow \R$ and $g_i:\R^n\rightarrow \R$ are polynomial functions, and $\X$ is a set described using polynomial inequalities. We will particularly be interested in the following special cases: (i) $\X = \bigl\{x\in \R^n\bigm| Cx\le d,\ x_i\in \{0,1\} \text{ for } i\le p\bigr\}$, where $C\in \R^{r\times n}$ and $d\in \R^{r}$, and (ii) functions $f_i(\cdot)$ and $g_i(\cdot)$ are affine and/or quadratic functions.

The literature on fractional programming can be divided into two categories. \revision{The first} category consists of problems \revision{that} can be solved in polynomial time. Two prominent examples include problems treated in \cite{charnes1962programming} and \cite{megiddo1978combinatorial}. These two seminal works focused on fractional programming problems where the objective function is a ratio of affine functions and the feasible region $\X$ is described, respectively, by linear inequalities or a combinatorial set over which linear functions can be optimized easily. The second category of problems do not admit polynomial time algorithms. A simple example is where the objective is to maximize a sum of two ratios of affine functions defined over a polytope~\cite{matsui1996np}. The resulting  problem is \revision{NP-hard}. Many applied problems belong to the second category, and are solved using branch \& bound algorithms that rely on convexification or reformulation strategies~\cite{falk1994image,tawarmalani2002global,benson2002using,mendez2014branch,borrero2017fractional,shen2018fractional,sen2018conic,mehmanchi2019fractional}. However, an intriguing gap exists in the literature. Specifically, if  problem instances of the first category are solved using techniques adopted for the second category, the resulting relaxations are not tight and the guarantee of polynomial time convergence is lost. This is arguably because the results that enable polynomial-time solution algorithms do not directly yield convexification results. In this paper, we discover a projective transformation that permits these results to be viewed from a convexification lens. Leveraging this insight, we improve the quality of relaxations for a variety of fractional programming problems that belong to the second category.

%However, if these problems are relaxed using standard nonconvex optimization techniques, there is a relaxation gap. This becomes important since many applications require more general structure than required by these techniques. For example, if the objective is a sum of ratios of affine functions instead of a single ratio, the problem becomes NP-Hard, and the above techniques do not apply directly. 

%However, these techniques have not been viewed from the convexification lens. In particular, if a problem is given to state-of-the-art nonconvex optimization solvers they treat these problems as nonconvex and create relaxations with gap and use branch-and-bound techniques to close the gap \cite{}. This raises the question of whether techniques in optimization software can automatically leverage these results to improve relaxation quality. We answer this question in the affirmative. 

Below, we highlight a few of our contributions:
\begin{enumerate}
    \item We reduce the task of convexifying fractional programming problems to that of convexifying semi-algebraic sets relying on a projective transformation that is of independent interest. The reduction unifies  approaches to solve single-ratio fractional optimization problems with those to create convex relaxations for harder problems.
    \item We establish a precise connection between 0-1 linear fractional programming, where the objective is a ratio of affine functions, and the convexification of the boolean quadric polytope (BQP) and/or cut polytope. This allows us to leverage valid inequalities for BQP \cite{barahona1986cut,padberg1989boolean} to improve relaxations for $0\mathord{-}1$ fractional programming problems. Our computational results demonstrate that the developed relaxations are significantly tighter on a collection of assortment optimization problems.
    \item We show that the problem of optimizing a ratio of quadratic functions over a polytope, possibly with binary constraints, can be reformulated into a copostive programming problem. This result extends prior results in copositive optimization literature \cite{burer2009copositive}.  
    \item We discover a one-to-one correspondence between moment hulls~\cite{karlin1951geometry,schmudgen2017moment} and the convex hull of inverses of univariate linear forms. These inverses occur in practical design problems involving chemical separations \cite{underwood1949fractional,gooty2022advances,gooty2023exergy} and we show that the resulting relaxations are significantly tighter than those typically used in the literature.
    \item When the objective is a sum-of-ratios of binary variables, we establish a finite hierarchy of reformulations that converges to the convex hull of the problem. 
    %Each level of the hierarchy limits the degree of multilinear terms introduced as variables in the formulation. Then, we decompose each fraction into simpler terms and express the multilinear terms using these fractional terms. 
    The result reveals an intriguing decomposition. We show that convexifying each ratio individually, in a specific way, leads to the simultaneous hull of multiple ratios. 
    \item We provide explicit convex hulls for \revision{small-dimensional} sets extending previous results~\cite{tawarmalani2001semidefinite,burer2015gentle} that are used in solvers to develop relaxations of factorable functions.
    \item We demonstrate theoretically and computationally that our reformulations and relaxations dominate state-of-the-art schemes in the literature~\cite{quesada1995global,tawarmalani2004global,mehmanchi2019fractional}. For example, on a collection of assortment planning problems \cite{mehmanchi2021solving} and sets arising in chemical process design, we show that the new relaxations close approximately 50\% of the gap. Our reformulations solve faster and enable an exact solution of larger problem instances.
\end{enumerate}
    
%We also consider a set that arises in models of distillation configuration design. We show that using our correspondence to moment hull, the new relaxations close over 50\% of the gap on most instances and, on 80\% of the instances the gap is 20-25\% of the original gap.

The paper is organized as follows. In Section~\ref{section:mainresult}, we  establish a precise correspondence between fractional optimization and polynomial optimization. This allows us to unify the techniques of \cite{charnes1962programming} and \cite{megiddo1978combinatorial} within a common framework. In Section~\ref{section:BQP}, we exploit inequalities for BQP to tighten relaxations for $0\mathord{-}1$ fractional programs. In Section~\ref{section:copositive}, we consider problems involving optimization of a ratios of quadratic functions and, when the domain is a polytope, provide a reformulation using copositive programming. In Section~\ref{section:univariatefractional}, we establish a correspondence between simultaneously convexifying inverses of univariate affine functions and the moment hull. In Section~\ref{section:general-linear-frac}, we develop a hierarchy of relaxations that converges to the convex hull of a sum-of-ratios optimization problem over binary variables. In Section~\ref{section:computation}, we propose a relaxation for linear fractional programming and prove that our relaxation dominates existing relaxations in the literature. We also present computational experiments with $0\mathord{-}1$ fractional programs and a set that appears in the design of distillation trains.  
Concluding remarks are provided in Section~\ref{section:conclusion}. All missing proofs are included in Appendix~\ref{app:proofs}. 

\textbf{Notation}
For a positive integer $n$, let $[n]$ denote $\{1, \ldots, n \}$.  We use $\R$ to denote the set of real numbers, and $\R_+$ (resp. $\R_{++}$) to denote the set of non-negative (resp. positive) real numbers. For a set $S \subseteq \R^n$, $\conv(S)$ denotes the convex hull of $S$, $\rec(S)$ denotes the recession cone of $S$, \revision{$\ri(S)$ denote the relative interior of $S$}, and $\cl S$ denote the closure of $S$. For $\lambda \geq 0$, $\lambda S := \{\lambda x \mid x \in S\}$ if $\lambda >0$ and $\lambda S = \rec(S)$ if $\lambda = 0$. We will write $\frac{x}{\alpha^\intercal x}$, where $x\in\R^n$, to denote that each element of $x$ is divided by $\alpha^\intercal x$. For any $d$, we use $\mathbb{S}^d$ to denote the set of symmetric $d \times d$ matrices, $\mathbb{S}^d_+$ to denote the set of symmetric positive semidefinite matrices. We will use $X \succeq 0$ to denote $X \in \mathbb{S}^d_+$. 
For two $m \times n$ matrices $X$ and $Y$, its inner product is given by $\langle X, Y \rangle = \Tr(XY^\intercal) = \sum_{i \in [m]}\sum_{j \in [n]}X_{ij}Y_{ij}$.
% \[
% \langle X, Y \rangle = \Tr(XY^\intercal) = \sum_{i \in [m]}\sum_{j \in [n]}X_{ij}Y_{ij}. 
% \]

%Given a set $S\subseteq \R^n$, $\homo(S) := \{ \lambda(1,S) \mid \lambda>0 \} = \{ (\lambda, \lambda x) \mid x \in S,\ \lambda>0 \}$. Given a cone $K \subseteq \R^{n} \times \R_{++}$, $\dehomo(K) = \{x \in \R^n \mid  (x,1) \in K \}$. For example, for a vector $v$ in $\R^n$, $\homo(v)$ is the ray passing through $(1,v)$, that is, $\{( \lambda, \lambda v) \mid \lambda > 0\}$. In contrast, for a ray $(r,\lambda)$ in $\R^{n} \times \R_{++}$, $\dehomo(r)$ yields a point $(r_1/\lambda,\ldots, r_n/\lambda )$ {\color{blue} ugly, redundant notation}
\section{Simultaneous convexification of fractional terms}\label{section:mainresult}
Fractional programming has been studied extensively and has a diverse set of applications. A seminal result of \cite{charnes1962programming} shows how to optimize a linear fractional function over a polytope as long as the denominator is positive over this region. Formally, the result solves a class of linear fractional programs in the following form:
\begin{equation}\label{eq:basicfractionallp}
    {\min} \biggl\{\frac{b_0+b^\intercal x}{a_0+a^\intercal x} \biggm| Cx\le d,\ x\ge 0\biggr\},
\end{equation}
where $C\in \R^{r \times n}$, $d$ is an $r$ dimensional vector, $(a_0,a)$ is a vector in $\R \times \R^{n}$ such that $a_0 + a^\intercal x$ is positive over the feasible region, and $(b_0,b)$ is a vector in $\R \times\R^{n}$. The key step reduces \eqref{eq:basicfractionallp} to the following linear program which can be solved in polynomial time
\begin{equation*}
    {\min} \Bigl\{b_0\rho+b^\intercal y \Bigm| Cy\le d\rho,\ a_0\rho+a^\intercal y = 1,\ y \geq 0,\ \rho \ge 0\Bigr\}.
\end{equation*}
%Observe that we have chosen the constant term to be 1 in both the numerator and the denominator. This is without loss of generality. In fact, consider  $\frac{b_0 + b^\intercal x}{a_0+a^\intercal x}$ and assume that the denominator is sign-invariant over the polytope. By multiplying the numerator with $-1$ we may assume that the denominator is positive. Then, we can rewrite $\frac{b_0 + b^\intercal x}{a_0+a^\intercal x}$ as $\frac{b_0}{a_0} \frac{1+\frac{1}{b_0}b^\intercal x}{1+\frac{1}{a_0}a^\intercal x}$. If $\frac{b_0}{a_0}$ is negative, we can replace minimization with maximization and vice-versa and drop the constant term. 
Another fundamental result in \cite{megiddo1978combinatorial} states that if for any $c\in \Z^n$ the problem 
\begin{equation}\label{eq:linearovercombinatorial}
    \min \bigl\{ c^\intercal x \bigm| x\in \X \subseteq \{0,1\}^n
    \bigr\}
\end{equation}
can be solved in $O\bigl(p(n)\bigr)$ comparisons and $O\bigl(q(n)\bigr)$ additions then, for $(a_0, a) \in \Z \times \Z^{n}$ and  $(b_0,b) \in \Z \times \Z^{n}$ such that $a_0+a^\intercal x$ is positive over $\X$, the problem 
\begin{equation}\label{eq:basicfractionalcomb}
    \min\biggl\{\frac{b_0+b^\intercal x}{a_0+a^\intercal x}\biggm| x\in \X\subseteq \{0,1\}^n\biggr\}
\end{equation}
can be solved in $O\left(p(n)\left(p(n)+q(n)\right)\right)$ time. These results enable polynomial-time solution procedures for a variety of fractional programming problems. Regardless, if either \eqref{eq:basicfractionallp} or \eqref{eq:basicfractionalcomb} is provided to a global optimization solver, the resulting relaxation does not exploit these tractability results, raising a question is whether the above results provide any insights for relaxation construction.

% These results however do not allow the separation of the convex hull of a fractional function over $\X$ in polynomial calls to an oracle that can optimize a linear function over $\X$.
We remark that the polynomial-time algorithms do not imply that the convex hull of a fractional function over $\X$ can be separated in polynomial time. To see this, for any $c\in\R^n$,  $\min\bigl\{ c^\intercal x \bigm|  x\in \X =  \{0,1\}^n \bigr\}$ is easy to solve, as it attains optimum at $x_i = 1$ if $c_i < 0$ and $x_i=0$ otherwise. Yet, it is hard to separate the convex hull of the graph a fractional function over $\{0,1\}^n$, given as follows:
\[
\biggl\{\Bigl(x,\frac{b_0+b^\intercal x}{a_0+a^\intercal x}\Bigr) \biggm| x \in \{0,1\}^n\biggr\}.
\]
This is because, if the separation problem could be solved in polynomial time, we could use the ellipsoid algorithm to solve the following optimization problem in polynomial time, see  Corollary 2.9 in~\cite{tawarmalani2013explicit},
\[
\min \biggl\{\frac{b_0+b^\intercal x}{a_0+a^\intercal x} + c^\intercal x \biggm| x \in  \{0,1\}^n \biggr\}.
\]
However, the latter problem is \revision{NP-hard} as shown in~\cite{hansen1990boolean} (also, see the later discussion in Example~\ref{ex:frac-linear}).  Therefore, unless $\text{P}=\text{NP}$, the separation problem for the graph of a linear fractional function over $\{0,1\}^n$ is not solvable in polynomial time. Moreover, as \cite{matsui1996np} showed, it is also \revision{NP-hard} to solve the following optimization problem:
\begin{equation*}
\min_x \Bigl\{x_1 - \frac{1}{x_2} \Bigm| Cx\le d,\ x\ge 0\Bigr\},   
\end{equation*} 
which implies that it is \revision{NP-hard} to find a hyperplane separating an arbitrary point from the convex hull of a set involving a single fractional term, that is 
\begin{equation*}\label{eq:convhullhard}
\left\{\left(x,\frac{1}{x_2}\right) \;\middle|\; Cx\le d,\ x\ge 0\right\}.  
\end{equation*} 
It follows that Charnes-Cooper reformulation and Megiddo's algorithm do not convexify the fractional function, but only allow us to optimize this function, over the corresponding domain. We will derive a convexification result, in a transformed space, that recovers polynomial-time solvability of \eqref{eq:basicfractionallp} and \eqref{eq:basicfractionalcomb} under conditions described above. More importantly, this result yields insights into relaxations of fractional functions for general mixed-integer nonlinear programming problems. 

Our relaxation scheme introduces following variables to represent the fractional terms
% denoting the fractional variables used in the Charnes-Cooper transformation that are also introduced in relaxation schemes for the fractional programs. In particular, we define:
\begin{equation*}
 \rho = \frac{1}{a_0+a^\intercal x} \quad \text{and} \quad y_i = \frac{x_i}{a_0+a^\intercal x} \quad \text{for }  i  \in [n].
\end{equation*}
Using these new variables, the problems \eqref{eq:basicfractionallp} and \eqref{eq:basicfractionalcomb} are rewritten as:
\begin{equation*}
 \min \Bigl\{ b_0\rho + b^\intercal y  \Bigm| x \in \X,\  (\rho,y)= \frac{(1,x)}{a_0+a^\intercal x}\Bigr\},
\end{equation*}
where, \eqref{eq:basicfractionallp} is modeled with $\X=\{x\mid Cx\le d, x\ge 0\}$, whereas, for \eqref{eq:basicfractionalcomb}, $\X$ is a subset of $\{0,1\}^n$ such that \eqref{eq:linearovercombinatorial} is easily solvable. Even though 
\[
\left\{\left(x,\frac{b_0+b^\intercal x}{a_0+a^\intercal x}\right)\;\middle|\; x\in \X\right\}
\]
cannot be easily separated over these domains, we will argue that convex hull of the following set can be separated in polynomial time:
\begin{equation}\label{eq:basicfractionalconv}
\Bigl\{ (\rho,y) \Bigm| x\in \X,\  (\rho, y) = \frac{(1,x)}{a_0+a^\intercal x}\Bigr\}.
\end{equation}
We highlight the main differences between these sets.
Specifically, the $x$ variables are projected out in \eqref{eq:basicfractionalconv} and one fractional function is replaced with many fractional terms.

\def \Rratio {\mathcal{R}}

The convex hull of \eqref{eq:basicfractionalconv} will follow from a more general fact that relates the convex hull of a set to that of its projective transform. Consider a transformation given as follows
\begin{equation}\label{eq:introducePhi}
\Phi(\s,x) = \left(\frac{1}{\s},\frac{x}{\s} \right) \quad \text{ for every } (\s, x) \in \R_{++} \times \R^n.
\end{equation}
To fix ideas, the transformation $\Phi$ maps the set~\eqref{eq:basicfractionalconv} to  $\bigl\{(a_0+a^\intercal x, x ) \mid x \in \X\bigr\}$. Theorem~\ref{thm:mainthem}, which we show below, allows us to use the convex hull of the transformed set to derive the convex hull of~\eqref{eq:basicfractionalconv}. 
For a subset $\S$ of $\R_{++} \times \R^n$, the image of $\S$ under $\Phi$ can be visualized as follows. First, we homogenize the set $\S$ with a scaling variable $\sigma$, and obtain a cone 
% \[
% \bigl\{(\sigma, \sigma \rho, \sigma x)  \bigm| (\rho,x) \in \S,\ \rho >0,\ \sigma > 0 \bigr\}. 
% \]
% Then, we obtain the image of $\S$ as the projection of the following conic section onto the space of the first and last arguments
% \[
% \bigl\{ (\sigma, \sigma \rho, \sigma x) \bigm| ( \rho,   x) \in  \S,\  \sigma \rho = 1,\ \rho > 0,\ \sigma > 0 \bigr\}. 
% \]
\[
\bigl\{(\sigma, \sigma\rho, \sigma x)  \bigm|  (\rho, x) \in \S,\ \rho >0,\ \sigma > 0 \bigr\}. 
\]
$\S$ can be viewed as the intersection of this cone with \{$\sigma=1$\} (after projecting out the first coordinate). Then, we consider a different cross-section of this cone,
\[
\bigl\{(\sigma, \sigma \rho, \sigma x) \bigm| (\rho,x) \in \S ,\ \rho >0,\ \sigma >0,\ \sigma\rho =1 \bigr\}. 
\]
Now, the image of $\S$ under $\Phi$ is obtained by projecting the new cross-section onto the space of the first and last coordinates.
% After projecting out the $\rho$ variable, we obtain the following set:
% \[
% %\begin{aligned}
% \biggl\{(\sigma, x)  \biggm| \Bigl(\frac{1}{\sigma},\frac{x}{\sigma}\Bigr) \in \S,\ \sigma > 0 \biggr\}
% =\bigl\{(\sigma, x)  \bigm| \Phi(\sigma, x) \in \S,\ \sigma > 0 \bigr\},
% %\end{aligned}
% \]
% which is precisely $\Phi(\S)$ by noting that $\Phi(\Phi(\sigma,x))=(\sigma,x)$ for any $(\sigma,x)\in\R_{++}\times \R^n$.
In Figure~\ref{fig:projectivetransform}, we exemplify this transformation by showing that quadratic $x^2$ and the reciprocal $\frac{1}{x}$ as different sections of the same cone.

\begin{figure}[ht]
\begin{center}
\subcaptionbox{Slice: $\left\{(y,z)\;\middle|\; z\ge y^2, y\ge 0\right\}$\label{fig:quad}}{
    \includegraphics{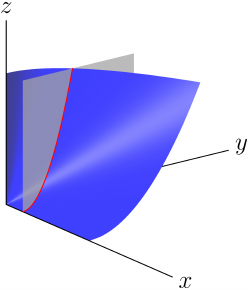}}
\subcaptionbox{Slice: $\left\{(x,z)\;\middle|\;z\ge \dfrac{1}{x}\right\}$\label{fig:inverse}}{
    \includegraphics{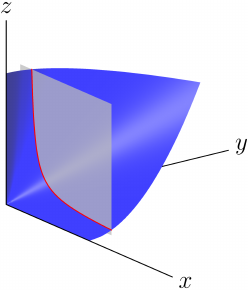}}
\end{center}
\caption{\revision{Two views of the cone $\left\{(x,y,z)\in \R^3_+\;\middle|\; y\le \sqrt{xz}\right\}$. When $x$-coordinate (resp. $y$-coordinate) is fixed, we obtain the slice shown in Figure\eqref{fig:quad} (resp. Figure\eqref{fig:inverse}). For $\lambda,\gamma,\theta\ge 0$, consider rays $\left(25\lambda,5\lambda,\lambda\right)$ and $\left(25\gamma,20\gamma,16\gamma\right)$ in the cone and the ray $\left(25\theta,8\theta,4\theta\right)$ that is in the convex cone generated by these rays. This observation leads to two convex representations of interest. Specifically, when $y$-coordinate (resp. $x$-coordinate) is scaled to $1$, we have $\left(\frac{25}{8},1,\frac{1}{2}\right) = \frac{1}{2} \left(5,1,\frac{1}{5}\right) + \frac{1}{2}\left(\frac{5}{4},1,\frac{4}{5}\right)$ (resp.
$\left(1,\frac{8}{25},\frac{4}{25}\right) = \frac{4}{5} \left(1,\frac{1}{5},\frac{1}{25}\right) + \frac{1}{5}\left(1,\frac{4}{5},\frac{16}{25}\right)$). The former is a combination of points on the curve $z=\frac{1}{x}$, while the latter is a convex combination of points on $z = y^2$.} \label{fig:projectivetransform}}
\end{figure}

\begin{theorem}\label{thm:mainthem}
    For a subset $\S$ of $\R_{++} \times \R^n $, we have
\[
\begin{aligned}
\conv(\S)  = \bigl\{(\rho, x)& \bigm| (1,x) \in \s  \conv\bigl(\Phi(\S)\bigr),\ \s >0 \bigr\} \\
\conv\bigl(\Phi(\S)\bigr) = \bigl\{(\t,y)& \bigm| (1,y) \in \t  \conv(\S),\ \t >0\bigr\}.
\end{aligned}
\]
\end{theorem}
\begin{proof}
	First, notice that 
\[
\Phi\bigl(\Phi(\S)\bigr) =\Phi\left(\left\{\Bigl(\frac{1}{\rho}, \frac{x}{\rho}\Bigr) \Bigm| (\rho, x) \in \S \right\} \right) =  \S. 
\]
Thus, it suffices to prove the first equation.  Let $R$ denote the right-hand-side of the first equation. First, we show that $R \subseteq \conv(\S)$. Let $(\s, x) \in R$. Then, $( \frac{1}{\s},\frac{x}{\s}) \in \conv\bigl(\Phi(\S)\bigr)$, and thus there exist convex multipliers $\lambda$ and a set of points $(\s^j,x^j)$ of $\S$ such that 
\[
\Bigl(\frac{1}{\s},\frac{x}{\s} \Bigr) =  \sum_{j} \lambda_j \Bigl(\frac{1}{\s^j}, \frac{x^j}{\s^j}\Bigr).
\]
Thus,
\[
(\s,x,1) =   \s \Bigl(1,\frac{x}{\s} , \frac{1}{\s} \Bigr) =\s \sum_{j} \lambda_j\Bigl( 1,\frac{x^j}{\s^j}, \frac{1}{\s^j} \Bigr) =   \sum_{j} \frac{\s\lambda_j}{\s^j} (  \s^j, x^j, 1). 
\]
This shows that $(\s, x) \in \conv(\revision{\mathcal{S}})$ since $(\s^j,x^j) \in \S$, $\sum_j\frac{\s\lambda_j}{\s^j} = 1$ and $\frac{\s\lambda_j}{\s^j}\geq 0$. Therefore, $R \subseteq \conv(\S)$. To show the reverse containment $\conv(\S) \subseteq R$, we consider a point $(\s,x) \in \S$. It follows readily that 
\[
\Bigl(\frac{1}{\s},\frac{x}{\s}\Bigr) =  \Phi(\s,x) \in \Phi(\S).
\]
 This shows that $(1,x) \in \s \Phi(\S)$. Therefore, $(\s,x) \in R$, and thus $\S \subseteq R$. \revision{In addition, the relaxation $R$ is a convex set since it is the projection onto the space of $(\rho,x)$ variables of the following intersection
 \[
 \Bigl\{(\rho,\sigma',x) \Bigm| (\sigma',x) \in \rho \conv\bigl(\Phi(\mathcal{S})\bigr),\ \rho > 0  \Bigr\} \cap  \bigl\{(\rho,\sigma',x) \bigm| \sigma'=1 \bigr\},
 \]
where the former set is the smallest convex cone that includes $\conv\bigl(\phi(\mathcal{S})\bigr)$, see Corollary 2.6.3 in~\cite{rockafellar2015convex}.} Hence, $\conv(\S) \subseteq R$ as $R$ is convex.  \Halmos
\end{proof}
\begin{remark}\label{rmk:eq-sep}
   In this remark, we show a polynomial-time equivalence of separating $\conv(\S)$ and $\conv\bigl(\Phi(\S)\bigr)$ from arbitrary points. This, together with the polynomial-time equivalence of separation and optimization \cite{grotschel2012geometric}, implies a polynomial time equivalence of optimization over the two convex hulls. Let $(\bar{\rho}, \bar{x}) \in \R_{++} \times \R^{n}$. We will use the separation oracle of $\conv\bigl(\Phi(\S)\bigr)$ to separate $(\bar{\rho}, \bar{x})$ from $\conv(\S)$. If $(\bar{\sigma}, \bar{y}):=\Bigl(\frac{1}{\bar{\rho}},\frac{\bar{x}}{\bar{\rho}}\Bigr) \in \conv\bigl(\Phi(\S)\bigr)$, then, by Theorem~\ref{thm:mainthem}, we conclude $(\bar{\rho}, \bar{x}) \in \conv(\S)$. Otherwise, the separation oracle returns a linear inequality $\bigl\langle \alpha, (\sigma, y) \bigr\rangle + \beta \leq 0$ to separate $(\bar{\sigma}, \bar{y})$ from $\conv(\Phi(\S))$. Then, by Theorem~\ref{thm:mainthem}, the inequality $\bigl\langle \alpha, (1, x) \bigr\rangle + \beta\cdot \rho\leq 0$  separates $(\bar{\rho}, \bar{x})$  from $\conv(\S)$. Conversely, using Theorem~\ref{thm:mainthem} and a separation oracle of $\conv(\S)$, we can devise a separation algorithm for $\conv\bigl(\Phi(\S)\bigr)$. 
    \Halmos
\end{remark}

% \begin{corollary}
% If $\rec(\conv(\Phi(S))) \cap \{(\sigma, y) \mid \sigma = 1\} = \emptyset$ then 
% \end{corollary}

% \begin{remark}
% If $\rec(\conv(\Phi(\S))) \cap \{(\sigma, y) \mid \sigma = 1\} = \emptyset$ then
% \[
%    \conv(\S) =  \bigl\{(\rho, x) \mid (1,x) \in \rho\conv(\Phi(\S)),\ \rho \geq 0 \bigr\}. 
%     \]
% \end{remark}

% \begin{corollary}
% Equivalence in extension complexity, separation, and optimization over the convex hull/relaxation. 
% \end{corollary}

% To illustrate the relevance of $\Phi$ in the convexification of fractional functions, we consider  

% combinatorial optimization with rational objective functions~\cite{megiddo1978combinatorial} in the next example.  

%\subsection{Illustrative applications}

Next, we present an application of Theorem~\ref{thm:mainthem} that, under certain conditions, relates the convex hull of fractional functions sharing a denominator to that of a set involving just the numerators. We show that convexifying these two sets, denoted $\G$ and $\F$ respectively, is polynomially equivalent, so that a formulation or separation procedure for one yields that for the other. 
%Given Remark~\ref{rmk:eq-sep}, this shows that optimizing linear functions over these sets is also polynomially equivalent. 
Consider a vector of base functions $f:\X \subseteq \R^n \to \R^m$ denoted as $f(x)=\bigl(f_1(x),\ldots,f_m(x)\bigr)$ and another vector of functions obtained from the base functions by dividing each of them with a linear form of $f$, {\it i.e.}, for $\alpha \in \R^{m}$, we consider
\begin{equation}\label{eq:defineFG}
    \F=\biggl\{f(x)\biggm| x\in \X \biggr\} \quad \text{and} \quad \G = \biggl\{\frac{f(x)}{\sum_{i \in [m]} \alpha_i f_i(x)}\biggm| x \in \X\biggr\}.  
\end{equation}
\revision{In Theorem~\ref{thm:fractionalmain}, we will show that, under certain conditions, the convex hull of one set can be described using that of the other as follows:}
\begin{subequations}\label{eq:eqFG}
    \begin{alignat}{3}
        \conv(\G) &= \bigl\{g \in \R^m \bigm| \revision{\exists \rho \geq 0 \st{} } g \in \rho\conv(\F),\  \alpha^\intercal g  = 1\bigr\}\label{eq:FprimeConv} \\
        \conv(\F) &= \bigl\{f \in \R^m \bigm| \revision{\exists \sigma \geq 0 \st{} } f\in \sigma\conv(\G),\ f_1=1 \bigr\}.\label{eq:FConv}
    \end{alignat}
\end{subequations}    
    % \begin{equation}\label{eq:FprimeConv}
    % \begin{aligned}
    %     \conv(\G) &= \bigl\{g \in \R^m \bigm| \revision{\exists \rho \geq 0 \st{} } g \in \rho\conv(\F),\  \alpha^\intercal g  = 1\bigr\}.
    %     % \conv(F) &= \bigl\{f\bigm| f\in \rho\conv(G),\  f_1=1 \bigr\}.
    % \end{aligned}
    % \end{equation}
    % \begin{equation}\label{eq:FConv}
    % \conv(\F) = \bigl\{f \in \R^m \bigm| \revision{\exists \sigma \geq 0 \st{} } f\in \sigma\conv(\G),\ f_1=1,\ \sigma \geq 0 \bigr\}.
    % \end{equation}
\revision{One of the conditions  for \eqref{eq:FprimeConv} to hold is the boundedness of $\mathcal{F}$.  The next example shows that such condition is required.} 
\begin{example}\label{example:boundedness}
    Consider a discrete set defined as
     \[
     \G = \biggl\{\frac{(1,x)}{1+x} \biggm|  x \geq 0,\ x \in \Z  \biggr\}.
     \]
    % \[
    % \G = \biggl\{\frac{(1,x)}{1+x} \biggm|  x \geq 0,\ x \in \Z  \biggr\}.
    % \]
Let $\F= \bigl\{(1,x) \bigm| x \geq 0,\ x \in \Z \bigr\}$ and notice that the convex hull of $\F$ is its continuous relaxation. For this setting, \revision{Theorem~\ref{thm:fractionalmain} is not applicable since $\F$ is not bounded. In fact, the equality in~\eqref{eq:FprimeConv} does not hold} since the right hand side of \eqref{eq:FprimeConv} reduces to \revision{$\bigl\{ g \bigm| \exists \rho \geq 0 \st 1 = g_1 + g_2,\ g_1 = \rho, g_2 \geq 0 \bigr\}$}, and consists of $(0,1)$, which does not belong to $\conv(\G)$ because $\frac{1}{1+x} > 0$ for all $x\in \G$. \Halmos
\end{example}
Example~\ref{example:boundedness} shows that the right hand side of \eqref{eq:FprimeConv} can include points not in $\conv(\G)$.  Assume that $\F$ is a polyhedron. These extra points arise because, when $\F$ is homogenized, with the homogenizing variable set to $0$, we obtain the recession cone of the polyhedron instead of a set containing just the origin. However, we can show that these additional points are typically included in $\cl\conv(\G)$.

% Later, Theorem~\ref{thm:fractionalmain} will be used to derive several results. In each case, we will denote the corresponding sets as $\G$ and $\F$ so that the relation to the setting of Theorem~\ref{thm:fractionalmain} is clear. 
% \begin{theorem}\label{thm:unbounded}
%     Let $\F$ and $\G$ be nonempty sets defined as in \eqref{eq:defineFG} and $0\cl\conv(\F)$ denote the recession cone of $\cl\conv(\F)$. Assume that  $\F\ne\emptyset$ and   \revision{$\sum_{i \in [m]}\alpha_if_i(x) > 0$ for every $x \in \mathcal{X}$}.
%     %, and $\sum_{i\in [m]}\alpha_if_i(x) > \epsilon>0$ over $\X$. 
%     Then, 
%     \begin{equation}\label{eq:FprimeClConv}
%     \begin{aligned}
%         \cl\conv(\G) &= \bigl\{g \in \R^m \bigm| g \in \sigma\cl\conv(\F),\  \alpha^\intercal g  = 1,\ \sigma \ge 0\bigr\}.
%         % \conv(F) &= \bigl\{f\bigm| f\in \rho\conv(G),\  f_1=1 \bigr\}.
%     \end{aligned}
%     \end{equation}
% \end{theorem}

\begin{theorem}\label{thm:fractionalmain}
  \revision{Let $\F$ and $\G$ be nonempty sets defined as in \eqref{eq:defineFG}, and assume that there exists an $\epsilon > 0$ such that  $\sum_{i \in [m]}\alpha_if_i(x) > \epsilon$ for all $x\in \X \subseteq \R^n$. If $\F$ is bounded,~\eqref{eq:FprimeConv} holds and if, in addition, $f_1(x) = 1$ then~\eqref{eq:FConv} holds. Moreover, 
        \begin{equation}\label{eq:FprimeClConv}
    \begin{aligned}
        \cl\conv(\G) &= \bigl\{g \in \R^m \bigm| \exists \rho \geq 0 \st g \in \rho\cl\conv(\F),\  \alpha^\intercal g  = 1 \bigr\},
        % \conv(F) &= \bigl\{f\bigm| f\in \rho\conv(G),\  f_1=1 \bigr\}.
    \end{aligned}
    \end{equation}
where  $0\cl\conv(\F)$ denote the recession cone of $\cl\conv(\F)$.}
    % and there exists an $\epsilon > 0$ such that  $\sum_{i \in [m]}\alpha_if_i(x) > \epsilon$ for all $x\in \X \subseteq \R^n$.  Then, 
    % \begin{equation}\label{eq:FprimeConv}
    % \begin{aligned}
    %     \conv(\G) &= \bigl\{g \in \R^m \bigm| \revision{\exists \rho \geq 0 \st{} } g \in \rho\conv(\F),\  \alpha^\intercal g  = 1\bigr\}.
    %     % \conv(F) &= \bigl\{f\bigm| f\in \rho\conv(G),\  f_1=1 \bigr\}.
    % \end{aligned}
    % \end{equation}
     %If $f_1(x) = 1$, then
    % \begin{equation}\label{eq:FConv}
    % \conv(\F) = \bigl\{f \in \R^m \bigm| \revision{\exists \sigma \geq 0 \st{} } f\in \sigma\conv(\G),\ f_1=1,\ \sigma \geq 0 \bigr\}.
    % \end{equation}
\end{theorem}
\begin{proof}
\revision{Here, we focus on proving the correctness of~\eqref{eq:eqFG}.   The correctness of~\eqref{eq:FprimeClConv} is shown in Appendix~\ref{app:closure}. Consider a lift of $\F$ and $\G$, defined as follows, respectively,
 \begin{equation}\label{eq:lift}
 \S := \Bigl\{( \alpha^\intercal f ,f)\Bigm| f\in \F \Bigr\} \quad \text{and} \quad \mathcal{T}:= \Phi(\S) = \Bigl\{(\rho,g):= \frac{(1,f)}{ \alpha^\intercal f} \Bigm| f\in \F  \Bigr\}.
 \end{equation}
We will invoke Theorem~\ref{thm:mainthem} to obtain a convex hull expression of $\S$ (resp. $\mathcal{T}$), which, after projection, yields the equality in~\eqref{eq:FprimeConv} (resp.~\eqref{eq:FConv}). }

\revision{
First, we establish the equality in~\eqref{eq:FprimeConv}. Since $\G$ is the projection of $\mathcal{T}$ onto the space of $g$ variable and convexification commutes with projection, it suffices to derive the convex hull of $\mathcal{T}$, which is obtained as follows:
\begin{equation*}
\begin{aligned}
\conv\bigl(\mathcal{T}\bigr)& = \bigl\{(\rho, g)\bigm| (1,g)\in \rho \conv(\S),\ \rho > 0 \bigr\}  \\
 &= \bigl\{(\rho,g) \bigm| g \in \rho \conv(\mathcal{F}),\ \alpha^\intercal g = 1,\ \rho >0 \bigr\} \\
 & = \bigl\{(\rho,g) \bigm| g \in \rho \conv(\mathcal{F}),\ \alpha^\intercal g = 1,\ \rho \geq 0 \bigr\},
\end{aligned}
\end{equation*}
where the first equality holds due to Theorem~\ref{thm:mainthem} and $\mathcal{S} = \Phi(\mathcal{T})$. The second equality is derived using $\conv(\S) = \bigl\{(\alpha^\intercal f, f)\bigm| f\in \conv(\F)\bigr\}$, where the equality hods since $\S$ is obtained as a linear transformation of $\F$. The last equality holds since if the rhs contains  a point $(\rho,g)$ with $\rho = 0$,  i.e. $g\in 0\conv(\F)$ then the boundedness of $\F$ implies $g=0$, contradicting to the equality constraint $\alpha^\intercal g=1$.}

\revision{
Then, we establish the equality in~\eqref{eq:FConv}. Similarly, since $\F$ is the projection of $\mathcal{S}$, it suffices to derive the convex hull of $\mathcal{S}$,
\begin{equation*}
\begin{aligned}
	\conv\bigl(\S\bigr)& = \bigl\{(\sigma, f)\bigm| (1,f)\in \sigma \conv( \mathcal{T} ),\ \sigma > 0 \bigr\} \\
 & =  \bigl\{(\sigma, f)\bigm|  f\in \sigma\conv(\G),\ f_1=1,\ \sigma > 0\bigr\}\\
 & =  \bigl\{(\sigma, f)\bigm|  f\in \sigma\conv(\G),\ f_1=1,\ \sigma \geq 0\bigr\},
\end{aligned}    
\end{equation*}
where the first equality holds due to Theorem~\ref{thm:mainthem} and $\mathcal{T} = \Phi(\mathcal{S})$. The second equality holds since, under the assumption that $f_1(x)=1$, $\conv(\mathcal{T}) = \bigl\{(\rho,g) \bigm| \rho = g_1,\ g \in \conv(\G) \bigr\}$. To see the last equality, we observe that $ \conv(\mathcal{G})$ is bounded since $\sum_i\alpha_if_i(x) > \epsilon$ and the boundedness of $\F$ implies that $\mathcal{G}$ is bounded.} \Halmos
\end{proof}

%In the next example, we revisit the Megiddo's result. 
It is useful to interpret~\eqref{eq:FprimeConv} as arising from a two-step procedure where we first homgenize $\conv(\F)$ and then intersect the resulting cone with $\alpha^\intercal g = 1$.
To illustrate the relevance of Theorem~\ref{thm:fractionalmain} in convexifying fractional programs, we revisit the result of \cite{megiddo1978combinatorial}. 
%we consider a family of fractional programs studied by~\cite{megiddo1978combinatorial}. 
\begin{example}\label{ex:megiddo}
%Consider a fractional combinatorial optimization problem defined as 
%\begin{equation}\label{eq:comb_rational}
%\min \biggl\{ \frac{1 + b^\intercal x}{1 + a^\intercal x} \biggm| x \in \X  \biggr\},
%\end{equation}
%where $\X$ is a subset of $\{0,1\}^n$ and models a combinatorial structure, $b$ is a vector in $\R^n$, and $a$ is a vector in $\R^n$ such that  $1 +  a^\intercal  x > 0$ for all $x \in \X$.  
We describe a polynomial time algorithm to solve \eqref{eq:basicfractionalcomb} using a polynomial-time algorithm to optimize linear functions over $\X$. Using the equivalence of separation and optimization, 
\eqref{eq:basicfractionalcomb} can be solved in polynomial time if  $\conv(\G)$ can be separated in polynomial time, where 
\[
\G=\Bigl\{\frac{(1,x)}{a_0+a^\intercal x}\Bigm| x\in \X \Bigr\}.
\]
As we showed in Theorem~\ref{thm:fractionalmain}, the convex hull of $\G$ can be derived from that of $\X$. We use $(\rho, y)$ to denote $g$. The use of $\rho$ is reasonable because the first coordinate is $1$ in $\conv(\F)$ which, by $g\in \rho \conv(\F)$, implies that $g_1 = \rho$. Specifically, the separation problem of $\conv(\G)$ is solved as follows. If a point $(\bar{\rho},\bar{y})\not\in\conv(\G)$ then $\bar{\rho} > 0$ and, therefore, $\frac{\bar{y}}{\bar{\rho}}\not\in \conv(\X)$. Assume $\alpha^\intercal x \le b$ is a valid inequality for $\conv(\X)$ that $\frac{\bar{y}}{\bar{\rho}}$ does not satisfy. Then, $ \alpha^\intercal \bar{y} -b\bar{\rho} > 0$ while $ \alpha^\intercal y -b\rho  \le 0$ for each $(\rho,y)\in \conv(\G)$. \Halmos
% \Halmos

% Megiddo~\cite{megiddo1978combinatorial} provides a constructive proof to show that~\eqref{eq:comb_rational} is polynomial-time solvable if optimizing a linear function over $\X$ is polynomial-time solvable. 

% Here, we use Theorem~\ref{thm:mainthem} to obtain a similar result. First, we consider an extended formulation~\eqref{eq:comb_rational}, that is,  $\min \{\rho + b^\intercal y \mid (\rho, y) \in \conv(\S) \}$, where
% \[
% \S :=\biggl\{ (\rho,y) \in \R_{++} \times \R^n \biggm| y_i  = x_i \rho \text{ for } i \in [n],\ \rho = \frac{1}{1 + a^\intercal x},\   x \in \X \biggr\}. 
% \]
% %It appears that the convex hull of $\S$ is harder to obtain than that of $\X$ as $\S$ is the image of $\X$ under a linear fractional map. However, 
% Although $\S$ consists of several nonconvex structures, its image, under the transformation $\Phi$, has a clean structure. More specifically, $\Phi(\S) = \{(1 + a^\intercal x, x ) \mid x \in \X \}$. In other words, the image is the graph of the denominator over the combinatorial feasible region $\X$, which is the only source of nonconvexity. Thus, it follows readily that the convex hull of $\Phi(\S)$ is $\{( 1 +  a^\intercal x, x ) \mid x \in \conv(\X) \}$. Therefore, 
% \[
% \conv(\S) = \bigl\{(\rho,y) \bigm| \rho +  a^\intercal y  =1,\ y \in \rho  \conv(\X),\ \rho \geq 0 \bigr\}.
% \]
% Moreover, we obtain a compact LP/SDP formulation of~\eqref{eq:comb_rational} provided given a compact LP/SDP formulation of $\conv(\X)$. \Halmos
\end{example}
As shown in \cite{hansen1990boolean}, minimizing the sum of a linear fractional function and a linear function over $\{0,1\}^n$ is \revision{NP-hard}. We relate this hardness result to 
Theorem~\ref{thm:fractionalmain} in the next example. 
\begin{example}\label{ex:frac-linear}
Consider a linear fractional problem defined as 
\begin{equation}\label{eq:frac-linear}
\min  \biggl\{ \frac{b_0 + b^\intercal x}{a_0+a^\intercal x} + c^\intercal x \biggm| x \in \{0,1\}^n \biggr\},  
\end{equation}
where $(a_0,a) \in \Z \times \Z^n$, $(b_0,b) \in \Z \times \Z^n$, and $c \in \Z^n$. We show that \eqref{eq:frac-linear} is \revision{NP-hard under Turing reductions}.  By the equivalence of separation and optimization, \eqref{eq:frac-linear} is polynomial time solvable for all $(b_0,b)$ if and only if the separation of the convex hull of $\G$ is polynomial time solvable, where 
\[
\G=\left\{\left(\frac{(1,x)}{1+a^\intercal x}, x \right)\;\middle|\; x\in \{0,1\}^n \right\}.
\]
By Remark~\ref{rmk:eq-sep} and Theorem~\ref{thm:fractionalmain}, if $\conv(\G)$ has a polynomial time separation algorithm, so does $\conv(\F)$,  where
    \[
\F =  \left\{ \left(1,x, x\bigl(1+a^\intercal x\bigr)\right) \;\middle|\; x \in \{0,1\}^n \right\}.
    \]
By the ellipsoid algorithm, the following problem that minimizes a linear function with integer coefficients over $\F$ is polynomial time solvable:
 \begin{equation}\label{eq:bilinear-opt}
    \min \bigl\{ \alpha +   \beta^\intercal x + (\gamma^\intercal x) (1 + a^\intercal x) \bigm| x \in \{0,1\}^n \bigr\},
\end{equation}
where $\alpha \in \Z$, and $\beta$ and $\gamma$ are vectors in $\Z^n$. 
However, \eqref{eq:bilinear-opt} is \revision{NP-hard} since the subset-sum problem can be polynomially reduced to it as follows. Consider $w\in \Z^n_+$ and $K\in \Z_+$. The subset sum is to decide if there exists an $x \in \{0,1\}^n$ such that $w^\intercal x = K$. Let $a = w$, $\alpha = 0$, $\beta = (-2K-1)\cdot w$ and  $\gamma = w$. Now, observe that
\[
\begin{aligned}
\alpha + \beta^\intercal x+ \gamma^\intercal x(1+a^\intercal x) = (w^\intercal x)\cdot(w^\intercal x) - 2K \cdot w^\intercal x = (w^\intercal x-K)^2 - K^2,
\end{aligned}
\]
which is $-K^2$ if and only if there exists an $x\in \{0,1\}^n$ satisfying $w^\intercal x = K$. \Halmos
\end{example}

\section{Connections between zero-one fractional and quadratic programming}
\label{section:BQP}
\def \F {\mathcal{F}}
% In this section, using Theorem~\ref{thm:mainthem}, we transform  valid inequalities of the boolean quadric polytope (\textsf{BQP}) to generate cutting planes for fractional programs. 
Let $G=(V,E)$ be a graph, with $V$ and $E$ as the set of nodes and edges respectively, that has no self-loops.
% and define 
%\[
%\F_G = \Bigl\{ (x,z) \in \R^{|V| + |E|}\Bigm| z_{ij} = x_ix_j \text{ for } (i,j) \in E,\ x \in \{0,1\}^{\vert V\vert} \Bigr\}.
%\]
The boolean quadric polytope (\textsf{BQP}) associated with the graph $G$, denoted as $\QP_G$,  is defined as the convex hull of
\[
\Bigl\{ (x,z) \in \R^{|V| + |E|}\Bigm| z_{ij} = x_ix_j \text{ for } (i,j) \in E,\ x \in \{0,1\}^{\vert V\vert} \Bigr\}.
\]
In particular, we denote by $\QP_G$ as $\QP$ if $G$ is a complete graph.  This polytope was introduced in~\cite{padberg1989boolean}, and 
its polyhedral structure and related algorithmic problems have been studied extensively~\cite{boros1992chvatal,boros1993cut,deza1997geometry,bonami2018globally,junger2021exact}. \textsf{BQP} is related to the cut polytope~\cite{barahona1986cut} via a bijective linear transformation~\cite{de1990cut}. 
We show that valid inequalities and convex hull characterizations in special cases of BQP and/or cut-polytope lead to improved relaxations for 0-1 fractional programs. 
% Inequalities for BQP and/or cut-polytope have substantially advanced our ability to solve Max-Cut problems~\cite[see, for example,][]{rendl2010solving}, and we now demonstrate their applicability to fractional programming.

We present several \textsf{BQP} inequalities in a homogenized form using an additional variable $\rho$ to make them easier to use later~\cite{padberg1989boolean}. The McCormick inequalities~\cite{mccormick1976computability} for $i, j \in [n]$ with $i\neq j$ are as follows:
\begin{equation}\label{eq:mc}
   \begin{aligned}
   -z_{ij} \leq 0,\ -z_{ij} + x_i+x_j \leq \rho ,\     z_{ij} - x_i \leq 0,\ z_{ij} - x_j \leq 0,
   \end{aligned} \tag{\textsc{McCormick}}
\end{equation}
the triangle inequalities for  $i,j,k \in [n]$ are: 
\begin{equation}\label{eq:tri}
    \tag{\textsc{Triangle}}
    \begin{aligned}
     x_i+x_j+x_k-z_{ij}-z_{jk}-z_{ik} &\leq \rho \\
     -x_i + z_{ij} + z_{ik} - z_{jk} &\leq 0 \\
     -x_j + z_{ij} - z_{ik} + z_{jk} &\leq 0 \\
     -x_k - z_{ij} + z_{ik} + z_{jk} &\leq 0, \\
\end{aligned}
\end{equation}
% \begin{equation}\label{eq:clique}
%     \tag{\textsc{Clique}}
% \end{equation}
and the odd-cycle inequalities are described next.
Consider a cycle $C \subseteq E$ and a subset $D \subseteq C$ that has an odd number of elements. Let $S_0 = \{u \in V \mid \exists e,f \in D, e \neq f, \text{ s.t. } e\cap f = u \}$ and $S_1 = \{u \in V \mid \exists e, f \in C\setminus D, e\ne f, \text{ s.t. } e\cap f = u \}$, {\it i.e.}, $S_0$ (resp. $S_1$) is the set of nodes with both (resp. none) of the adjacent edges in $D$. Then, the odd-cycle inequality is 
\begin{equation}\label{eq:oddcycle}
  \sum_{i\in S_0}x_i - \sum_{i\in S_1}x_i + \sum_{(i,j) \in C\setminus D }z_{ij} - \sum_{(i,j) \in D} z_{ij} \leq \frac{\vert D \vert -1}{2}\rho. \tag{\textsc{Odd-Cycle}}
\end{equation}
Odd-cycle inequalities can be separated in polynomial-time \cite{de1990cut,barahona1986cut}.  
% In theory, the odd-cycle inequalities giving the Chv{\'a}tal-Gomory closure of the McCormick relaxation of $\QP_G$ for arbitrary $G$~\cite{bonami2018globally}. In practice, efficient separation of odd-cycle cuts is a key component of binary quadratic optimization solvers~\cite{RRW10,charfreitag2022mcsparse,rehfeldt2022faster}

Besides the explicit inequalities in the space of \textsf{BQP} variables, we will also exploit the reformulation-linearization technique (RLT) to relax \textsf{BQP}~\cite{sherali1990hierarchy} by introducing new variables that represent monomials of higher degrees. Specifically, given a positive integer  $2 \leq k \leq \vert V \vert$, for each subset $e$ of $V$ of cardinality 
from $2$ to $k$, we represent the monomial $\prod_{i \in e} x_i$ with a variable $z_e$. Then, the $k^{\text{th}}$ RLT relaxation of \textsf{BQP}, denoted as $\RLT_k$, is the system of linear inequalities obtained from the following system of polynomial inequalities by using relation $x_i = x_i^2$  and replacing each monomial with its linearization\footnote{Observe that unlike the standard notation, our usage of $\text{RLT}_k$ linearizes degree $k+1$ monomials.}:
\begin{equation*}\label{eq:RLT}
\prod_{i \in S}  x_i \prod_{j \in T} (1-x_j) \geq 0 \quad \text{ for all } S \cap T = \emptyset \text{ and } \vert S \cup  T\vert = k+1. 
\end{equation*}
\begin{lemma}[\cite{sherali1990hierarchy}]\label{lemma:rlt}
    For a given $d \in \Z_+$ with $d \geq 2$, the convex hull of $\bigl\{(x,z)\mid x \in \{0,1\}^d,\ z_e=\prod_{i\in e} x_i\; \forall e \subseteq [d] \text{ s.t. } \vert e \vert \geq 2\bigr\}$ is given by $\RLT_{d-1}$. 
\end{lemma}

\subsection{Zero-one linear fractional programs and \textsf{BQP}}\label{section:linear_frac_poly}
In this subsection, we focus on the problem of minimizing the sum of a linear fractional function and a linear function with binary variables
\begin{equation}\label{eq:frac+linear}
    \min \biggl\{\frac{ b_0 + b^\intercal x }{a_0 + a^\intercal x } + c^\intercal x \biggm| x \in \{0,1\}^n \biggr\},
\end{equation}
where $(a_0,a)$ and $(b_0,b)$ are vectors in $\R \times \R^n$, $c$ is a vector in $\R^n$, and the $a_0 + a^\intercal x$ is assumed to be positive on $\{0,1\}^n$.  \revision{For a subset $I$ of $[n]$, consider the linear fractional zero-one set defined as:
\begin{equation*}
\begin{aligned}
\G_{I} := \biggl\{ \Bigl(\frac{(1,x)}{a_0 + a^\intercal x},(x_i)_{i\in I} \Bigr) \biggm|   x \in \{0,1\}^n \biggr\},     
\end{aligned}
\end{equation*}
where $a_0+  a^\intercal x$ is assumed to be positive on $\{0,1\}^n$. We will use valid inequalities of \textsf{BQP} to study the convex hull of $\G_{[n]}$, and will refer to this convex hull \revision{as the \textit{linear fractional polytope} (\textsf{LFP}). 
Clearly, the fractional program~\eqref{eq:frac+linear} is equivalent to maximizing a linear function over \textsf{LFP}. 
The linear fractional set $\G_I$ has also appears in reformulations of linear fractional zero-one programs~\cite{tawarmalani2002global,borrero2017fractional,mehmanchi2019fractional}, where the objective} is to optimize the sum of ratios of affine functions subject to a set of linear constraints. This class of problems has many applications including  assortment planning, facility location, and mean dispersion, see~\cite{borrero2017fractional}. }

\revision{A tractable (polynomial-size) description of \textsf{LFP} is not possible, unless P=NP, since such a description would yield a polynomial time algorithm to solve \eqref{eq:frac+linear}, which is shown to be NP-hard in Example~\ref{ex:frac-linear}. 
Instead, we propose a hierarchy of polyhedral relaxations of \textsf{LFP}. For $k \in [n]$, the \textit{$k$-term relaxation} of \textsf{LFP} is defined as follows:
\begin{equation*}
 \bigcap_{I \subseteq [n] \colon\vert I \vert =k} \conv(\G_I), 
\end{equation*}
where $\conv(\G_I)$ is extended to the space of $\G_{[n]}$ by appending the remaining variables so that the intersection is well-defined. In particular, we will use valid inequalities of \textsf{BQP} to describe the $k$-term relaxation of $\textsf{LFP}$. To make this connection apparent, we begin with a remark on a slightly general setting.}

%Recall that Example~\ref{ex:frac-linear} reduces a class of bilinear programs to~\eqref{eq:frac+linear}.
%Here, we leverage valid inequalities for \textsf{BQP} to derive tight polyhedral relaxations for~\eqref{eq:frac+linear}.
% These polyhedral relaxations will be used in Section~\ref{section:general-linear-frac} to construct a relaxation hierarchy for general zero-one linear fractional programs, where the objective is to optimization the sum of ratios of affine functions subject to a set of linear constraints. 
%This connection follows from Theorem~\ref{thm:fractionalmain} as follows. Let  $f: \{0,1\}^n \to \R^m$ be a bounded function, and consider 
\revision{\begin{remark}\label{rmk:frac-func}
Let $f(\cdot)$ be a vector of functions mapping from $\{0,1\}^n$ to $\R^m$, and consider two sets
    \begin{equation*}
\begin{aligned}
\mathcal{G}&:= \left\{ \left(\frac{1}{a_0 + a^\intercal x},\frac{x}{a_0 + a^\intercal x},\frac{f(x)(a_0+a^\intercal x)}{a_0+a^\intercal x} \right)\in\R\times\R^{n}\times\R^{m}  \;\middle|\;  x \in \{0,1\}^n \right\} \\
\mathcal{F}&:=\bigl\{(1,x,f)\in \R \times \R^{n}\times\R^{m}  \bigm|  f=f(x)(a_0+a^\intercal x), \ x\in\{0,1\}^n\bigr\}.
\end{aligned}
\end{equation*}
Assume that $a_0+a^\intercal x$ is positive on $\{0,1\}^n$. By setting $\alpha = (a_0, a^\intercal, 0)\in \R\times\R^n\times\R^m$ and invoking Theorem~\ref{thm:fractionalmain}, the convex hull of $\mathcal{G}$ is given as follows:
\begin{equation*}
\conv(\mathcal{G})=\bigl\{(\rho,y,g)  \bigm| (\rho, y,g) \in \rho \conv(\mathcal{F}),\ a_0\rho + a^\intercal y =1,\ \rho \geq 0\bigr\}. \hfill \Halmos
\end{equation*}
\end{remark}
}
% Theorem~\ref{thm:fractionalmain} shows that the convex hull of $\G$ can be described using that of $\F$.   
% % In the rest of this subsection, we will consider the case where $f(x) = x$. In this case, the fractional program~\eqref{eq:frac+linear} is equivalent to maximizing $b_0\rho + b^\intercal x$ over $\conv(\G)$. Due to Corollary~\ref{cor:frac-func}, we can leverage existing relaxation techniques on the 0-1 bilinear set $\F$ to derive relaxations for $\conv(\G)$. 
% \begin{remark}\label{rmk:frac-func}
% \revision{Assume that $a_0+a^\intercal x$ is positive on $\{0,1\}^n$.  Then, the convex hull of $\mathcal{G}$ is given by
% \begin{equation*}
%     \conv(\mathcal{G})=\bigl\{(\rho,y,g)  \bigm| (\rho, y,g) \in \rho \conv(\mathcal{F}),\ a_0\rho + a^\intercal y =1,\ \rho \geq 0\bigr\}.
% \end{equation*}
%  If $f(x) = x$ then 
% \[
% \conv(\mathcal{G})=\bigl\{(\rho,y,g)  \bigm| (y,w) \in \rho \QP,\ a_0\rho + a^\intercal y =1,\ \rho \geq 0,\ g_i = \ell_i(y,w) \text{ for } i \in [n] \bigr\},
% \]
% where $\ell_i(y,w) := (a_0+a_i)y_i +  \sum_{j \neq i}a_j w_{ij}$. }
% % where $\mathcal{F}:=\bigl\{(x,f)\in\R^{n}\times\R^{m} \bigm| f=f(y)(a_0+a^\intercal x), \ x\in\{0,1\}^n\bigr\}$.
% % \begin{equation*}
% %     \text{ where } \mathcal{F}:=\bigl\{(x,f)\in\R^{n}\times\R^{m} \bigm| f=f(y)(1+a^\intercal y), \ y\in\{0,1\}^n\bigr\}.
% % \end{equation*}
% % When $f(\cdot)$ is linear, the 0-1 LFP reduces to bilinear programming.
% \end{remark}

\revision{Now, we use Remark~\ref{rmk:frac-func} to relate the convex hull of $\mathcal{G}_I$ and \textsf{BQP}. For $I \subseteq [n]$, let $G(I)$ be the graph with $V = [n]$ and $E = \bigl\{(i,j) \bigm| i \in I,\ j \neq i \bigr\}$. Letting $f(x) = (x_i)_{i \in I}$, Remark~\ref{rmk:frac-func} yields 
\begin{equation}\label{eq:lfp}
 \begin{aligned}
     \conv(\G_I) =  \Bigl\{\bigl(\rho,y, (x_i)_{i \in I} \bigr) \Bigm|  (y,w) &\in \rho \QP_{G(I)},\ \rho \geq 0,\ a_0\rho +  a^\intercal y  =1,\ \\
     &  x_i = \ell_i(y,w)  \text{ for } i \in I  \Bigr\},
 \end{aligned}
 %\left\{ (\rho, x,y,z) \left| \;
 % \begin{aligned}
 %    &\rho +  a^\intercal y  =1 &\\
 %    &0 \leq y_i  \leq \rho & \text{ for } i \in [n] \\
 %    & x_i = (1+a_i) y_i + \sum_{j \neq i} a_jz_{ij} & \text{ for } i \in S\\
 %    &Ay + Bz + Cw \leq b \rho &
 %    \end{aligned}
 %   \right.
%    \right\}.
\end{equation}
where $\ell_i(y,w) := (a_0+a_i)y_i +  \sum_{j \neq i}a_j w_{ij}$. Observe that we relax $x_i\times (a_0+a^\intercal x)$ instead of $x_i\times \frac{1}{a_0+a^\intercal x}$ over $\{0,1\}^n$ using techniques in~\cite{quesada1995global,tawarmalani2001semidefinite}. Since we can distribute the product over the summation and allows us to leverage \textsf{BQP} literature to derive inequalities, our approach performs significantly better (see Section~\ref{section:computation}). }

Next, we use valid inequalities of \textsf{BQP} and~\eqref{eq:lfp} to derive, for various settings, explicit descriptions of the $k$-term relaxations. For $I \subseteq V$ with $\vert I \vert = 1$, the graph $G(I)$ is acyclic. Proposition 8 in~\cite{padberg1989boolean} shows that, in this case, $\QP_{G(I)}$ is described by McCormick inequalities. This, together with~\eqref{eq:lfp}, yields that the following explicit description of the $1$-term relaxation of \textsf{LFP}:
\begin{equation}\label{eq:1term}
\begin{aligned}
\Bigl\{(\rho,y,x) \Bigm|  (\rho,y,w) &  \text{ satisfies}~(\ref{eq:mc} ) ,\ a_0\rho + a^\intercal y = 1,\ \rho\geq 0,\   \\
&  x_i = \ell_i(y,w) \text{ for } i \in [n] \Bigr\}.
\end{aligned}\tag{\textsc{1-Term}}
\end{equation}
%The triangle inequalities of \textsf{BQP} yield an explicit description of the $2$-term relaxation of \textsf{LFP}.
% \[
% \begin{aligned}
% \Bigl\{(\rho,x,y) \Bigm| (\rho,y,w) & \text{ satisfies}~(\ref{eq:tri} ) \text{ for } i,j,k \in [n],\  \\
% &\rho + a^\intercal y = 1,\ 0\leq y \leq \rho,\  x = L(y,w)  \Bigr\}.
% \end{aligned}
% \]
We next characterize the $k$-term relaxation for $k \geq 2$ using a classical decomposition result (see \cite{SCHRIJVER1983104} and Theorem 1 in~\cite{del2018decomposability}).  
\begin{lemma}\label{lemma:common_simplex}
For $i \in [m]$, let $D_i$ be a subset in the space of variable $(x_i,y)$, and let $D := \bigl\{(x,y) \bigm| (x_i,y) \in D_i \; \text{ for } i \in [m] \bigr\} $. If $\proj_y(D_1) = \cdots = \proj_y(D_m) = V$  and $V$ is a set of finite affinely independent points, $\conv(D) = \bigl\{(x,y) \bigm| (x_i,y) \in \conv(D_i) \text{ for }  i \in [m] \bigr\}$. 
\end{lemma}

\begin{corollary}\label{cor:two-term}
        The $2$-term relaxation is given by~\eqref{eq:1term} and~\eqref{eq:tri}.
% \[
% \begin{aligned}
% \bigl\{(\rho,x,y) \bigm|\eqref{eq:1term} \text{ and }(\ref{eq:tri} ) \bigr\}.
% \end{aligned}
% \]
\end{corollary}

The $k$-term relaxation requires a characterization of $\QP_{G(I)}$, for $I\subseteq V$ with $\vert I \vert = k$. Unlike the case for $k=2$, when $k\ge 3$, Lemma~\ref{lemma:common_simplex} does not decompose $\QP_{G(I)}$ into lower dimensional sets. However, $k^{\text{th}}$-level RLT can be used to describe the $k$-term relaxation. This result relies on the observation that $k^\text{th}$-level RLT introduces variables so that the decomposition lemma can be invoked. 

\begin{proposition}\label{prop:rlt-lfp}
     For $2\le k \le n-1$,  the $k$-term relaxation is given as follows
     \[
     \begin{aligned}
\Bigl\{(\rho,y,x) \Bigm| (y,w) & \in \rho \RLT_{k},\ a_0\rho + a^\intercal y = 1,\ 0\leq y \leq \rho,\ \\
& x_i = (a_0+a_i)y_i + \sum_{j \neq i}a_{ij}w_{ij} \text{ for } i \in [n]  \Bigr\},         
     \end{aligned}
     \]
     where, for each $S\subseteq [n]$ with $2 \leq \vert S\vert \leq k+1$, the variable $w_S$ represents $\frac{\prod_{j \in S}x_j}{a_0+a^\intercal x}$. 
\end{proposition}

\subsection{Bilinear fractional programs with box constraints and \textsf{BQP}}
\def \Q {\mathcal{Q}}
In this subsection, we consider the bilinear fractional programming \revision{problem}  with box constraints
\begin{equation}\label{eq:bfp}
    \min \biggl\{ \frac{x^\intercal B x + d^\intercal x + d_0}{x^\intercal A x + c^\intercal x + c_0} \biggm| x \in [0,1]^n \biggr\}, 
\end{equation}
where $A$ and $B$ are two $n \times n$ upper triangular matrices with \textit{zero-diagonal}, and $c,d\in \R^n$. We assume that the denominator  $x^\intercal A x + c^\intercal x + c_0$ is positive over the box. The main result of this subsection, Proposition~\ref{prop:qfp-bqp}, implies that the continuous optimization problem~\eqref{eq:bfp} is equivalent to its discrete counterpart, that is, there exists an optimal solution $x^*$ such that $x^*\in\{0,1\}^n$. This class of  optimization problems appears in many application areas, \textit{e.g.}, bond portfolio optimization~\cite{konno1989bond}, maximum mean dispersion problem~\cite{prokopyev2009equitable}, and feature
selection~\cite{mehmanchi2021solving}. 

We will show that \eqref{eq:bfp} is closely related to \textsf{BQP}. Towards this end, we linearize the objective function by introducing variables. Let $G=(V,E)$ be a graph with $V = [n]$ and 
\[
E = \bigl\{(i,j) \bigm| a_{ij} \neq 0 \bigr\} \cup \bigl\{(i,j) \bigm| b_{ij} \neq 0 \bigr\},
\]
which consists of the indices of the non-zero entries in $A$ and $B$. With the graph $G$, we denote the denominator as 
\begin{equation}\label{eq:defqG}
q_G(x) = \sum_{(i,j) \in E}a_{ij} x_ix_j + \sum_{i \in V} c_ix_i +c_0.
\end{equation}
Now, we introduce a variable $\rho\in \R$, a variable $y_i$ for each node $i \in V$, and $w_{ij}$ for each edge $(i,j) \in E$, and require that $(\rho,y,w)\in \G_G$, where
\[
\begin{aligned}
\G_G := \left\{ \left( \frac{(1,x,z)}{q_G(x)} \right) \;\middle|\; x \in [0,1]^{\vert V \vert},\ z_{ij}  = x_ix_j \text{ for } (i,j) \in E \right\}.     
\end{aligned}
\]
It follows readily that the bilinear fractional program~(\ref{eq:bfp}) is equivalent to 
\[
\min \Biggl\{ \sum_{(i,j)\in E} b_{ij}w_{ij} + \sum_{i \in V}d_iy_i + d_0\rho  \Biggm| (\rho, y,w) \in \G_G \Biggr\}.
\]
The set $\G_G$ above can be replaced with $\conv(\G_G)$ without affecting the optimal value since the objective function is linear. In other words, it suffices to study  $\conv(\G_G)$.

Now, we show that the convexification of $\G_G$ and of the boolean quadric polytope $\QP_G$ are equivalent, and, thus polyhedrality of $\QP_G$ implies that of $\conv(\G_G)$.
% and that of the convex hull of the bilinear functions with box constraints 
%
%One of the challenges in the proof is to argue the polyhedrality of $\conv(\G_G)$. 
%%This is done by invoking a well-known result in~\cite{burer2009nonconvex} on the polyhedral 
% Instead of proving this directly, we first use a well-known result in~\cite{burer2009nonconvex} to argue that the convex hull of the image of $\Q_G$, under the transformation $\Phi$, is a polytope, and then use Theorem~\ref{thm:fractionalmain} to obtain a polyhedral description for $\conv(\Q_G)$.
\begin{proposition}\label{prop:qfp-bqp}
    Assume that $q_G(\cdot)$ is positive on $[0,1]^{\vert V \vert}$. Then,
    \[
    \begin{aligned}
    \conv(\G_G)= \biggl\{ (\rho, y,w) \biggm| ( y,w) \in &\rho  \QP_G ,\ \rho \geq 0,\ \\
    & \sum_{(i,j) \in E} a_{ij}w_{ij} + \sum_{i \in V}c_i y_{i} + c_0\rho  = 1  \biggr\}. 
    \end{aligned}
    \]
    Moreover, $\QP_G = \bigl\{ (x,z) \bigm| (1,x,z) \in \sigma \conv(\G_G),\ \sigma \geq 0 \bigr\}$.
    % \[
    % \QP_G = \Bigl\{ (x,z) \Bigm| (1,x,z) \in \sigma \conv(\Q_G),\ \sigma > 0 \Bigr\}. 
    % \]
\end{proposition}
\revision{Next, we discuss some consequences of this equivalence result. It is shown in \cite{fiorini2012linear} that, for a complete graph $G$, there does not exist a polynomial-sized extended formulation for $\conv(\G_G)$. Neither does a polynomial-time separation algorithm exist unless $\text{P}=\text{NP}$. Nevertheless, we exploit \textsf{BQP} to construct relaxations of $\G_G$. Let the \textit{odd-cycle} relaxation of $\G_G$ be as follows: }
\begin{equation*}
    \begin{aligned}
        \Bigl\{(\rho,y,w) \Bigm| (\rho,y,w)& \text{ satisfies}~(\ref{eq:oddcycle}) \text{ for all possible $C$ and $D$},\  \\
      & \rho \geq 0,\  \sum_{(i,j) \in E} a_{ij}w_{ij} + \sum_{i \in V}c_i y_{i} + c_0\rho  = 1 
    \Bigr\}.
    \end{aligned}
\end{equation*}
\revision{ Using standard results obtained for $\QP_G$, this relaxation describes the convex hull of $\G_G$ if $G$ is a series-parallel graph.  A graph is \textit{series-parallel} if it arises from a forest by repeatedly replacing edges by parallel edges or by edges in series. }
%Equivalently, it is shown in~\cite{duffin1965topology} that a graph is series-parallel if and only if it has no $K_4$ minor. 
\begin{corollary}\label{cor:poly-case}
   If $G$ is series-parallel graph then~$(\ref{eq:bfp})$ is polynomial-time solvable using the odd-cycle relaxation of $\G_G$. 
\end{corollary}
% Last, we remark that for an arbitrary graph $G$, the odd-cycle relaxation is stronger than the Chv{\'a}tal-Gomory closure of the McCormick relaxation of $\Q_G$. 
% \begin{remark}
%     s
% \end{remark}
% In Section~\ref{section:computation}, we implement the odd-cycle relaxation, and ....  This coincides with the fact that efficient  separation of odd-cycle inequalities is a key component of binary quadratic optimization solvers~\cite{RRW10,charfreitag2022mcsparse,rehfeldt2023faster}.

\begin{remark}
\def \M {\mathcal{M}}
    Proposition~\ref{prop:qfp-bqp} generalizes to multilinear fractional programming with box constraints. This problem is equivalent to minimizing a linear function over the \textit{multilinear polytope}, a polytope that has been studied extensively, see~\cite{del2017polyhedral,del2018multilinear,del2018decomposability,del2021running,del2022multilinear} and references therein. We elaborate on how this research can be leveraged for multilinear fractional programming. Let $G = (V, E)$ be a hypergraph, where $V$ is the set of nodes of $G$, and $E$ is a set of subsets of $V$ of cardinality at least two, named the hyperedges of $G$. With a hypergraph $G$, the multilinear polytope $\MP_G$ is defined as the convex hull of 
    \[
    \biggl\{(x,z) \biggm| x\in \{0,1\}^{\vert V\vert},\ z_e =  \prod_{i \in e} x_i \text{ for } e \in E \biggr\}.
    \]
On the other hand, minimizing the ratio of two multilinear functions over the box constraints is equivalent to minimizing a linear function over a set $\M_G$, which is defined as follows. With a hypergraph $G$, we associate a multilinear function $q_G(x) = \sum_{e \in E} a_e\prod_{i \in e} x_i + c^\intercal x +c_0$ and a set
\[
\begin{aligned}
\M_G = \biggl\{(\rho, y,w) \biggm| x &\in [0,1]^{\vert V\vert},\ \rho = \frac{1}{q_G(x)} ,\ \\
&y_i = x_i\rho \text{ for } i \in V ,\ w_e = \rho\prod_{i\in e}x_i \text{ for } e \in E \biggr\}. 
\end{aligned}
\]
Then, by Theorem~\ref{thm:fractionalmain}, it follows that
\[
 \conv(\M_G) = \biggl\{(\rho,y,w) \biggm| (y,w) \in \rho \MP_G,\ \rho \geq 0,\  \sum_{e \in E} a_e w_e + \sum_{i\in V} c_i x_i + c_0\rho =1 \biggr\}.
\]
Here, although the original problem is stated over $[0,1]^n$, the restriction to $\{0,1\}^{|V|}$ in $\text{MP}_G$ does not change the optimal value because simultaneous convex hull of multilinear functions over $[0,1]^n$ is the same as that over $\{0,1\}^n$, see Corollary 2.7 in  \cite{tawarmalani2010inclusion}. \Halmos
\end{remark}

\section{Ratio of quadratics via copositive optimization}\label{section:copositive}
In this section, we consider a class of quadratic fractional optimization problems defined as follows
\begin{equation}\label{eq:frac_quadratic}
    \min \biggl\{ \frac{x^\intercal B x + b^\intercal x + b_0}{x^\intercal A x + a^\intercal x + a_0} \biggm| x \in \X \cap \mathcal{L} \biggr\}, 
\end{equation}
where $\X$ is a subset of $\R^n$, and $\mathcal{L}$ is an affine set of $\R^n$.  Assume that $\X \cap \mathcal{L}$ is non-empty and bounded, and the denominator $q(x):=x^\intercal A x + a^\intercal x + a_0$ is positive over $\X \cap \mathcal{L}$. Let 
\[
\G = \biggl\{\frac{(1,x,xx^\intercal)}{q(x)} \biggm| x \in \X \cap \mathcal{L} \biggr\}.
\]
Then, the quadratic fractional program~\eqref{eq:frac_quadratic} is equivalent to minimizing the linear function $b_0\rho + \langle b,y \rangle + \langle B, Y\rangle$ over $\G$. Thus, optimzing the linear function over the convex hull of $\G$ would yield a relaxation with no gap. In the following, we will establish an equivalence between the convexification of $\G$ and $\F$, where 
\[
\F = \bigl\{(x,X) \bigm| x \in \X,\ X = xx^\intercal \bigr\}.
\]
The set $\F$ has been studied extensively in various settings;~see \cite{sturm2003cones,burer2009copositive,anstreicher2010computable,burer2015trust,ho2017second,yang2018quadratic,joyce2023convex} for various recent results and \cite{burer2015gentle} for a survey of earlier results. 

To describe the convex hull characterizations of $\F$, we introduce the following notation that we will use throughout the rest of the section. \revision{A matrix $M \in \SM^d$  is called completely positive if there exists $k$ nonnegative vectors $h_1, h_2, \ldots, h_k$ in $\R^d_+$ such that $M = \sum_{i \in [k]} h_i{h_i}^\intercal$.} We will use $\cp_d$ to denote the set of $d\times d$ completely positive matrices.  Given $(\rho,x,X) \in \R \times \R^d \times \mathbb{S}^d$, let 
\[
Z(\rho,x,X) := \begin{pmatrix}
	\rho && x^\intercal \\
	x && X
\end{pmatrix} \in \SM^{d+1}. 
\]

Now, we prove the equivalence of convexification between $\G$ and $\F$. Besides the presence of fractions in $\G$, the sets $\G$ and $\F$ also differ since the domain in $\G$ is intersected with an affine set. Thus, the equivalence result requires a facial decomposition besides Theorem~\ref{thm:fractionalmain}. The techniques are inspired by the proof of Theorem 2.6 in \cite{burer2009copositive}. 
\begin{proposition}\label{prop:frac-quad}
Assume that $q(x) >0$ over a non-empty bounded set $\X \cap \mathcal{L}$. Suppose that $\mathcal{L} = \{x \mid Cx=d\}$, where $C\in \R^{m\times n}$ and $d$ is a vector in $\R^m$. Then, 
	\[
	\begin{aligned}
\conv(\G) = \Bigl\{(\rho,y,Y) \Bigm| (y,Y)& \in \rho \conv(\F),\ \langle A, Y \rangle + \langle a, y \rangle + a_0\rho =1,\ \rho \geq 0, \\
& \Tr\bigl( CYC^\intercal - Cyd^\intercal - dy^\intercal C^\intercal + \rho dd^\intercal \bigr) = 0  \Bigr\}. 		
	\end{aligned}
\]
If the affine set $\mathcal{L}$ is $\R^n$ then $\conv(\F) = \bigl\{ (x,X) \bigm| (1,x,X) \in \sigma \conv(\G),\ \sigma \geq 0 \bigr\}$.
%\[
%\conv(\Q) = \Bigl\{ (x,X) \Bigm| (1,x,X) \in \sigma \conv(\G),\ \sigma \geq 0 \Bigr\}.
%\]
\end{proposition}

%\begin{remark}
%{\color{blue}A remark about square-then-convexification?}
%\end{remark}

Next, we specialize the results in Proposition~\ref{prop:frac-quad} to three cases. First, we consider the case when $\X$ is the non-negative orthant. It follows readily that the convex hull of $\F$ can be described by the cone of completely positive matrices, that is 
\[
\conv(\F)= \bigl\{(x,X) \bigm|  Z(1,x,X) \in \cp_{n+1} \bigr\}.
\]
Thus, by Proposition~\ref{prop:frac-quad}, we obtain a copositive programming reformulation for minimizing a quadratic fractional function over a polytope. This result extends the copositive programming formulation for minimizing a quadratic function over a polytope shown in~\cite{burer2009copositive,burer2015gentle}. 
\begin{corollary}
If $\X$ is the non-negative orthant then~\eqref{eq:frac_quadratic} can be formulated as 
\[
\begin{aligned}
&\min  && b_0\rho + b^\intercal y +  \langle B, Y \rangle \\
&\st && Z(\rho, y, Y) \in \cp_{n+1} 	\\
&&& \langle A, Y \rangle + \langle a, y \rangle + a_0\rho =1 \\
&&&  \Tr\bigl( CYC^\intercal - Cyd^\intercal - dy^\intercal C^\intercal + \rho dd^\intercal \bigr) = 0.
\end{aligned}
\]
\end{corollary}

\begin{remark}
    Binary variables can be handled in this setup easily. We will rewrite $0\le x_i\le 1$ as $x_i+s_i=1$ with $x_i, s_i\ge 0$ . Assume that in the matrix $X$ the row (column) index of variable $x_i$ is $p$ and that of $s_i$ is $q$. All entries in a completely positive matrix and, therefore, $X_{pq}$ is non-negative. In other words, if we require $X_{pq}=0$, we are restricted to a face of the set of completely positive matrices. Since each extreme ray of the completely positive matrices is of the form $hh^\intercal$ for some $h\ge 0$, it follows that extreme rays of the cone that belong to this face satisfy that either $h_p$ or $h_q$ is zero. Moreover, since all extreme rays satisfy $h_p+h_q=1$, because $x_i+s_i=1$ is imposed as $\mathcal{L}$ is enforced in Proposition~\ref{prop:frac-quad}, it follows that each extreme ray, $h$, satisfies the binary conditions. \Halmos
\end{remark}

Second, we consider the case when $\X$ is a ball of radius $1$ centered at origin, that is $\{x\mid\|x\|_2\le 1\}$ where $\| \cdot \|_2$ is the Euclidean norm.  In this case, it follows from~\cite{sturm2003cones,burer2015gentle} that the convex hull of $\F$ can be described using the Shor's semidefinite programming (SDP) relaxation  and reformulation-linearization constraints, 
\[
\conv(\F) = \bigl\{(x,X) \bigm| Z = Z(1,x,X),\ Z \succeq 0,\ \langle L, Z \rangle \geq 0\bigr\},
\] 
where $L:= \Diag(1,-1, \ldots, -1)$. \revision{This, together with Proposition~\ref{prop:frac-quad}, yields a polynomial time solvable SDP formulation for~\eqref{eq:frac_quadratic}. When we specialize this formulation to the case where the affine space $\mathcal{L}$ is $\R^n$, we obtain a polynomial time solvable SDP formulation for optimizing the ratio of two quadratics over an ellipsoid, thus recovering  the result of  \cite{beck2009convex}. On the other hand, this result can be generalized by using results that convexify $\F$ when $\X$ is jointly defined by a ball and additional non-intersecting constraints~\cite{burer2015trust,ho2017second,yang2018quadratic,joyce2023convex}.}
% \begin{corollary}\label{cor:ball}
% 	If $\X$ is a ball $\{x\mid\|x\|_2\le 1\}$ where $\| \cdot \|_2$ is the Euclidean norm then~\eqref{eq:frac_quadratic} can be formulated as
% 	\[
% \begin{aligned}
% &\min  && b_0\rho +  b^\intercal y +  \langle B, Y \rangle \\
% &\st && Z = Z(\rho,y,Y),\  Z\succeq 0	\\
% &&& \langle L, Z \rangle  \geq 0 \\
% &&& \langle A, Y \rangle + a^\intercal y + a_0\rho =1 \\
% &&& \Tr\bigl( CYC^\intercal - Cyd^\intercal - dy^\intercal C^\intercal + \rho dd^\intercal \bigr) = 0,
% \end{aligned}
% \]
% where $L:= \Diag(1,-1, \ldots, -1)$.
% \end{corollary}
% Using results that convexify $\F$ when $\X$ is jointly defined by a ball and additional non-intersecting constraints~\cite{burer2015trust,ho2017second,yang2018quadratic,joyce2023convex}, Proposition~\ref{prop:frac-quad} yields a slight generalization of Corollary~\ref{cor:ball}. On the other hand, when we specialize Corollary~\ref{cor:ball} to the case where the affine space $\mathcal{L}$ is $\R^n$, we obtain a polynomial time solvable SDP formulation for optimizing the ratio of two quadratics over an ellipsoid, thus recovering  the result of  \cite{beck2009convex}. 

Last, we use Proposition~\ref{prop:frac-quad} to derive tractable envelopes of bivariate quadratic fractional functions. More specifically, we consider the graph of the ratio of bivariate quadratic and linear functions  over a convex quadrilateral
\[
\S = \biggl\{ (t,x) \in \R \times  \R^2 \biggm| t = \frac{x^\intercal B x + b^\intercal x + b_0}{a^\intercal x +a_0},\ x \in \X \biggr\},
\]
where $\X$ is a convex quadrilateral in $\R^2$, and  $a^\intercal x +a_0$ is assumed to be positive over the domain $\X$. For the special case where the domain is a box in the positive orthant, the numerator is a non-negative linear function, and the denominator is a positive linear function,~\cite{tawarmalani2001semidefinite} relies on the linearity of the numerator to limit the domain of the function, and then derives the convex and concave envelope of the fraction. Special cases of these relaxations are used in mixed-integer nonlinear programming solvers~\cite{tawarmalani2005polyhedral,misener2014antigone,bestuzheva2023global} to relax factorable functions. We will now present an SDP formulation for the convex hull of $\S$. This generalizes the class of fractional functions to allow for a nonlinear numerator, a setting in which the convex extension argument does not apply. 
%This approach is useful since in general it is not evident that the convex extension theory 

\begin{corollary}\label{cor:bivariate}
    Assume that $\X = \{x \mid Cx \leq d\}$ is a convex quadrilateral in the plane, and $a^\intercal x +a_0$ is positive on $\X$. Then, an SDP formulation for $\conv(\S)$ is given as follows
    \[
    \begin{aligned}
&t = b_0\rho + b^\intercal y + \langle B, Y \rangle \\    	
& x_i = a_1Y_{i1} + a_2 Y_{i2} + a_0y_i ,\  i =1,2 \\
& a_0\rho + a^\intercal y = 1 \\
& Z = Z(\rho,y,Y),\ Z \succeq 0 \\
& CYC^\intercal -Cyd^\intercal - dy^\intercal C^\intercal + \rho dd^\intercal \geq 0.
    \end{aligned}
    \]
\end{corollary}

\section{Multiple denominators}

\revision{In this section, we treat the case where there are multiple denominator expressions. In Section~\ref{section:univariatefractional}, we consider the univariate case in two settings. First, we treat the case where the variable is raised to various powers in the denominator, and second, we consider the case where expressions in the denominator involve shifts of this variable by different constants. In each case, we derive convex hull formulation using the moment-hull characterization. In Section~\ref{section:general-linear-frac} we consider the multivariate case for $0-1$ variables, where multiple linear expressions form denominators. In this case, we develop a hierarchy of relaxations converging to the convex hull. }

\subsection{Univariate case: convex hull formulation}\label{section:univariatefractional}
 \def \M {\mathcal{M}}

In this section, we are interested in developing convex hulls for:
\[
\G_{p,q} = \biggl\{\Bigl(\frac{1}{x^p},  \frac{1}{x^{p-1}}, \cdots,\frac{1}{x}, 1, x, \cdots, x^{q}\Bigr) \biggm| x\in \X  \biggr\},
\]
and
\[
\G_{r} = \biggl\{\Bigl(1,\frac{1}{x-r_1}, \ldots, \frac{1}{x-r_n}, x-r_0 \biggm| x \in \X\biggr\},
\]
where $p,q\in \Z_+$, $r\in \R^n$ is a vector such that $r_1<...<r_n$, $\X$ is a closed set such that none of the denominators attain zero.
%an interval that does not contain $0$ if $p > 0$ and $I_r$ is a closed interval contained in $\bigcup_{i=1}^{n-1}(r_i,r_{i+1})\cup(-\infty, r_1)\cup(r_n,\infty)$.
The first result generalizes the moment hull to allow fractional terms. This is useful to simultaneously relax fractions and powers instead of the current practice of relaxing them separately. \revision{The set $\G_r$ occurs in chemical process design problems where Underwood equations are modeled using functions in $\G_r$ and $r_i$ are relative volatilities of components in a chemical mixture~\cite{underwood1949fractional}.} 
%These equations model minimum vapor requirements of a distillation column and in models of exergy associated with heat exchangers~\cite{gooty2023exergy}. 

We will relate the convex hulls of $\G_r$ and $\G_{p,q}$ to the convex hull of the moment curve which is described as follows. Let $d \in \Z_+$ and $\X$ be a subset of $\R$. Then, the moment curve in $\R^{d+1}$ with support $\X$ is defined as follows 
\[
\mathcal{M}_d = \bigl\{(1, x, x^2, \ldots, x^d) \bigm| x \in \X \bigr\}. 
\]
When $\X$ is a discrete set (resp. an interval) the convex hull of $\M_d$ is referred to as a \textit{cyclic polytope} (resp. \textit{moment hull}). The Gale evenness condition in~\cite{gale1963neighborly}  provides a characterization of facets of a cyclic polytope,~\cite{bogomolov2015small} provides an extended formulation for a cyclic polytope, and ~\cite{fawzi2016sparse} provides a polynomial-sized SDP formulation for a certain class of cyclic polytopes. The moment hull is SDP representable as follows, see Section 3.5.4 in~\cite{blekherman2012semidefinite}. Suppose that $\X=[a,b]$, and let $H:\R^{2d+1} \to \mathbb{S}^{d+1}$ be a linear transformation that maps $\mu = (\mu_0,\mu_1, \ldots, \mu_{2d})$ to a Hankel matrix $H(\mu)$, that is
\[
H(\mu) = \begin{pmatrix}
    \mu_0 & \mu_1 & \mu_2 & \cdots & \mu_d \\
    \mu_1 & \mu_2 & \mu_3 &\cdots & \mu_{d+1} \\
    \mu_2 & \mu_3 & \mu_4 &\cdots & \mu_{d+1} \\
    \vdots & \vdots & \vdots & \ddots & \vdots \\
    \mu_d & \mu_{d+1} & \mu_{d+2} &\cdots & \mu_{2d} 
\end{pmatrix}.
\]
If $d$ is odd, $(\mu_0, \mu_1, \ldots, \mu_d) \in  \conv(\M_d)$  if and only if
\[
\begin{aligned}
H(\mu_1, \ldots,\mu_{d} ) - aH(\mu_0, \ldots,\mu_{d-1}) &\succeq 0 \\
bH(\mu_0, \ldots,\mu_{d-1} ) - H(\mu_1, \ldots, \mu_d) &\succeq 0 \\
\mu_0&=1.
\end{aligned}
\]
If $d$ is even, $(\mu_0, \mu_1, \ldots, \mu_d) \in  \conv(\M_d)$  if and only if
\[
\begin{aligned}
H(\mu_0,\ldots, \mu_{d})&\succeq 0 \\
-H(\mu_2, \ldots, \mu_{d}) +(a+b) H(\mu_1,  \ldots, \mu_{d-1}) - ab H(\mu_0,  \ldots, \mu_{d-2}) &\succeq 0 \\
\mu_0&=1.
\end{aligned}
\]
In our study, we will use the homogenization of $\conv(\M_d)$, that is, the conic hull of $\M_d$, denoted as $\cone(\M_d)$. An explicit description of $\cone(\M_d)$ is obtained by dropping the constraint $\mu_0 = 1$ in a polyhedral/SDP representation of $\conv(\M_d)$. 

We start by establishing that convexifying the moment curve and the curve $\G_{p,q}$ are equivalent. The proof of this result follows directly from Theorem~\ref{thm:fractionalmain}. 
\begin{corollary}
    Let $p,q \in \Z_+$, and assume that $x^p>0$ for every $x \in \X$. Then,
    \[
    \begin{aligned}
        \conv(\G_{p,q}) &= \bigl\{ (\nu_0,\nu_1, \ldots, \nu_{p+q})\bigm| \nu \in  \cone(\mathcal{M}_{p+q}),\ \nu_p = 1  \bigr\} \\
        \conv(\M_{p+q}) &= \bigl\{ (\mu_0,\nu_1, \ldots, \mu_{p+q}) \bigm| \mu \in  \cone(\G_{p,q}),\ \mu_0 = 1  \bigr\}.
    \end{aligned}
    \]
\end{corollary}
%\begin{proof}{Proof.}
%This result follows from Theorem~\ref{thm:fractionalmain}. \Halmos
%To obtain the convex hull of $\G_{p+1}$, we use the transformation $\Phi$  with the coordinate $\frac{1}{x^p}$ treated as  $\rho$
%\[
%\begin{aligned}
%\Phi(\G_{p,q}) &= \bigl\{\tilde{\nu} \in \R^{1+p+q} \bigm| \tilde{\nu}_0 = \tilde{\nu}_{p},\ (1,\tilde{\nu}_1, \ldots, \tilde{\nu}_{p+q}) \in \M_{p+q} \bigr\}.
%\end{aligned}
%\]
%To obtain the convex hull of $\M_{p+q}$, we use the transformation $\Phi$ with the coordinate $x^p$ treated as $\rho$
%\[
%\Phi(\M_{p+q}) = \bigl\{\tilde{\mu} \in \R^{1+p+q} \bigm| \tilde{\mu}_0 = \tilde{\mu}_p,\ (\tilde{\mu}_{0}, \ldots, \tilde{\mu}_{p-1},1,\tilde{\mu}_{p+1},\ldots, \tilde{\mu}_{p+q}) \in \G_{p,q} \bigr\}.
%\]
%\end{proof}
%Then, the results follow from Theorem~\ref{thm:mainthem}. \Halmos
% Consider the case when $\X = [a,b]$. The conic hull of $\mathcal{M}_{p+q}$ can be described using semidefinite programming. If $d:=p+q$ is even then $\nu \in \conv(\G_{p,q})$ if and only if 

% \[
% \begin{aligned}
% \mathcal{H}(\nu_1, \ldots,\nu_{d} ) - a\mathcal{H}(\nu_0, \ldots,\nu_{d-1}) &\succeq 0 \\
% b\mathcal{H}(\nu_0, \ldots,\nu_{d-1} ) - \mathcal{H}(\nu_1, \ldots, \nu_d) &\succeq 0 \\
% \nu_p&=1;
% \end{aligned}
% \]
% if $d$ is even then 
% \[
% \begin{aligned}
% \mathcal{H}(\nu_0,\ldots, \nu_{d})&\succeq 0 \\
% \mathcal{H}(\nu_2, \ldots, \nu_{d}) +(a+b) \mathcal{H}(\nu_1,  \ldots, \nu_{d-1}) - ab \mathcal{H}(\nu_0,  \ldots, \nu_{d-2}) &\succeq 0 \\
% \nu_p&=1;
% \end{aligned}
% \]

% \subsection{Power}

% \subsection{Power}

Next, we consider $\G_r$ defined in the beginning of this section, which we refer to as $\G$ by dropping the subscript for the remainder of this section.
In the following, we show that the convex hull of $\G$ is isomorphic to the convex hull of $\M_{n+1}$. Our proof consists of two  steps. Using partial fraction decomposition and Theorem~\ref{thm:mainthem}, we will establish an isomorphism between the convex hull of $\G$ and that of the following curve given by polynomials, 
\[
\F = \Bigl\{\bigl(f_0(x), f_1(x), \ldots, f_{n+1}(x) \bigr) \Bigm| x \in \X \Bigr\},
\]
where $f_{0}(x)= \prod_{j =1}^n(x-r_j)$,  for $i \in [n]$, $f_i(x) = \prod_{j \neq i,0}(x-r_j)$,  and  $f_{n+1}(x)= \prod_{j =0}^n(x-r_j)$. To use this result, we first establish an isomorphism between $\F$ and the moment curve $\M_{n+1}$ in the following lemma. 
% Consider $T$ with $(i,j)$ is $\beta_{ij}$
% \[
% \begin{aligned}
%     \beta_{i0} & = \frac{(r_{0})^i-\sum_{k \in [n]} \beta_{ik}f_k(r_{0})}{f_0(r_{0})} \\
%     \beta_{ij} & = \frac{(r_j)^i}{f_j(r_j)} \quad \quad \quad \text{ for } j \in [n] \\
%     \beta_{in+1} & = \frac{(r_{n+1})^i - \sum_{k =0}^n\beta_{ik}f_k(r_{n+1})}{f_{n+1}(r_{n+1})}
% \end{aligned}
% \]
\begin{lemma}\label{lemma:change-basis}
    Let $r_{n+1}$ be a real number not in $\{r_0, \ldots, r_{n} \}$, and consider an $(n+2) \times (n+2)$ matrix $T$ with entries $(\beta_{ij})_{i,j=0}^{n+1}$, where for each row $i \in \{0\} \cup [n+1]$ 
    \[
    \beta_{ij} = \begin{dcases}
        \frac{(r_j)^i}{f_j(r_j)} &  j \in [n] \\
         \frac{(r_{0})^i-\sum_{k=1}^n \beta_{ik}f_k(r_{0})}{f_0(r_{0})} & j = 0 \\
        \frac{(r_{n+1})^i - \sum_{k =0}^n\beta_{ik}f_k(r_{n+1})}{f_{n+1}(r_{n+1})} &j = n+1.
    \end{dcases}
    \]
    Then, $\mathcal{M}_{n+1} = T \F$, and $\F = T^{-1}\M_{n+1}$. 
    % Consider the linear transformation $T: f \in \F \mapsto \mu \in \M_{n+1}$ defined as follows
    % \[
    % \mu_i = \frac{(r_{0})^i-\sum_{k \in [n]} T_{ik}f_k(r_{0})}{f_0(r_{0})}f_0 + \sum_{j=1}^n\frac{(r_j)^i}{f_j(r_j)} f_j + \frac{(r_{n+1})^i - \sum_{i =0}^n\beta_if_i(r_{n+1})}{f_{n+1}(r_{n+1})}f_{n+1}. 
    % \]
    % Its inverse $T^{-1}$ maps $\M_{n+1}$ to $\F$ is given as 
\end{lemma}
% To define $\tilde{\M}$, we introduce a family of monomials, that is, for $i \in [n]$, $\ell_i(x) := \prod_{j \neq 0,1}(x-r_j)$,  $\ell_{n+1}(x)= \prod_{j \geq 1}(x-r_j)$, and  $\ell_{n+2}(x)= \prod_{j \geq 0}(x-r_j)$. Then, consider 
% \[
% \tilde{\M} = \bigl\{ (\ell_1(x), \ldots, \ell_{n+1}(x), \ell_{n+2}(x) \bigm| x \in \X \bigr\}. 
% \]
We first assume that $\prod_{i=1}^n (x-r_i) > 0$ for all $x\in \X$. The case with $\prod_{i=1}^n (x-r_i) < 0$ can be handled similarly by considering $-f_i(x)$ for some $i$ instead. The applications described above satisfy that $\prod_{i=1}^n(x-r_i)$ is sign-invariant over $\X$.
\begin{proposition}\label{prop:momenthull}
    Assume $r_0,r_1,\ldots,r_n$ are distinct reals such that $\prod_{i=1}^n(x-r_i) > 0$. Then, 
    \[
    \begin{aligned}
        \conv(\G) & = \bigl\{\nu = (\nu_0, \nu_1, \ldots, \nu_{n+1})  \bigm| T\nu \in \cone(\mathcal{M}_{n+1}),\ \nu_0=1 \bigr\}\\
        \conv(\mathcal{M}_{n+1}) & = \bigl\{ \mu = (\mu_0, \mu_1, \ldots, \mu_{n+1})  \bigm| T^{-1}\mu \in  \cone(\G),\ \mu_0=1 \bigr\}.
    \end{aligned}
    \]
\end{proposition}
The assumption that $\prod_{i=1}^n(x-r_i)$ does not change sign can be relaxed as follows.  
%In particular, this assumes that $\X$ consists of points $x$ where the number of $r_i$ values larger than $x$ is even. Since $r_1,\ldots,r_n\not\in \X$, it follows that $\X$ can be split into two subsets, one which satisfies the above condition, and the other where the number of $r_i$ larger than $x$ is odd. 
We can split $\X$ into two subsets, one where the number of $r_i$ values larger than $x$ is even and another where this number is odd.
%The second case can be handled similarly by constructing the convex hull of $\bigl(-f_0(x),\ldots,-f_{n+1}(x)\bigr)$ instead. 
Then, since $\prod_{i=1}^n (x-r_i)$ is sign-invariant over these subsets, the convex hull of their union can be obtained using \revision{disjunctive programming~\cite{balas1998disjunctive}}.
%\begin{proof}{Proof.}
%By partial fraction decomposition, there is $\alpha_1, \ldots, \alpha_n$ such that 
%\[
%\frac{1}{f_0(x)} = \alpha_1 \frac{1}{x-r_1} + \ldots + \alpha_n \frac{1}{x-r_n}.
%\]
%Thus, we obtain $1 = \sum_{i \in [n]}\alpha_ip_i(x)$. Consider 
%\[
%\tilde{\mathcal{L}} = \Bigl\{(\frac{1}{f_0(x)}, 1, \frac{1}{x-r_1}, \ldots, \frac{1}{x-r_n},x-r_{n+1} ) \Bigm| x \in \X \Bigr\}.
%\]
%Then, $\tilde{\mathcal{L}} = \{(\rho,\nu) \in  \R_+ \times \mathcal{L} \mid \rho = \sum_{i\in [n]}\alpha_i\nu_i\}$. It follows readily that 
%\[
%\begin{aligned}
%	\conv(\mathcal{L}) &= \{\nu \mid (\rho,\nu) \in \conv(\tilde{\mathcal{L}})\} \\
%	& = \{\nu \mid (1,\nu) \in \rho \cdot \conv( \Phi(\tilde{\mathcal{L}})),\ \rho >0 \} \\
%	& = \{\nu \mid \nu \in \rho \cdot \conv(\mathcal{\F}),\ \nu_0 =1,\ \rho = \sum_{i\in [n]}\alpha_i\nu_i \}
%\end{aligned}
%\]
%
%	First, we show that $\conv(\mathcal{L}) = \{\nu \mid \nu \in \rho \cdot \conv(\mathcal{P}),\ \nu_0 =1,\ \rho = \sum_{i\in [n]}\alpha_i\nu_i \}$ 
%\end{proof}
% \begin{itemize}
%     \item By partial fraction decomposition, $\conv(\mathcal{L}) \simeq \conv(\tilde{\mathcal{L}})$, where $\tilde{\mathcal{L}}$ is an extension of $\mathcal{L}$ given by adding additional $0^{\text{th}}$ coordinate $\frac{1}{\prod_{i\geq 0}(x-r_i)}$. 
%     \item By Theorem~\ref{thm:mainthem}, $\conv(\tilde{\mathcal{L}}) \simeq \conv(\Phi(\tilde{\mathcal{L}}))$
%     \item By a modified Lagrange interpolation, $\conv(\Phi(\tilde{\mathcal{L}})) \simeq \conv(\M_{n+1})$. 
% \end{itemize}

\subsection{Multivariate case: a hierarchy of relaxations}\label{section:general-linear-frac}
% \begin{itemize}
%     \item an example to illustrate the hierarchy. 
%     \item introduce $z_S = \prod_{j \in S}x_j$ and RLT is done on $z_S$ variables. 
% \end{itemize}
Consider the following problem:
\begin{equation}\label{eq:sumlinearfrac}
   \max  \biggl\{ \sum_{i \in [m]} \frac{b_{i0} + b_i^\intercal x }{a_{i0}+a_i^\intercal x} \biggm|x \in \X \subseteq \{0,1\}^n \biggr\},
\end{equation}
where $m,n\in \Z_+$, $a_i,b_i\in \R^n$, $a_{i0},b_{i0}\in \R$, and $a_{i0}+ a_i^\intercal x > 0$ for every $x \in \X$. Assume $\X$ can be described using a system of linear inequalities, that is, $\X = \{x \mid Cx \geq d\} \cap  \{0,1\}^n $, where $C$ is \revision{an} $r \times n$ matrix and $d$ is an $r$-dimensional vector. Clearly, this maximization problem is equivalent to maximizing  $ \sum_{i \in [m]}b_{i0}\rho^i + b_i^\intercal y^i$ over the following set
\[
 \G:=\biggl\{(\rho,y) \biggm| x \in \X,\ (\rho^i,y^i) =  \frac{(1,x)}{a_{i0}+a_i^\intercal x} \text{ for } i \in [m] \biggr\},
\]
where $\rho = (\rho^1, \ldots, \rho^m)$ and $y = (y^1, \ldots, y^m)$. In this section, we present a hierarchy of convex outer-approximations for $\G$. For $k \in [n]$, the $k^\text{th}$ level relaxation of $\G$ is related to the $k^\text{th}$ level RLT of the feasible region $\X$, denoted as $\RLT_k(\X)$. The relaxation $\RLT_k(\X)$ is defined as the system of linear inequalities obtained by linearizing the following system of  polynomial inequalities using the relation $x_i = x_i^2$ and representing the monomial $\prod_{i \in e}x_i$ with an introduced variable $z_e$ after distributing products:
\[
(Cx-d)\prod_{i\in S}x_i\prod_{j\in T}(1-x_j) \geq 0 \text{ for } S \cap T = \emptyset \text{ with } \vert S \cup T\vert = k.
\]
% \[
% \prod_{i\in S}x_i\prod_{j\in T}(1-x_j) \geq 0 \text{ for } S \cap T = \emptyset \text{ with } \vert S \cup T\vert = \min\{n,k+1\}.
% \]

We begin with a base relaxation  using ideas in Section~\ref{section:linear_frac_poly}. Consider the following lifting $\G^1$ of $\G$ in the space of $(\rho,y,x)$ variables, that is
\[
\G^1 = \biggl\{(\rho,y,x) \biggm| x \in \X ,\ (\rho^i,y^i) =  \frac{(1,x)}{a_{i0}+a_i^\intercal x} \text{ for } i \in [m] \biggr\}. 
\]
More generally, $\G^k$ will denote the set used to construct the $k^{\text{th}}$ level relaxation in our proposed hierarchy. Then, we project $\G^1$ into $m$ subspaces, where for $i \in [m]$, $\G^1_i$ denotes the projection of $\G^1$ onto the space of $(\rho^i,y^i,x)$ variables.
To relax each projection, we introduce variable $w^i_{jr}$ to linearize $\frac{x_jx_r}{a_{i0}+a_i^\intercal x}$, and construct the following relaxation:
\[
   \begin{aligned}
       \mathcal{R}^1_i =  \biggl\{(\rho^i,y^i,x) \biggm|  (y^i,w^i) &  \in \rho^i \RLT_{1}(\X),\ a_{i0}\rho^i + a_i^\intercal y^i = 1,\ \rho^i\geq 0, \\
       &x_j = (a_{i0}+a_{ij})y^i_{j} + \sum_{r \neq j}a_{ir}w^i_{jr} \text{ for } j \in [n]    \biggr\},
   \end{aligned}
\]
where the four linear constraints are obtained by linearizing the following constraints, respectively, 
\[
\begin{aligned}
    \frac{(x,xx^\intercal)}{a_{i0}+a_i^\intercal x} \in \frac{1}{a_{i0}+a_i^\intercal x}\RLT_1(\X) & \\
    \frac{a_{i0}}{a_{i0}+a_i^\intercal x} +\sum_{r \in [m]}\frac{a_{ir}x_{r}}{a_{i0}+a_i^\intercal x} = 1 & \\
    \frac{1}{a_{i0}+a_i^\intercal x} \ge 0 & \\
    x_j = \frac{(a_{i0}+a_{ij})x_j +\sum_{r\ne j} a_{ir}x_jx_r}{a_{i0}+a_i^\intercal x} .&
\end{aligned}
\]
\revision{Now, the base relaxation $\mathcal{R}^1$ is the intersection of relaxations $\mathcal{R}^1_i$, that is }
\[
\mathcal{R}^1 = \bigl\{(\rho,y,x) \bigm| (\rho^i,y^i,x) \in \mathcal{R}^1_i \text{ for } i \in [m] \bigr\}.
\]
We motivate our relaxation hierarchy by first considering a simple way to tighten $\mathcal{R}^1$, which does not yield the convex hull of $\G$. This relaxation replaces $\RLT_1(\X)$ in the definition of $\mathcal{R}^1_i$ with the $n^{\text{th}}$-level RLT relaxation of $\X$. 

\begin{example}
    Consider the case where $m=2$ and $n=2$. Let $a_{10} = 1$ and $a_1 = (2,4)$, and $a_{20} = 1$ and $a_2 = (3,5)$. In this case, the set $\G$ of our interest consists of the following  four points listed in the space of variables   $(\rho^1,y^1_1,y^1_2,\rho^2,y^2_1,y^2_2)$:
    \[
\Bigl(1,0,0,1,0,0\Bigr) \quad \Bigl(\frac{1}{3},\frac{1}{3},0,\frac{1}{4},\frac{1}{4},0\Bigr) \quad \Bigl(\frac{1}{5},0,\frac{1}{5},\frac{1}{6},0,\frac{1}{6}\Bigr) \text{ and } \Bigl(\frac{1}{7},\frac{1}{7} ,\frac{1}{7},\frac{1}{9}, \frac{1}{9} ,\frac{1}{9}\Bigr).
    \]
It can be verified that the linear inequality $25\rho^1 - 4y^1_1 - 24\rho^2  \leq  1$ is valid for $\G$. However, this inequality is not valid for the relaxation obtained by individually convexifying each linear fractional function. To see this, let  
    $
    \G^1_i = \bigl\{(\rho^i,y^i,x) \bigm| (\rho^i,y^i) = \frac{(1,x)}{1+  a_i^\intercal x },\ x \in \{0,1\}^2 \bigr\}
    $.
 Now, maximizing $ 25\rho^1 - 4y^1_1 - 24\rho^2 $ over $\{(\rho,y,x)\mid (\rho^i,y^i,x) \in \conv(\mathcal{G}^1_i) \; \text{ for } i = 1,2\}$ yields $9$, showing that intersecting $\conv(\G^1_i)$ does not yield $\conv(\G)$. \Halmos
%    \begin{aligned}
%    \mathcal{R}_i = \Bigl\{&(\rho^i,y^i,x)  \Bigm| w^i \geq 0,\ w^i \geq y^i_1+y^i_2-\rho^i,\ w^i \leq y^i_1,\  w^i \leq y^1_i,\ \rho^1\geq 0 \\
%    & \rho^i + a_i^\intercal  y^i =1,\  x_1= (1+a_{i1})y^i_1+a_{i2}w^i,\ x_2= a_{i1}w^i+(1+a_{i2})y^1_2 \Bigr\}    
%    \end{aligned}
%    \]
%    \[
%    \begin{aligned}
%    \mathcal{R}^1_2 = \{(\rho^2,y^2,x) \mid w^2 &\geq 0,\ w^2 \geq y^2_1+y^2_2-\rho^2,\ w^2 \leq y^2_1,\  w^2 \leq y^2_2,\\
%    & \rho^2 + 3y^2_1+5y^1_2=1,\rho^2\geq 0,\ x_1= 4y^2_1+5w^1,\ x_1= 3w^2+6y^1_2 \}    
%    \end{aligned}
%    \]
%    \[
%        25\rho^1 - 4y^1_1 - 24\rho^2  \leq  1
%    \]
\end{example}
This discrepancy occurs because we have not exploited that $(y^i,w^i) \in \rho^i \RLT_n(\X)$ and $(y^j,w^j) \in \rho^j \RLT_n(\X)$ are scalar multiples of one another. To exploit this, first observe that the convex hull of $U$ is a simplex, where
\[
 U = \Bigl\{u \Bigm| x \in \X,\ u_S = \prod_{j \in S}x_j \text{ for } \emptyset \neq S\subseteq [n] \Bigr\}. 
\]
Therefore, for any point in the convex hull, there is a unique set of barycentric coordinates (convex multipliers associated with the extreme points of $\conv(U)$.) Since $(y^i,w^i)$ scales to $(y^j,w^j)$, for all $i\ne j$, we write:
\[
u_S = \frac{\bigl(a_{i0}+\sum_{j\in S}a_{ij}\bigr)\prod_{t \in S}x_t +\sum_{r\not\in S} a_{ir}\prod_{t \in S\cup\{r\}}x_t}{a_{i0}+a_i^\intercal x},
\] 
and introduce variable $w^i_S$ to represent $\frac{\prod_{j \in S}x_j}{a_{i0}+a_i^\intercal x}$. 
For each $i\in [m]$ and $S\subseteq [n]$, the above equality describes a linear relation between $u_S$ and $w^i_S$. Moreover, every function of $0\mathord{-}1$ variables, and, therefore, $w^i_{\bar{S}}$, can be expressed as a linear combination of $(u_{S})_{S\subseteq [n]}$.

Formally, for each $2 \leq k \leq n$, we consider a lift of $\G$ given as follows,
\[
\begin{aligned}
\G^k = \biggl\{ (\rho,y,u) \biggm|  (\rho^i,y^i)& =  \frac{(1,x)}{a_{i0}+a_i^\intercal x} \text{ for } i \in [m], x \in \X,\ \\
& u_S = \prod_{j\in S}x_j \text{ for } S \subseteq [n] \text{ with } 1\leq  \vert S \vert \leq k \biggr\}.
\end{aligned}
\]
For each $i \in [m]$, let $\G^k_i$ denote the projection of $\G^k$ onto the space of $(\rho^i,y^i,u)$ variables. A relaxation of $\G^k_i$ is given as follows
    \[
    \begin{aligned}
        \mathcal{R}^k_i &= \biggl\{(\rho^i,y^i,u) \biggm|  w^i   \in \rho^i \RLT_{k}(\X),\ y^i_j=w^i_{\{j\}}\text{ for } j\in[n],\ a_{i0}\rho^i + a_i^\intercal y^i = 1,\\ 
        & \quad \quad u_S =  a_{i0}+\sum_{j \in S}\Bigl(a_{ij} w^i_{S\cup\{j\}}\Bigr) \text{ for } S \subseteq [n] \text{ s.t. } 1 \leq \vert S \vert \leq k \biggr\},
    \end{aligned}
    \]
    where $w^i_T = y^i_T$ if $T = \{j\}$ for some $j \in [n]$. Let $(\mathcal{R}^k)_{k=1}^n$ be a hierarchy of relaxations for $\G$, where
\[
\mathcal{R}^k =  \Bigl\{(\rho,y,u) \Bigm| (\rho^i,y^i,u) \in \mathcal{R}^k_i \text{ for } i \in [m] \Bigr\} \quad \text{ for } k \in [n]. 
\]
% Although for each $\kappa \in [n]$ and $i \in [m]$ the projection of $\mathcal{R}^\kappa_i$ onto the space of $(\rho^i,y^i)$ variable reminds same, the presence of $z$ variable generates linear relations across all fractional functions. 
Clearly, $\mathcal{R}^1 \supseteq \cdots \supseteq \mathcal{R}^n$. In fact, $ \proj_{(\rho,y)} (\mathcal{R}^n) = \conv(\G)$. 
\begin{theorem}\label{them:RLTconverge}
 The following holds: $\proj_{(\rho,y)} (\mathcal{R}^1) \supseteq \cdots \supseteq  \proj_{(\rho,y)} (\mathcal{R}^n) = \conv(\G)$. 
\end{theorem}

\section{Theoretical comparisons and computational experiments}\label{section:computation}
In this section, we consider the constrained mixed-binary fractional programs, propose a relaxation for them using quadratic programming, and show that the resulting relaxation is tighter than relaxations used in the fractional programming literature. Then, we perform a computational study of relaxations for binary fractional programs occurring in assortment optimization using our reformulation based on \eqref{eq:1term}. Then, we consider a set arising from chemical process design and demonstrate that our relaxation based on the construction in Section~\ref{section:univariatefractional} closes significant gap relative to standard relaxation techniques. Our computational experiments are performed on a MacBook Pro with Apple M1 Pro with 10-cores CPU and 16 GB of memory. The code is written in \textsc{Julia v1.6}~\cite{bezanson2017julia}, and our formulations are modeled using \textsc{JuMP v1.4.0}~\cite{lubin2023jump} and solved using  \textsc{Gurobi v10.0.2}~\cite{gurobi} as an MIP solver  and \textsc{SCIP v8.0}~\cite{bestuzheva2023global} as an MINLP solver. 
\subsection{A polynomial-sized relaxation for fractional programs and theoretical comparisons}\label{section:comparisons}
In this section,
%inspired by the computational performance of \eqref{eq:1term},
we leverage Theorem~\ref{thm:fractionalmain} and simple outer-approximations for quadratic programs to propose a polynomial-sized relaxation (see~\eqref{eq:PolytopeQP-MC} and \eqref{eq:BQP-SDP}/\eqref{eq:BQP-SDPa} below) for constrained fractional programs, and theoretically compare its tightness relative to existing relaxations. Our results show that this approach yields relaxations that are tighter than those in the literature. 
%This provides insight into the improved computational performance of our relaxations demonstrated above. 
Specifically, we will focus our attention on:
\begin{equation}\label{eq:computation:lfp}
\max_{x\in\X}~~~\sum\limits_{i=1}^{m}\frac{b_{i0}+b_{i}^\intercal x}{a_{i0}+a_{i}^\intercal x}+c^\intercal x,\tag{\textsc{FP}}
\end{equation}
where we assume that $\X=\bigl\{x\in \{0,1\}^p\times\R^{n-p}\bigm| Cx\le d\bigr\}$ for $C\in \R^{r\times n}$, $d\in \R^r$, $\X$ is bounded, and $a_{i0} + a_i^\intercal x > 0$ for all $x$ that satisfy $Cx\le d$ and $x_j\in [0,1]$ for $j\in [p]$. The special case with $p=n$ will be referred to as \textit{binary fractional programming} (BFP). 
%In this special case, the assumptions reduce to $a_{ij}\ge 0$ for $i\in \{1,\ldots,m\}$ and $j\in\{0,\ldots,n\}$. 
%Although, for many applications, $c=0$, we have introduced it to emphasize the importance of having $x$ variables in our reformulations and relaxations. 

An approach used to reformulate \eqref{eq:computation:lfp} is to introduce $\rho^i=\frac{1}{a_{i0}+a_{i}^\intercal x}$ and  $y^i_j=\frac{x_j}{a_{i0}+a_{i}^\intercal x}$ for all $i\in [m] $ and $j\in [n]$. Then, $\eqref{eq:computation:lfp}$ is reformulated as
\begin{subequations}\label{eq:cc-reformulation}
\begin{alignat}{3}
&\max&\quad&\sum\limits_{i=1}^{m}\big(b_{i0}\rho^{i}+b_i^\intercal y^i\big)+c^\intercal x \notag \\
&\text{s.t.} && x_i \in \{0,1\} && i\in [p] \\
&&&y_j^i= \rho^ix_j&\quad&i\in[m],~j\in[n]\label{eq:ybilinear}\\
&&&a_{i0}\rho^i+a_i^\intercal y^i=1&&i\in[m]\label{eq:normalize}\\
%&&&x_j\in [0,1]&&j\in [p]\label{eq:domain01}\\
&&& \overline{C}x\le \bar{d},\label{eq:domainPolytope} 
\end{alignat}
\end{subequations}
% \begin{alignat}{3}
% &\max_{\boldsymbol{x}\in \X}&\quad&\sum\limits_{i=1}^{m}\big(b_{i0}\rho^{i}+b_i^\intercal y^i\big)+c^\intercal x\\
% &\text{s.t.}&&y_j^i= \rho^ix_j&\quad&i\in[m],~j\in[n]\label{eq:ybilinear}\\
% &&&a_{i0}\rho^i+a_i^\intercal y^i=1&&i\in[m]\label{eq:normalize}\\
% %&&&x_j\in [0,1]&&j\in [p]\label{eq:domain01}\\
% &&&\overline{C}x\le \bar{d},\label{eq:domainPolytope}
% \end{alignat}
where $\overline{C}x\le \bar{d}$ consists of the inequalities $Cx\le d$, and $x_j\in [0,1]$ for $j\in [p]$.
This reformulation leads to one of the most popular polyhedral relaxations for~\eqref{eq:computation:lfp} \cite{li1994global,mehmanchi2019fractional}. The relaxation replaces \eqref{eq:ybilinear} with McCormick inequalities. In particular, letting $\rho^i(L)$ and $\rho^i(U)$ (resp. $x_j(L)$ and $x_j(U)$) be  lower and upper bounds on variable $\rho^i$ (resp. $x_j$), a  polyhedral relaxation for the feasible region of~\eqref{eq:cc-reformulation} is given as follows: 
\begin{equation}\label{eq:LFP-MC}
\mathcal{R}_{\textsf{LEF}}: = \left\{(\rho,y,x) \;\middle|\;
\begin{aligned}
   &y_j^i\leq x_{j}(U)\rho^i+\rho^i(L)x_{j}-x_{j}(U)\rho^i(L)&& \text{for }  i\in[m], j\in[n]\\
&y_j^i\leq x_{j}(L)\rho^i+\rho^i(U)x_{j}-x_{j}(L)\rho^i(U)&& \text{for }  i\in[m],j\in[n]\\
&y_j^i\geq x_{j}(U)\rho^i+\rho^i(U)x_{j}-x_{j}(U)\rho^i(U)&& \text{for }  i\in[m],j\in[n]\\
&y_j^i\geq x_{j}(L)\rho^i+\rho^i(L)x_{j}-x_{j}(L)\rho^i(L)&& \text{for }   i\in[m],j\in[n]\\
&\eqref{eq:normalize}, \eqref{eq:domainPolytope} 
\end{aligned}
\right\},
\end{equation}
where $\rho:=(\rho^i)_{i=1}^m $ and $ y :=(y^i)_{i=1}^m$.
The strength of McCormick inequalities relies heavily on the bounds of the associated variables. For 
unconstrained BFP, we have that $x_{j}(L)=0$, $x_{j}(U)=1$, $\rho^i(L)=\frac{1}{\sum_{j=0}^n a_{ij}}$, and $\rho^i(U)=\frac{1}{a_{i0}}$. For the general case, we assume that these bounds are obtained by minimizing/maximizing their defining functions over the linear constraints using linear programming.

We now consider an alternate relaxation, based on the techniques developed in this paper. To solve \eqref{eq:computation:lfp}, it suffices to convexify $\G:=\cap_{i=1}^m\G^i$, where 
\begin{equation*}
\begin{aligned}
\G^i := \biggl\{ \bigl(\rho^i,y^i,x \bigr) \biggm| & x \in \{0,1\}^p\times\R^{n-p}, \overline{C}x\le \bar{d},\\
 &~ \rho^i=\frac{1}{a_{i0}+a_{i}^\intercal x},~ y^i_j=\frac{x_j}{a_{i0}+a_{i}^\intercal x} \text{ for } j\in[n]\biggr\}. 
\end{aligned}
\end{equation*}
Although, tighter relaxations can be derived using copositive programming as in Section~\ref{section:copositive} or the hierarchy of Section~\ref{section:general-linear-frac}, we focus, here, on the polynomial-sized relaxations from Section~\ref{section:BQP} and RLT constraints obtained at the first level. Applying Theorem~\ref{thm:fractionalmain}, we get
\begin{equation*}
 \begin{aligned}
     \conv(\G^i) =  \Bigl\{\bigl(\rho^i,y^i, x \bigr) \Bigm|  (y^i,W^i) &\in \rho^i \conv(\F),\ \rho^i \geq 0,\ a_{i0}\rho^i +  a_i^\intercal y^i  =1,\  \\
     & x_j = (a_{i0}+a_{ij})y^i_j +  \sum_{k \neq j}a_{ik} W^{i}_{jk},  \text{ for } j \in [n]  \Bigr\},
 \end{aligned}
\end{equation*}
where $\F = \bigl\{(x,xx^\intercal)\bigm| \overline{C}x\le \bar{d},\ x\in\{0,1\}^p\times\R^{n-p}\bigr\}$. As mentioned before, since it is \revision{NP-hard} to optimize linear functions over $\F$, we will replace $\conv(\F)$ with its 1st-level RLT relaxation, $\widehat{F}$. In particular, $\widehat{\F}$ will be defined as:
\begin{equation*}\label{eq:widehatF}
    \widehat{F} = \Bigl\{(x,X)\Bigm| \overline{C}X\overline{C}^\intercal - \bar{d}x^\intercal \overline{C}^\intercal - \overline{C}x\bar{d}^\intercal + \bar{d}\bar{d}^\intercal \ge 0,\ X_{jj} = x_j \text{ for }j\in[p]\Bigr\}.
\end{equation*}
This construction yields a polyhedral relaxation for $\conv(\G)$,
% where $\mathcal{R}_{\QP}$ is the set of points $\bigl((\rho^i)_{i=1}^m,(y^i)_{i=1}^m,x\bigr)$ satisfying 
% \begin{subequations}
% \begin{align}
% \overline{C}W^i\overline{C}^\intercal - \bar{d}(y^i)^\intercal \overline{C}^\intercal - \overline{C}y^i\bar{d}^\intercal + \rho^i\bar{d}\bar{d}^\intercal \ge 0 \\
% a_{i0}\rho^i +  a_i^\intercal y^i  =1,\ \rho^i\geq 0, \quad i\in[m] \\
%     &&&x_j = a_{i0}y^i_j + \sum_{k}a_{ik} W^{i}_{jk},  && i\in[m], ~j \in [n]\\
%     &&&\overline{C}x\le \bar{d}\\
%     &&&\overline{C}y^i\le \rho^i\bar{d} &&i\in [m]\\
%     &&&W^i_{jj} = y^i_j &&j\in[p]

% \end{align}
% \end{subequations}
\begin{equation}\label{eq:PolytopeQP-MC}
\mathcal{R}_{\QP}:=\left\{(\rho,y,x)  \;\middle|\;
    \begin{alignedat}{3} 
    &&&\overline{C}W^i\overline{C}^\intercal - \bar{d}(y^i)^\intercal \overline{C}^\intercal - \overline{C}y^i\bar{d}^\intercal + \rho^i\bar{d}\bar{d}^\intercal \ge 0 \text{ for } i\in[m]  \\
    &&& a_{i0}\rho^i +  a_i^\intercal y^i  =1,\ \rho^i\geq 0  \quad \text{ for } i\in[m] \\
%    &&&w^i_{jk} \ge 0,\ w^i_{jk} \ge y^i_j+y^i_k - \rho^i, 
%    && i\in[m], ~1\leq j<k\leq n\\     &&&w^i_{jk}\leq y^i_j,\ w^i_{jk}\leq y^i_k, && i\in[m], ~1\leq j<k\leq n,\\
    &&&x_j = a_{i0}y^i_j + \sum_{k}a_{ik} W^{i}_{jk}  \quad \text{ for }i\in[m], ~j \in [n]\\
    &&&\overline{C}y^i\le \rho^i\bar{d} \quad \text{ for } i\in [m]\\
    &&&W^i_{jj} = y^i_j \quad \text{ for } j\in[p] \\
        &&&\overline{C}x\le \bar{d}
    \end{alignedat}\right\}.
\end{equation}
It can be shown that  this relaxation is tighter than the relaxation $\mathcal{R}_{\textsf{LEF}}$.
\begin{proposition}\label{prop:QP-MC}
$\mathcal{R}_{\QP}\subseteq \mathcal{R}_{\textsf{LEF}}$.
\end{proposition}

\revision{The containment in Proposition~\ref{prop:QP-MC} is often strict as our computations and the example in Appendix~\ref{app:strictcontainment} demonstrate.}
In mixed-integer nonlinear programming (MINLP) literature, another form of McCormick relaxation is used to solve \eqref{eq:computation:lfp}, see \cite{quesada1995global} for example. This  relaxation writes $z_i=\frac{b_{i0}+b_i^\intercal x}{a_{i0} + a_i^\intercal x}$ as $z_{i}\bigl(a_{i0}+a_i^\intercal x\bigr) = b_{i0} + b_i^\intercal x$ and relaxes the left hand side using McCormick inequalities. Since McCormick inequalities require bounds, we minimize/maximize $\frac{b_{i0}+b_i^\intercal x}{a_{i0} + a_i^\intercal x}$ (resp. $a_{i0}+a_i^\intercal x$) over $\overline{C}x\le \bar{d}$ to derive bounds for $z_i$ (resp. $a_{i0}+a_i^\intercal x$). We show that the McCormick inequalities, so derived, are also implied by $\mathcal{R}_{\QP}$. 

\begin{proposition}\label{prop:genMcCormick}
    Assume $z_i = \frac{b_{i0}+b_i^\intercal x}{a_{i0} + a_i^\intercal x}$ $(\text{resp. $d_i = a_{i0} + a_i^\intercal x$})$ is bounded between $z_i^L$ and $z_i^U$ $($resp. $d_i^L>0$ and $d_i^U$$)$ for all $x$ that satisfy $\overline{C}x\le \bar{d}$. Then, $\mathcal{R}_{\QP}$ implies the inequalities obtained by relaxing the left-hand-side of $z_i d_i = b_{i0} + b_i^\intercal x$ using McCormick inequalities. 
\end{proposition}
The relaxation described in Proposition~\ref{prop:genMcCormick} can be tightened when the bounds on $a_{i0}+a_i^\intercal x$ are implied by bounds on $x$ \cite{tawarmalani2004global}. The idea is to disaggregate the product $z_i(a_{i0}+a_i^\intercal x)$ by distributing the product over summation and relaxing $z_ix_j$ using McCormick envelopes. We show in Appendix~\ref{ec:disaggregate} that the inequalities so obtained are also implied in $\mathcal{R}_{\QP}$.

The following conic inequalities were proposed in~\cite{sen2018conic} to tighten $\mathcal{R}_{\textsf{LEF}}$:
\begin{subequations}\label{eq:conic}
    \begin{align}
        y_j^i(a_{i0}+a_i^\intercal x)\geq x_j^2& \quad  \text{ for } i\in[m],~j\in[p]\label{eq:Conic_a}\\
        \rho^i(a_{i0}+a_i^\intercal x)\geq 1& \quad \text{ for } i\in[m]\label{eq:Conic_b}.
    \end{align}
\end{subequations}
Inequality~\eqref{eq:Conic_b} is commonly used in MINLP literature as $\rho^i\ge \frac{1}{a_{i0} +a_i^\intercal x}$, exploiting that the defining function of $\rho^i$ is convex. It is also used to derive nonlinear underestimators for $x_i\times \frac{1}{a_{i0} +a_i^\intercal x}$ by viewing it as a product of two convex functions \cite{mccormick1976computability,quesada1995global,tawarmalani2001semidefinite}. A conic relaxation for the feasible region of~\eqref{eq:cc-reformulation} that utilizes \eqref{eq:conic} is then given as follows:
\begin{equation*}\label{eq:LFP-CO}
    \mathcal{R}_{\textsf{CEF}}=\Bigl\{\bigl((\rho^i)_{i=1}^m,(y^i)_{i=1}^m,x\bigr) \Bigm| \bigl((\rho^i)_{i=1}^m,(y^i)_{i=1}^m,x\bigr)\in \mathcal{R}_{\textsf{LEF}} \text{ and satisfies \eqref{eq:conic}} \Bigr\}.
\end{equation*}
In \cite{mehmanchi2019fractional}, formulations based on $\mathcal{R}_{\textsf{CEF}}$ are found to be the best performing formulations. Here, we show that \eqref{eq:conic} arises naturally from standard relaxations used in quadratic programming literature. In particular, it is standard to relax $\begin{pmatrix}1\\x\end{pmatrix}\begin{pmatrix}1 & x^\intercal\end{pmatrix}\succeq 0$ into $\begin{pmatrix}1&x^\intercal\\x&X\end{pmatrix}\succeq 0$. Following Theorem~\ref{thm:fractionalmain}, we homogenize this matrix inequality as follows: 
\begin{equation}\label{eq:BQP-SDP}
\begin{aligned}
    \begin{pmatrix}\rho^i & (y^i)^\intercal \\ y^i & W^i \end{pmatrix}\succeq 0 \quad \text{ for } i\in[m],\\
\end{aligned}
\end{equation}
and augment $\mathcal{R}_{\QP}$ with this inequality.
We show this relaxation implies \eqref{eq:conic}. First, we rewrite \eqref{eq:BQP-SDP} in an equivalent form:
\begin{equation}\label{eq:BQP-SDPa}
    \begin{pmatrix}
    a_{i0}+x^\intercal a_i  & 1 & x^\intercal\\
    1 & \rho^i & (y^i)^\intercal\\
    x & y^i & W^i
    \end{pmatrix}\succeq 0 \quad \text{ for } i\in[m].
\end{equation}

\begin{proposition}\label{prop:SQPconic}
    Adding \eqref{eq:BQP-SDPa} to $\mathcal{R}_{\QP}$ is equivalent to adding \eqref{eq:BQP-SDP} to $\mathcal{R}_{\QP}$.
\end{proposition}
Since the matrix in \eqref{eq:BQP-SDPa} is positive-semidefinite, the determinants of its $2\times 2$ principal minors are non-negative. Therefore, Proposition~\ref{prop:SQPconic} shows that, adding \eqref{eq:BQP-SDPa} or, equivalently, \eqref{eq:BQP-SDP} to $\mathcal{R}_{\QP}$ implies the inequalities in \eqref{eq:conic}. In particular, \eqref{eq:Conic_a} (resp. \eqref{eq:Conic_b}) follows by selecting rows and columns indexed $1$ and $j+2$ (resp. $1$ and $2$). Moreover, \eqref{eq:BQP-SDPa} implies the following conic inequalities:
\begin{subequations}\label{eq:conicnew}
    \begin{alignat}{2}
        &y^i_j\rho^i\geq \bigl(y^i_j\bigr)^2 &\quad& \text{ for } i\in[m] \text{ and } j\in[n]\label{eq:Conicnew_a}\\
        &y^i_j y^i_k \geq \bigl(W^i_{jk}\bigr)^2&& \text{ for }i\in[m] \text{ and } j, k\in [n].\label{eq:Conicnew_b}
    \end{alignat}
\end{subequations}
using other $2\times 2$ principal minors of the associated matrix. Clearly, \eqref{eq:Conicnew_a} is redundant since $0\le y^i_j\le \rho^i$. However, \eqref{eq:Conicnew_b} tightens the linear relaxation $\mathcal{R}_{\QP}$ above.

\subsection{Computational comparison on binary fractional programming}\label{section:01fractional}
In this subsection, we present numerical experiments on synthetically generated instances of binary fractional programming (BFP). First, we present our formulation~\textsf{1Term-Conic} for BFP. Recall that~\eqref{eq:cc-reformulation} reformulates~\eqref{eq:computation:lfp}. Here, we use a relaxation of $\mathcal{R}_{\QP}$ along with the conic constraint~\eqref{eq:Conic_b} to relax the feasible region of~\eqref{eq:cc-reformulation} and obtain a mixed-integer conic formulation for BFP, referred to as \textsf{1Term-Conic}:
% \[
% \begin{aligned}
%     \textsf{LEF}:\; \max \quad   &  \sum\limits_{i=1}^{m}\big(b_{i0}\rho^{i}+b_i^\intercal y^i\big)+c^\intercal x \notag \\
%     \text{s.t.} \quad & x \in \{0,1\}^n ,\ \bigl((\rho^i)_{i=1}^m,(y^i)_{i=1}^m,x\bigr)\in \mathcal{R}_{\textsf{LEF}}
% \end{aligned}
% \]
\begin{subequations}\label{eq:1termconic}
\begin{align}
 \max \quad   &  \sum\limits_{i=1}^{m}\big(b_{i0}\rho^{i}+b_i^\intercal y^i\big)+c^\intercal x \notag \\
	\text{s.t.} \quad & x \in \X \subseteq\{0,1\}^n,\ \eqref{eq:LFP-MC}  \notag \\
 & \max \{ 0,\ y^i_{j} + y^i_k - \rho^i\} \leq W^i_{jk} \leq \min  \{ y^i_j,\  y^i_k \}     && \text{for } i \in [m], j,k \in [n] \label{eq:1termconic-2}  \\ 
&a_{i0}\rho^i +  a_i^\intercal y^i  =1,\ \rho^i\geq 0 && \text{for }i\in[m]  \label{eq:1termconic-3} \\
& x_j = a_{i0}y^i_j + \sum_{k}a_{ik} W^{i}_{jk}   &&\text{for }i\in[m], j \in [n]  \label{eq:1termconic-4} \\
& y^i_j\le \rho^i && \text{for }i\in [m]  \label{eq:1termconic-5}\\
&W^i_{jj} = y^i_j  && \text{for }j\in[n]  \label{eq:1termconic-6}\\
& \rho^i(a_{i0}+a_i^\intercal x)\geq 1  && \text{for } i\in[m]  \label{eq:1termconic-7}.
	% \text{s.t.} \quad 	 
 % 	&\x \in \{0,1\}^n,\ \eqref{eq:CH-1},\ \eqref{eq:CH-2},\ \eqref{eq:CH-3}    \\
 % &(u_{i0} + \sum_{j \in N}u_{ij}x_j ) y_{i0} \geq 1 \qquad \text{for } i \in M^+ \\
\end{align} 
\end{subequations}
Constraints~\eqref{eq:1termconic-2} - \eqref{eq:1termconic-6} are derived by specializing $\mathcal{R}_{\QP}$ to unconstrained BFP. Namely, for unconstrained BFP, $\mathcal{R}_{\QP}$ is equivalent to constructing the McCormick envelopes of $x_jx_k$ for $j,k\in [n]$ before homogenizing these inequalities using $\rho^i$. This was referred to as the \eqref{eq:1term} relaxation in Section~\ref{section:BQP}. Although \eqref{eq:LFP-MC} is redundant in this formulation, we include these constraints because we will relax~\eqref{eq:1termconic-2} and~\eqref{eq:1termconic-4} partially to allow the formulation to scale to larger problem sizes. The last constraint~\eqref{eq:1termconic-7} is the conic constraint~\eqref{eq:Conic_b}. We remark that the formulation \textsf{1Term-Conic} is related to the formulation \textsf{CH} proposed in~\cite{chen2024integer} for assortment planning in the context of quick-commerce. There are two types of constraints in formulation \textsf{CH} of~\cite{chen2024integer}, with one being the constraint~\eqref{eq:Conic_b} and another that relaxes certain telescoping sums. 

We will compare our formulation \textsf{1Term-Conic} with two formulations, \textsf{LEF} and \textsf{CEF}, studied in~\cite{mehmanchi2019fractional}. The formulation \textsf{LEF} and \textsf{CEF} are built on the relaxation $\mathcal{R}_{\textsf{LEF}}$ and $\mathcal{R}_{\textsf{CEF}}$, respectively. More specifically, the formulation \textsf{LEF} (resp. \textsf{CEF}) maximize the objective function of \textsf{1Term-Conic} over $\mathcal{R}_{\textsf{LEF}}$ (resp. $\mathcal{R}_{\textsf{CEF}}$) and the binary constraint $x \in \{0,1\}^n$.

% \begin{equation*}\label{eq:one-term}
% \mathcal{R}_{\text{1-term}}=\left\{\bigl((\rho^i)_{i=1}^m,(y^i)_{i=1}^m,x\bigr) \;\middle|\;
%     \begin{alignedat}{3} 
%     &&& w^i_{ij} \geq \max\{0,\ y^i_{j} + y^i_k - \rho^i\},\ w^i_{jk} \leq \min\{ y^i_j, y^i_k\}  \\
%     &&& a_{i0}\rho^i +  a_i^\intercal y^i  =1,\ \rho^i\geq 0,\;  i\in[m] \\
% %    &&&w^i_{jk} \ge 0,\ w^i_{jk} \ge y^i_j+y^i_k - \rho^i, 
% %    && i\in[m], ~1\leq j<k\leq n\\     &&&w^i_{jk}\leq y^i_j,\ w^i_{jk}\leq y^i_k, && i\in[m], ~1\leq j<k\leq n,\\
%     &&&x_j = a_{i0}y^i_j + \sum_{k}a_{ik} w^{i}_{jk},   i\in[m], ~j \in [n]\\
%     &&& y^i_j\le \rho^i,\; i\in [m]\\
%     &&&w^i_{jj} = y^i_j,\; j\in[n]
%     \end{alignedat}\right\}.
% \end{equation*}

In our first experiment, we compare the three formulations, $\textsf{LEF}$, $\textsf{CEF}$ and $\textsf{1Term-Conic}$, in terms of their continuous relaxations on instances of unconstrained BFP, generated using the procedure described in~\cite{mehmanchi2019fractional}. Specifically, for each $(n,m) \in \{(30,3),(50,5), (70,7)\}$, we randomly generated $30$ instances of unconstrained BFP as follows. The coefficient $a_{ij}$ is sampled from a $U[0,20]$ distribution expect for $a_{i0}$ which is sampled from $U[1,20]$, the coefficient $b_{ij}$ is sampled from $U[-20,0]$, and the coefficient $c_i = 0$ for all $i \in [n]$. We solve the continuous relaxation of each of the three formulations for each instance, and then compute the closed $\textsf{LEF}$ gap as
\[
100\% \times  \frac{v_{\textsf{LEF}} - v_{\textsf{model}}}{v_{\textsf{LEF}} - \hat{v}} \quad \text{ for } \textsf{model} \in \{\textsf{LEF}, \textsf{CEF}, \textsf{1Term-Conic}\},
\]
where $\hat{v}$ is the objective value of the best integer solution obtained from $\textsf{LEF}$ after $600$ seconds of computation, and $v_{\textsf{model}}$ is the optimal value of the continuous relaxation of the formulation \textsf{model}. Table~\ref{tab:uniform} shows the continuous relaxation of \textsf{1Term-Conic}  closes up to $60$\% of the relaxation gap while that of \textsf{CEF} closes up to $30\%$. 

\begin{table}[ht]
\setlength{\tabcolsep}{7pt} % Default value: 6pt
\renewcommand{\arraystretch}{1} % Default value: 1
\begin{center}
\begin{tabular}{lllllr}
\hline
% \multicolumn{2}{c}{Item} \\
% \cline{1-2}
\multirow{2}{*}{$(n,m)$} & \multirow{2}{*}{model}    & \multicolumn{4}{c}{closed \textsf{LEF} gap (\%)}   \\
\cline{3-6}
& & avg. & min. & max. & std. \\
\hline
\multirow{2}{*}{$(30,3)$}  & \textsf{1Term-Conic}      & 64.2    &  44.7  &  100 & 14.7   \\
% & \textsf{1-Term}      & 59.7    &  37.3  &  100 & 16.3   \\
&\textsf{CEF}  &  33.1    & 19.7    & 60.1     &  8.9  \\
\cline {2-6}
\multirow{2}{*}{$(50,5)$}  & \textsf{1Term-Conic}      & 41.8    &  33.5  &  50.7 & 4.2  \\
 % & \textsf{1-Term}      & 37.9    &  30.3  &  46.3 & 3.9  \\
&\textsf{CEF}&  20.3   &  15.6 & 26.9    &  2.9  \\
\cline {2-6}
\multirow{2}{*}{$(70,7)$}  & \textsf{1Term-Conic}      & 33.5    & 26.4 &  40.7 & 3.6  \\
 % & \textsf{1-Term}      & 31.2    & 25.6 &  38.5& 3.4  \\
&\textsf{CEF}&  15.4   &  11.5 & 18.7 &  1.8 \\
\hline
% \multirow{3}{*}{$(100,10)$}  & \textsf{1-Term-c}      & 2.4   & 0.4   & 14.0 & 4.2 & 2 & 10 & 0.0\% \\
% &\textsf{LEF}  & 1125.2    &   26.5  & 3600    & 1432.6 & 5189583 & 8 & 1.5\% \\
% &\textsf{CEF}& 9.6      & 2.0 &  17.0   & 4.4 &14.6 & 10&0.0\% \\
% \cline {2-9}
% \multirow{3}{*}{$(200,20)$}  & \textsf{1-Term-c}      & 5.8  & 2.9  & 13.0 & 3.1 & 5.4 & 10 & 0.0\% \\
% &\textsf{LEF}  & 3600    &   3600  & 3600    & 0 & 1749502 & 0 & 4.0\% \\
% &\textsf{CEF}& 50.0      & 26.1 &  137.4   & 36.1 &263.4 & 10 &0\% \\
% \cline {2-9}
% \multirow{3}{*}{$(400,40)$} & \textsf{1-Term-c}      & 65.2    & 38.3   & 118.0 &  23.1 & 1 & 10 & 0.0\% \\
% &\textsf{LEF} &  3600    & 3600    & 3600    & 0 & 969919 & 0 &5.7\%  \\
% &\textsf{CEF}  &   3600  &  3600    &  3600   &  0 & 3 & 0 & $\dagger$ \\
% \cline {2-9}
% \multirow{3}{*}{$(800,80)$} & \textsf{1-Term-c}  &626.4    & 433.9  &  805.3 & 115.9   &  1  &  10  &  0.0\%  \\
% &\textsf{LEF} &   3600    & 3600    & 3600    & 0 & 1 & 0 &5.9\% \\
% &\textsf{CEF}  &   3600  &  3600    &  3600   &  0 & 1 & 0 & $\dagger$ \\
% \hline
\end{tabular}
\end{center}
\caption{Comparison of continuous relaxations of \textsf{LEF}, \textsf{CEF} and \textsf{1Term-Conic}}
    \label{tab:uniform}
\end{table}

%For each such configuration,  the feasible region is given by $\X = \bigl\{x \in \{0,1\}^n \bigm| \sum_{i \in n}x_{i} \leq \kappa \bigr\}$, where we vary $\kappa \in \{10\% \cdot n,\ 20\% \cdot n,\ 100\% \cdot n\}$. 

% We perform computations with two datasets. The first dataset is generated as follows. \textcolor{red}{Describe the generation procedure.} Table~\ref{tab:uniform} shows the continuous relaxation of (1-Term) closes up to $50$\% relaxation gap. Moreover, we remark that it takes \textcolor{red}{so much time} to achieve the same dual (upper) bound as the linear relaxation of (1-Term), when using SCIP \cite{} and Gurobi~\cite{} are used to solve~\ref{eq:computation:ubfp} with formulaion~\textcolor{red}{which formulation?} 

Our second numerical experiment compares the three formulations by solving assortment optimization problems to global optimality. For~\textsf{1Term-Conic}, we do not introduce $W^i_{jk}$ variables. Instead, we begin with~\textsf{1Term-Conic} without Constraints~\eqref{eq:1termconic-2} and~\eqref{eq:1termconic-4}. Then, for the root relaxation solution, we use the active bound from Constraints~\eqref{eq:1termconic-2} to replace $W^i_{jk}$ in~\eqref{eq:1termconic-4} and introduce the resulting constraint. The formulation so obtained will be referred to as~\textsf{1Term-Conic-R}. All three formulations are solved using~\textsc{Gurobi} with a time limit of 1 hour. The problems are generated using the procedure described in \cite{mehmanchi2019fractional}. Specifically, for $(n,m) \in \{(50,5),(100,10),(200,20),(400,40),(800,80)\}$, we randomly generate 10 instances as follows: $a_{ij} \sim U[0,1]$ and $b_{ij} = a_{ij}r_{i}$ with $r_{i} \sim U[1,3]$ for all $i \in [m]$ and $j \in [n]$, $a_{i0} = 0.1n$ and $b_{i0}=0$ for all $i \in [m] $, $c_j = 0$ for all $j \in [n]$, and $\X = \bigl\{x \in \{0,1\}^n \bigm| \sum_{i \in [n]}x_i\leq \kappa \bigr\}$ with $\kappa = 20\% \cdot n$. Table~\ref{tab:assortment} shows that \textsf{1Term-Conic-R} outperforms the other formulations. More specifically, for most problems, \textsf{1Term-Conic-R} a few branch-and-bound nodes to discover the optimal solution while \textsf{LEF} and \textsf{CEF} require larger branch-and-bound trees. In particular, \textsf{LEF} is unable to solve problems with $(n,m)\in \{(200,20),(400,40),(800,80)\}$ and even a few problems of smaller size. The formulation \textsf{CEF} performs significantly better than \textsf{LEF}, but cannot solve instances with $(n,m)\in \{(400,40),(800,80)\}$. For smaller problems, the time taken by \textsf{LEF} and \textsf{CEF} is often an order of magnitude or more higher than that by \textsf{1Term-Conic-R}.

\begin{table}[ht]
\setlength{\tabcolsep}{5.8pt} % Default value: 6pt
\renewcommand{\arraystretch}{0.9} % Default value: 1
\begin{center}
\begin{tabular}{llllllllr}
\hline
% \multicolumn{2}{c}{Item} \\
% \cline{1-2}
\multirow{2}{*}{$(n,m)$} & \multirow{2}{*}{formulation}    & \multicolumn{4}{c}{time (seconds)}  & \multirow{2}{*}{nodes}  &  \multirow{2}{*}{solved} & \multirow{2}{*}{rgap} \\
\cline{3-6}
& & avg. & min. & max. & std. \\
\hline
\multirow{3}{*}{$(50,5)$}  & \textsf{1Term-Conic-R}      & 0.10    &  0.07  &  0.15 & 0.03  &  1 & 10  & 0.0\% \\
&\textsf{LEF}  &  0.56    & 0.30    & 1.58     &  0.40 & 2173.4  & 10  & 0.0\% \\
&\textsf{CEF}&   0.80     &  0.20 & 2.33    &  0.60 & 1 & 10  &0.0\%   \\
\cline {2-9}
\multirow{3}{*}{$(100,10)$}  & \textsf{1Term-Conic-R}      & 2.4   & 0.4   & 14.0 & 4.2 & 2 & 10 & 0.0\% \\
&\textsf{LEF}  & 1125.2    &   26.5  & 3600    & 1432.6 & 5189583 & 8 & 1.5\% \\
&\textsf{CEF}& 9.6      & 2.0 &  17.0   & 4.4 &14.6 & 10&0.0\% \\
\cline {2-9}
\multirow{3}{*}{$(200,20)$}  & \textsf{1Term-Conic-R}      & 5.8  & 2.9  & 13.0 & 3.1 & 5.4 & 10 & 0.0\% \\
&\textsf{LEF}  & 3600    &   3600  & 3600    & 0 & 1749502 & 0 & 4.0\% \\
&\textsf{CEF}& 50.0      & 26.1 &  137.4   & 36.1 &263.4 & 10 &0\% \\
\cline {2-9}
\multirow{3}{*}{$(400,40)$} & \textsf{1Term-Conic-R}      & 65.2    & 38.3   & 118.0 &  23.1 & 1 & 10 & 0.0\% \\
&\textsf{LEF} &  3600    & 3600    & 3600    & 0 & 969919 & 0 &5.7\%  \\
&\textsf{CEF}  &   3600  &  3600    &  3600   &  0 & 3 & 0 & $\dagger$ \\
\cline {2-9}
\multirow{3}{*}{$(800,80)$} & \textsf{1Term-Conic-R}  &626.4    & 433.9  &  805.3 & 115.9   &  1  &  10  &  0.0\%  \\
&\textsf{LEF} &   3600    & 3600    & 3600    & 0 & 325.1 & 0 &5.9\% \\
&\textsf{CEF}  &   3600  &  3600    &  3600   &  0 & 1 & 0 & $\dagger$ \\
\hline
\end{tabular}
\end{center}
\caption{Comparison of \textsf{LEF}, \textsf{CEF} and \textsf{1Term-Conic-R} on instances of assortment planning. Here, $(n,m)$ specify problem dimensions. Each row consists of data accumulated over $10$ instances. The ``nodes'' column reports the number of branch \& bound nodes required to solve the problem, ``solved'' is the number of instances solved to optimality, and ``rgap'' is the average of relative gaps, $\frac{z_{\text{best}}-z_{\text{opt}}}{z_{\text{LEF}}-z_{\text{opt}}}$, remaining at the end of the solution process. The ``$\dagger$'' symbol denotes that \textsc{Gurobi} was unable to fully process the root node of the branch-and-bound tree within the time limit of $1$ hour.}\label{tab:assortment}
\end{table}

\subsection{Univariate fractional terms in chemical process design }
Consider 
\begin{equation}\label{eq:chemE}
    \begin{aligned}
    &\max  &&  -ax - b^\intercal y + c^\intercal z \\
    &\st && z_i = \frac{y_i}{x-r_i} \quad \text{ for } i \in [m] \\
    &&& x^L\leq  x \leq x^U ,\ y_i^L \leq y_i \leq y_i^U  \quad \text{ for } i \in [m],
    \end{aligned}\tag{PD}
\end{equation}
where $a \in \R^m$, $b \in \R$, $c \in \R^m$, $r \in \R^m$, and $x^L$, $x^U$, $y_i^L$ and $y_i^U$ are constants. The motivation for this problem comes from chemical process design. In particular, to perform distillation, a column requires vapor flow, and generating this vapor flow requires energy. As such, research has focused on identifying the configuration for a given separation task with the minimum vapor flow. Under standard technical assumptions, the minimum vapor flow can be computed using Underwood's method, see~\cite{gooty2022advances} for details. Similarly, \cite{gooty2023exergy} shows that the liquid phase mole fraction in the condenser/reboiler that is in equilibrium with the vapor can be modeled using a univeriate fraction such as $z_i$ in \eqref{eq:chemE}.

A typical approach to relax the feasible region of~\eqref{eq:chemE} is to reformulate $z_i = \frac{y_i}{x-r_i}$ as $z_i(x-r_i) = y_i$, and then relax the product using McCormick inequalities. To derive a lower (resp. an upper ) bound  on $z_i$, we minimize (resp. maximize) $\frac{y_i}{x-r_i}$ over $x \in [x^L, x^U ]$ and $y_i \in  [y_i^L, y_i^U ]$. We refer to this relaxation as \textsf{Uni-MC}. 
% use McCormick inequalities to relax $y_i\times u_i$, where $u_i=\frac{1}{x-r_i}$ is relaxed by exploiting that $\frac{1}{x-r_i}$ is convex if $r_i<0$ and concave if $r_i > 1$. we refer to this relaxation as (UNI-MC). 
Instead, we relax~\eqref{eq:chemE} using the techniques developed in Section~\ref{section:univariatefractional}. Let
\[
\G = \Bigl\{(1,\nu, x ) \Bigm| 0 \leq  x \leq 1,\ \nu_i =\frac{1}{x-r_i} \text{ for } i \in [m]  \Bigr\}.
\]
We derive the convex hull of $\G$ using Proposition~\ref{prop:momenthull}, and relax $z_i=y_i\nu_i$ using McCormick inequalities. Then, we obtain the following relaxation, referred to as \textsf{Uni-MH}, 
\begin{equation*}
 \begin{aligned}
  	&(1,\nu, x) \in \conv(\G) \\
 &z_i \geq \max \Bigl\{ \frac{1}{1-r_i} y_i + \nu_i - \frac{1}{1-r_i},\ -\frac{1}{r_i} y_i + 2\nu_i + 2\frac{1}{r_i}  \Bigr\} && \text{ for } i \in [m] \\
% & z_i \geq -\frac{1}{r_i} y_i + 2\nu_i + 2\frac{1}{r_i} && \text{ for } i \in [m]\\
   & z_i \leq \min \Bigl\{\frac{1}{1-r_i} y_i + 2\nu_i - 2\frac{1}{1-r_i},\ -\frac{1}{r_i} y_i + \nu_i +\frac{1}{r_i} \Bigr\} && \text{ for } i \in [m]. \\
%   & z_i \leq -\frac{1}{r_i} y_i + \nu_i +\frac{1}{r_i} && \text{ for } i \in [m]\\
% 	& 0 \leq x \leq 1,\  1 \leq  y_i \leq 2 && \text{ for } i \in [m]
 \end{aligned}
\end{equation*}

We numerically test our relaxation \textsf{Uni-MH} on synthetically generated problem instances. For each $m \in \{5,9,13\}$, we fix $x^L = 0$, $x^U = 1$, $y_i^L = 1$ and $y_i^U=2$, and randomly generate $100$ instances of~\eqref{eq:chemE} as follows: $c_i \sim U[-1,1]$, $r_i \sim U[-1,-0.1]$ for $i \in  [\frac{m+1}{2}]$ and $ r_i \sim U[1.1,2]$ otherwise,  and $(a,b)$ is the gradient of $\sum_{i \in [m]} c_i\frac{y_i}{x-r_i}$ at a point uniformly sampled from $[0,1] \times [1,2]^m$. We solve \textsf{Uni-MC} and \textsf{Uni-MH} for each instance, and then compute the remaining gap as 
\begin{equation*}
    \text{Relative remaining gap} = 100\% \times \frac{v_{\textsf{Uni-MH}}-v_{\textsf{SCIP}}}{v_{\textsf{Uni-MC}}-v_{\textsf{SCIP}}},
\end{equation*}
where $v_{\textsf{SCIP}}$ is the optimal objective value of~\eqref{eq:chemE} returned by \textsf{SCIP}~\cite{bestuzheva2023global}, and  $v_{\textsf{Uni-MH}}$ (resp.  $v_{\textsf{Uni-MC}}$) is the optimal objective value of \textsf{Uni-MH} (resp. \textsf{Uni-MC}). 
In Figure~\ref{fig:relgapchemE}, we show the empirical distribution function of the relative remaining gap for the $100$ instances. Observe that for $50\%$ of the instances, 
the relative remaining gap is below 5\%,  and for 80\% of the instances, the relative remaining gap is below 25\%.
\begin{figure}[ht!]
\begin{center}
	\begin{tikzpicture}[scale=0.6]
\pgfplotsset{
    tick label style={font=\small},
    label style={font=\small},
    legend style={font=\small},
    every axis/.append style={
    very thick,
    tick style={semithick}}
    }
\begin{axis}[
	xmin=-0.01,   xmax=1.03,
	ymin=0.3,   ymax=1.03,
	xlabel={Relative remaining gap},
	ylabel={Fraction of instances},
	xtick={0,0.2,0.4,0.6,0.8,1},
	ytick={0,0.2,0.4,0.6,0.8,1}, 
	legend style={draw=none,
        at={(0.97,0.4)}},
    grid=major,
%   very thick,
    cycle list name=color list
        ]
\pgfplotstableread{cdf.txt}\mydata;
\addplot table [x=poi,y=size5] {\mydata};
\addplot table [x=poi,y=size9] {\mydata};
\addplot table [x=poi,y=size13] {\mydata};
%\addplot table [x=poi,y=cr] {\mydata};
%\addplot table [x=poi,y=mip] {\mydata};
%%\addplot table [x=poi,y=mip1] {\mydata};
%\addplot table [x=poi,y=scip] {\mydata};
%\addplot table [x=poi,y=mc] {\mydata};
 \legend{$m = 5$,$m = 9$,$m = 13$}
%\addplot[smooth, point meta=explicit] coordinates {
%	(-2,3) (-1.5,2) (-0.3,-0.2) 
%	(1,1.2) (2,2) (3,5)};
\end{axis}
\end{tikzpicture}
\caption{Cumulative distribution plot of the relative remaining gap for instances of \eqref{eq:chemE}.}\label{fig:relgapchemE}
\end{center}
\end{figure}
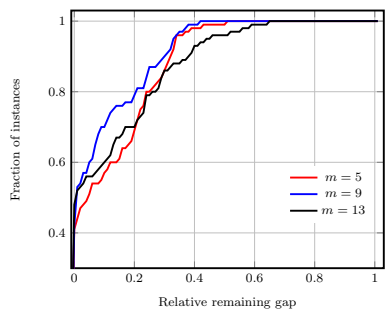

\section{Conclusions}\label{section:conclusion}
\revision{In this paper, we show that the study of fractional programming problems can leverage results in polynomial optimization literature. This result hinges upon a projective transformation that involves homogenization and renormalization of the convex hull of polynomial functions. We exemplify the usage of this result by connecting with existing studies on a number of convexification techniques. Numerical results indicate that the reformulations and relaxations developed in this paper outperform existing literature in solving assortment optimization problems and close significant gap on a set appearing in chemical process design literature. }

\revision{It is necessary to investigate the practical applications of our convexification techniques. This includes using our formulations on non-synthetic instances, such as assortment planning and distillation configuration design, and incorporating recent developments towards making copositive and semidefinite programming based relaxations tractable from a computational standpoint. 
It is also interesting to study whether the projective transformation, introduced in Theorem~\ref{thm:mainthem}, can be applied in designing approximation algorithms for fractional programming problems. 
}

% Using this result, we develop a variety of results. For $0\mathord{-}1$ fractional programming problems, we show that valid inequalities for the boolean quadric polytope can be used to develop relaxations and reformulations. We showed that the multilinear polytope can similarly be used to develop relaxations for optimizing a ratio of multilinear functions of binary variables. Then, in Section~\ref{section:copositive} we showed that semidefinite and copositive programming can be used to develop  reformulations for maximizing a ratio of quadratic functions over an ellipsoid and a polytope respectively. Our results can also accomodate the case where some of the variables are binary. In Section~\ref{section:univariatefractional} we considered two sets involving various fractions of a single variable. We showed that the moment-hull can be used to develop the convex hull of these sets. In Section~\ref{section:general-linear-frac}, we developed a hierarchy of relaxations that converges to the convex hull of a sum of fractions of binary variables. In Section~\ref{section:computation}, we showed that the reformulations and relaxations developed in this paper outperform existing literature in solving assortment optimization problems and close significant gap on a set appearing in chemical process design literature. Finally, we proposed a relaxation for constrained mixed-binary fractional programming and demonstrated theoretically that our relaxation is tighter than existing relaxations in fractional and mixed-integer nonlinear programming literature.

\appendix
\section{Missing Proofs and Discussions}\label{app:proofs}

\subsection{Proof of Theorem~\ref{thm:fractionalmain}}\label{app:closure}
\revision{
Let $\text{CG}$ denote the right hand side set of \eqref{eq:FprimeClConv}. 
Since $0\cl\conv(\F)$ represents the recession cone of $\cl\conv(\F)$, by Theorem~9.6 in \cite{rockafellar2015convex}, it follows that $\text{CG}$ is closed. By using Theorem~\ref{thm:mainthem} and lifts of $\mathcal{F}$ and $\mathcal{G}$ defined as in~\eqref{eq:lift}, we obtain that  $\conv(\G)\subseteq \text{CG}$. Therefore, it follows that $\cl\conv(\G)\subseteq \text{CG}$. Now, we show the reverse inclusion. Let $g\in \text{CG}$. First, assume $\rho > 0$. Let $f=\frac{g}{\rho}$ and consider a sequence $(f^k)_{k} \subseteq \conv(\F)$ such that $f^k\rightarrow f$. Since $\alpha^\intercal g = 1$, it follows that $\alpha^\intercal f = \frac{1}{\rho} > 0$. Then, there exists a $K$ such that for all $k\ge K$, $\alpha^\intercal f^k > \frac{1}{2\rho}$. We limit the sequence to $(f^k)_{k\ge K}$.
%By passing to a subsequence, if necessary, we can assume that $\|f^k-f\|\le \delta$ for all $k$. 
%Define $g^k = \frac{f^k}{\alpha^\intercal f^k}$ and observe that $\|g_k\|\le \frac{\|f\|+\delta}{\epsilon}$. 
Observe that 
\[
0\le \Bigl\|\frac{f^k}{\alpha^\intercal f^k} - \frac{f}{\alpha^\intercal f}\Bigr\| \le \Bigl\|\frac{f^k}{\alpha^\intercal f^k} - \frac{f}{\alpha^\intercal f^k}\Bigr\| + \Bigl\|\frac{f}{\alpha^\intercal f^k} - \frac{f}{\alpha^\intercal f}\Bigr\|\rightarrow 0,
\]
where the convergence holds since $\|f^k-f\|\rightarrow 0$, $\alpha^\intercal f^k > \frac{1}{2\rho}$, and  $\alpha^\intercal f = \frac{1}{\rho} > 0$. Thus, $g^k := \frac{f^k}{\alpha^\intercal f^k}$ converges to $g := \frac{f}{\alpha^\intercal f}$. Moreover, $g^k \in \conv(\mathcal{G})$, implying that $g \in \cl \conv(\mathcal{G})$. To see that $g^k \in \conv(\mathcal{G})$, we invoke Theorem~\ref{thm:mainthem} and lifts defined as in~\eqref{eq:lift} as follows. As $f^k \in \conv(\mathcal{F})$, its lift $\bigl(\alpha^\intercal f^{k}, f^{k}\bigr)$ belongs to $\conv(\mathcal{S})$, where   $\mathcal{S}$ is defined as in~\eqref{eq:lift}. It follows from Theorem~\ref{thm:mainthem} that $\bigl(\frac{1}{\alpha^\intercal f^k}, \frac{f^k}{\alpha^\intercal f^k}\bigr)\in \conv\bigl(\Phi(\S)\bigr)$, and we have $g^k= \frac{f^k}{\alpha^\intercal f^k}\in \conv(\G)$ by projecting out the first coordinate. To complete the proof, we treat the case where $\rho = 0$. In this case, $g$ belongs to the recession cone of $\cl\conv(\F)$ and $\alpha^\intercal g = 1$. Since $\F\ne\emptyset$, we may choose $f\in \ri\bigl(\conv(\F)\bigr)$.
%and consequently, by Theorem~6.6 in \cite{rockafellar2015convex}, $(1,f)\in \ri(\conv(\S))$. 
Consequently, $f + \beta g\in \conv(\F)$ for all $\beta > 0$ (see Corollary 8.3.1 in \cite{rockafellar2015convex}). As above,  $\frac{(f+\beta g)}{\alpha^\intercal f + \beta}\in \conv(\G)$, where we have used that $\alpha^\intercal g = 1$. 
Taking $\beta\rightarrow \infty$, we have that $g\in \cl\conv(\G)$.
\Halmos
}

% \subsection{Proof of Corollary~\ref{cor:frac-func}}
% By setting $\alpha = (a_0, a^\intercal, 0)\in \R\times\R^n\times\R^m$, the second statement follows from Theorem~\ref{thm:fractionalmain}. If $f(x) = x$ then, 
%  by  introducing $z_{ij} = x_ix_j$ and exploiting $x_i^2 = x_i$, we obtain $  \conv(\F) = \bigl\{(1,x, f) \bigm| (x,z)  \in \QP,\ f = \ell_i(x,z) \text{ for } i \in [n] \bigr\}$.
%   % \[
%   % \conv(\F) = \bigl\{(1,x, f) \bigm| (x,z)  \in \QP,\ f = \ell_i(x,z) \text{ for } i \in [n] \bigr\}. 
%   % \]
% This completes the proof. \Halmos

\subsection{Proof of Corollary~\ref{cor:two-term}}
Let $I = \{i,j\}$ where $i,j \in [n]$ with $i\neq j$. Due to~\eqref{eq:lfp}, it suffices to describe $\QP_{G(I)}$.  For $k \in V \setminus \{i,j\}$, let $C_{k}$ denote the triangle subgraph of $G(I)$ whose vertex set is $V$ and edge set is $\bigl\{(i,j),(j,k),(i,k)\bigr\}$, and define $\F_{C_k} = \bigl\{(x_i,x_j,x_k, z_{ij},z_{jk},z_{ik}) \in  \{0,1\}^6 \bigm| z_{ij} = x_{i}x_j,\ z_{jk} = x_{j}x_k,\ z_{ik} = x_{i}x_k \bigr\}$. Then, it follows that $\QP_{G(I)}$ is the convex hull of $\bigl\{(x,z)\bigm| (x_i,x_j,x_k, z_{ij},z_{jk},z_{ik}) \in \F_{C_k} \forall k \in V\setminus \{i,j\} \bigr\}$. It follows from~\cite{padberg1989boolean} that $\conv(\F_{C_k})$ can be described using McCormick and triangle inequalities. Moreover, for each $k \neq i,j$, the projection of $\F_{C_k}$ onto the space of $(x_i,x_j,z_{ij})$ variables is $\{(x_i,x_j,x_ix_j)\mid x_i \in \{0,1\},\ x_j \in \{0,1\}\}$, a set of  affinely independent points. Thus, the proof is complete by invoking Lemma~\ref{lemma:common_simplex} to obtain $\QP_{G(I)}$ as the intersection of $\conv(\F_{C_k})$. \Halmos

\subsection{Proof of Proposition~\ref{prop:rlt-lfp}}
Let $I \subseteq V$ with $\vert I \vert = k$. Due to~\eqref{eq:lfp}, it suffices to show that  $\RLT_{k}$  yields a description of $\QP_{G(I)}$.  
Consider a set defined by all multilinear monomials in variable $(x_i)_{i \in I}$, that is, $\mathcal{M}_0 = \bigl\{(x,z^0) \bigm|  x \in \{0,1\}^n,\ z^0_e = \prod_{t \in e}x_t \; \forall e \subseteq I \text{ s.t. } \vert e \vert \geq 2 \bigr\}$.
% \[
% \mathcal{M}_0 = \biggl\{(x,z^0) \biggm|  x \in \{0,1\}^n,\ z^0_e = \prod_{t \in e}x_t \; \forall e \subseteq I \text{ s.t. } \vert e \vert \geq 2 \biggr\}.
% \]
% \[
% \mathcal{M}_0 = \Bigl\{(x,z^0) \Bigm| x \in \{0,1\}^n,\ z^0_e = \prod_{t \in e}x_t \; \forall e \subseteq I \text{ s.t. } \vert T \vert \geq 2 \Bigr\}. 
% \]
For each $j \notin I$, consider a set defined by all multilinear monomials in variable $(x_i)_{i\in I \cup \{j\}}$,
\[
\mathcal{M}_{j}:=\biggl\{(x,z^0,z^j) \biggm| (x,z^0) \in \mathcal{M}_0, z^j_e =  \prod_{t\in e}x_t \forall  e\subseteq I \cup \{j\} \text{ s.t. } \vert e \vert \geq 2 \text{ and }  j \in e \biggr\}.
\]
Let $\mathcal{M} := \bigl\{(x,z^0,(z^j)_{j \notin I}) \bigm| (x,z^0,z^j) \in \mathcal{M}_{j} \; \forall j \notin I \bigr\}$. Clearly, the convex hull of $\mathcal{M}$ is an extended formulation of $\QP_{G(I)}$. The proof is complete if we argue that $\RLT_{k}$ describes the convex hull of $\mathcal{M}$. It follows readily from Lemma~\ref{lemma:rlt} that $\RLT_{k}$ describes the convex hull of $\mathcal{M}_j$ for every $j \notin I$. Now, observe that for each $j \notin I$, $\proj_{(x,z^0)}(\mathcal{M}_{j}) = \mathcal{M}_0$, and $\mathcal{M}_0$ is a set of affinely independent points. Thus, we can invoke the Lemma~\ref{lemma:common_simplex} to obtain $\conv(\mathcal{M}) =\bigl\{ (x,z^0,(z^j)_{j \notin I} ) \bigm| (x,z^0,z^j)  \in \conv(\mathcal{M}_{j}) \; \forall j \notin I \bigr\}$.  \Halmos

\subsection{Proof of Proposition~\ref{prop:qfp-bqp}}
Let $f_G(x) = (1,x,z)$, where $z_{ij}=x_ix_j$ for $(i,j)\in E$, and consider $\F_G=\{f_G(x) \mid x \in [0,1]^V\}$.  Then, $\conv(\F_G) = \conv\bigl\{ f_G(x) \bigm| x \in \{0,1\}^V \bigr\} = \bigl\{(1,x,z) \bigm| (x,z) \in \QP_G \bigr\}$,
% \[
% \conv(\F_G) = \conv\bigl\{ f_G(x) \bigm| x \in \{0,1\}^V \bigr\} = \bigl\{(1,x,z) \bigm| (x,z) \in \QP_G \bigr\},
% \]
where the first equality follows from Theorem~1 in \cite{burer2009nonconvex} and the second equality is by the definition of $\QP_G$. Moreover, by \eqref{eq:defqG} and boundedness of $\G_G$, Theorem~\ref{thm:fractionalmain} yields the expression of $\conv(\G_G)$ and $\QP_G$. \Halmos

\subsection{Proof of Corollary~\ref{cor:poly-case}}
    If $G$ is series-parallel graph, by Theorem 10 of~\cite{padberg1989boolean}, the $\rho\BQP_G$ is described using the odd-cycle inequalities, and thus, by Proposition~\ref{prop:qfp-bqp}, the convex hull of $\G_G$ is given by the odd-cycle relaxation. Moreover, by~\cite{barahona1986cut}, the separation problem of $\BQP_G$ is polynomial time solvable and, thus, by Remark~\ref{rmk:eq-sep}, so is the separation problem of the odd cycle relaxation. Therefore, due to the polynomial equivalence between separation and optimization~\cite{grotschel2012geometric}, minimizing a linear function over the odd-cycle relaxation is polynomial-time solvable.  Hence,~(\ref{eq:bfp}) is polynomial-time solvable. \Halmos

\subsection{Proof of Proposition~\ref{prop:frac-quad}}
\begin{lemma}[Lemma 5.1 in \cite{balas1998disjunctive}]\label{lemma:facial-decom}
	Consider a subset $\S$ of $\R^n$, and a hyperplane $H:=\{x \mid \ell(x)  = \beta \}$ so that $\ell(x)  \leq \beta$  is a valid linear inequality of $\S$. Then, $\conv(\S \cap H) = \conv(\S) \cap H$. 
\end{lemma}
If $\mathcal{L} = \R^n$, that is $C = 0$ and $d=0$, the result follows from Theorem~\ref{thm:fractionalmain} by letting $f(x) = (1,x,xx^\intercal)$ and $\alpha = (1, a, A)$. Now, assume that $\mathcal{L}$ is not $\R^n$. By Theorem~\ref{thm:fractionalmain}, to obtain the convex hull of $\G$, it suffices to derive the convex hull of $\F'$, where $\F':=\{(x, X)\mid x \in \X \cap \mathcal{L},\ X = xx^\intercal \}$. We argue that $\F'$ can be obtained as the intersection of $\F$ and a supporting hyperplane.  Let $c^i$ be the $i^\text{th}$ row of $C$, and $d_i$ be the $i^\text{th}$ entry of $d$. Then, $x$ satisfies $c^i x = d_i$ for $i \in [m]$ if and only if $\Tr(Cxx^\intercal C^\intercal - Cxd^\intercal - dx^\intercal C^\intercal + dd^\intercal) = \sum_{i \in [m]}(c^ix - d_i)^2 = 0$. Thus, $\F'= \bigl\{(x,X) \in \F \bigm| \Tr(CX C^\intercal - Cxd^\intercal - dx^\intercal C^\intercal + dd^\intercal) =0 \bigr\}$.  Since $\Tr(CX C^\intercal - Cxd^\intercal - dx^\intercal C^\intercal + dd^\intercal) \geq 0$ is a valid linear inequality for $\F$, it follows from Lemma~\ref{lemma:facial-decom} that 
\[
\conv(\F') = \Bigl\{(x,X) \Bigm| \Tr\bigl( CXC^\intercal - Cxd^\intercal - d x^\intercal C^\intercal + dd^\intercal \bigr) = 0, (x,X) \in \conv(\F) \Bigr\}.
\]
This, together with Theorem \ref{thm:fractionalmain}, gives the description of $\conv(\G)$. \Halmos

\subsection{Proof of Corollary~\ref{cor:bivariate}}
Lift $\S$ to $	\{(t,x) \mid t = b_0\rho +  b^\intercal y + \langle B, Y \rangle,\ x_i = a_1Y_{i1} + a_2 Y_{i2} + y_i\;  i =1,2,\ (\rho,y,Y) \in \G\}$, where 
	\[
	\G = \biggl\{\frac{(1,x,xx^\intercal)}{a^\intercal x +a_0} \biggm| Cx\leq d \biggr\}.
	\]
Since $\S$ is obtained as a linear transformation of $\G$, it suffices to convexify $\G$. Let $\F=\bigl\{(x,X) \bigm| Cx\leq d,\ X = xx^\intercal \bigr\}$. By Theorem 2 in~\cite{burer2015gentle}, we obtain $\conv(\F) = \bigl\{(x,X) \bigm| Z(1,x,X) \succeq 0,\ CXC^\intercal -Cxd^\intercal - dx^\intercal C^\intercal + dd^\intercal \geq 0 \bigr\}$.
% \[
% \conv(\F) = \bigl\{(x,X) \bigm| Z(1,x,X) \succeq 0,\ CXC^\intercal -Cxd^\intercal - dx^\intercal C^\intercal + dd^\intercal \geq 0 \bigr\}.
% \]
Thus, using Proposition~\ref{prop:frac-quad}, we obtain a convex hull description of $\G$. \Halmos

\subsection{Proof of Lemma~\ref{lemma:change-basis}}
Let $\mathcal{P}^{n+1}$ be the vector space of univariate polynomials with degree at most $n+1$. The set of monomials $\{x^0,x^1, \ldots x^{n+1}\}$ is a basis of $\mathcal{P}^{n+1}$. Now, we show that $\{f_0(\cdot), f_1(\cdot),\ldots, f_{n+1}(\cdot)\}$ is also a basis of $\mathcal{P}^{n+1}$, and $T$ is an invertible map between two bases. For each $i \in \{0 \} \cup [n+1]$, the linear combination of $f_0(\cdot), \ldots, f_{n+1}(\cdot)$ with weights $\beta_{i0}, \ldots, \beta_{i(n+1)}$ interpolates the monomial $x^i$ over the given set of reals $\{r_0, \ldots, r_{n+1}\}$, that is, $        x^i = \beta_{i0}f_0(x) + \cdots + \beta_{i(n+1)}f_{n+1}(x) $ for $ x \in \{r_0, \ldots, r_{n+1}\}$.
        % \[
        % x^i = \beta_{i0}f_0(x) + \cdots + \beta_{i(n+1)}f_{n+1}(x) \quad \text{ for } x \in \{r_0, \ldots, r_{n+1}\}.
        % \]
This can be easily verified because $f_i(x)$ is $0$ when $x = r_j$ for $j\ne i$. Since $r_0, \ldots, r_{n+1}$ are $n+2$ distinct reals, the interpolating function is unique. In other words, $(x^0, x,\ldots, x^{n+1}) = T(f_0(x), \ldots, f_{n+1}(x))$ for all $x$. Thus, by a counting argument, we conclude that $\{f_i(\cdot)\}_{i=0}^{n+1}$ is also a basis and $T$ is invertible.  In particular, $\M_{n+1} = T\F$ and $\F = T^{-1}\M_{n+1}$, completing the proof. 
%         More specifically, $T^{-1}$ maps $\mu$ to $f$, where 
%         \[
% \begin{aligned}
%     f_0 &=  \sum_{k=0}^n\biggl(\sum_{ S \subseteq [n] , \vert S \vert =k}\prod_{j \notin S}(-r_j) \biggr)\mu_k \\
%     f_i & = \sum_{k=0}^{n-1} \biggl( \sum_{ S \subseteq [n] \setminus i,  \vert S \vert =k} \prod_{j\notin S}(-r_j) \biggr) \mu_k \\
%     f_{n+1} & = \sum_{k=0}^{n+1} \biggl( \sum_{ S \subseteq [n] \cup \{0\},  \vert S \vert =k} \prod_{j\notin S}(-r_j) \biggr) \mu_k 
% \end{aligned}
% \]
        \Halmos

\subsection{Proof of Proposition~\ref{prop:momenthull}}
First, we obtain the convex hull characterization of $\G$. Observe that 
	\[
	\begin{aligned}
	\conv(\G) &= \bigl\{ \nu  \bigm| \nu \in \rho\conv(\F),\ \nu_0=1,\ \rho \geq 0 \bigr\}
	 = \bigl\{\nu \bigm| \nu \in \rho\conv(T^{-1}\M_{n+1}),\ \nu_0=1,\ \rho \geq 0 \bigr\} \\
		& = \bigl\{\nu \bigm| T\nu \in \rho\conv(\M_{n+1}),\ \nu_0=1,\ \rho \geq 0 \bigr\} 
		 = \bigl\{\nu \bigm| T\nu \in \cone(\M_{n+1}),\ \nu_0=1 \bigr\}.
	\end{aligned}
	\]
	where the first equality follows from Theorem~\ref{thm:fractionalmain},  the second equality holds due to $\F = T^{-1}\mathcal{M}_{n+1}$ in Lemma~\ref{lemma:change-basis}, the third equality holds since  convexification commutes with linear transformation, and the last equality holds by the definition of $\cone(\M_{n+1})$. 
 
	Next, we derive the convex hull of $\M_{n+1}$ in terms of $\conv(\G)$. By partial fraction decomposition, there exists real numbers $\alpha_1, \ldots, \alpha_n$ such that 
\[
\frac{1}{f_0(x)} = \alpha_{1} \frac{1}{x-r_1} + \ldots + \alpha_{n} \frac{1}{x-r_n}. 
\]
This implies that $1 = \sum_{i \in [n]}\alpha_i f_i(x)$. As defined in the statement of Lemma~\ref{lemma:change-basis} and shown in its proof, $(\beta_{0i})_{i=0}^{n+1}$ is the unique weight vector such that $1=\sum_{i = 0}^{n+1} \beta_{0i}f_i(x)$. Thus, $\alpha_{i} = \beta_{0i}$ for $i \in [n]$ and $\beta_{00} = \beta_{0(n+1)} = 0$. We will use Theorem~\ref{thm:mainthem} to obtain the convex hull of $\F$. To do so, we derive the convex hull of $\Phi(\F)$ (see \eqref{eq:introducePhi}), which is as follows
\[
\begin{aligned}
\conv\bigl(\Phi(\F)\bigr) & = \conv \Bigl\{\Bigl(\frac{1}{f_0(x)}, \frac{1}{x-r_1}, \ldots, \frac{1}{x-r_n}, x-r_0 \Bigr) \Bigm| x \in \X \Bigr\} \\
&=\conv \Bigl\{ \Bigl(\sum_{i \in [n]} \alpha_i \nu_i, \nu_{1}, \ldots, \nu_{n+1} \Bigr) \Bigm| (1,\nu_1, \ldots, \nu_{n+1}) \in \G \Bigr\}    \\
&= \Bigl\{ \Bigl(\sum_{i \in [n]}\alpha_i \nu_i, \nu_{1}, \ldots, \nu_{n+1} \Bigr) \Bigm| (1, \nu_1,\ldots, \nu_{n+1}) \in \conv(\G) \Bigr\} ,
\end{aligned}
\]
where the first equality holds by the definition of $\Phi$, the second equality follows from the partial fraction decomposition, and the last equality follows because convexification commutes with linear transformation. 
Using Theorem~\ref{thm:mainthem}, we obtain 
\[
\begin{aligned}
\conv(\F) &= \Bigl\{(f_0,f_1, \ldots, f_{n+1})\Bigm| (1,f_1, \ldots,f_{n+1}) \in f_{0}\conv\bigl(\Phi(\F)\bigr),\ f_0 \geq 0 \Bigr\} \\
	& = \Bigl\{(f_0,f_1, \ldots, f_{n+1})\Bigm|  \sum_{i \in [n]} \alpha_i f_i  = 1,\  (f_0,f_1, \ldots,f_{n+1}) \in \cone(\G) \Bigr\},
\end{aligned}
\]
where the first equality holds from Theorem~\ref{thm:mainthem}, and the second equality follows from the convex hull description of $\Phi(\F)$. Let $\beta_{0 }$ be the $0^\text{th}$ row of $T$, that is  $\beta_{0 }:= (\beta_{00}, \ldots, \ldots, \beta_{0(n+1)})$. Now, 
\[
\begin{aligned}
\conv(\M_{n+1}) &= \conv(T\F)= T\conv(\F) \\
&= \bigl\{\mu \bigm| \langle \beta_{0 }, T^{-1} \mu \rangle = 1,\ T^{-1}\mu \in \cone(\G) \bigr\} \\
&= \bigl\{\mu \bigm| \mu_0 = 1,\ T^{-1}\mu \in \cone(\G) \bigr\},    
\end{aligned}
\]
where the first equality follows from Lemma~\ref{lemma:change-basis}, the third equality is obtained by using the fact that $\beta_{0i} = \alpha_{i}$ for $i \in [n]$ and $\beta_{00} = \beta_{0(n+1)}=0$.  \Halmos

\subsection{Proof of Theorem~\ref{them:RLTconverge}}
We will show that $\conv(\G^n) = \mathcal{R}^n$. This implies that  
\[
\proj_{(\rho,y)}(\mathcal{R}^n) = \proj_{(\rho,y)}\bigl(\conv(\G^n)\bigr) =\conv\bigl( \proj_{(\rho,y)}(\G^n)\bigr) = \conv(\G),
\]
where the second equality holds since convexification commutes with projection, and the last equality holds  since $\proj_{(\rho,y)}(\G^n) = \G$. 
 For each $i \in [m]$, we obtain that $\proj_u(\G^n_i)  = U$.  Since $U$ is a set of affinely independent points, it follows from Lemma~\ref{lemma:common_simplex} that the convex hull of $\G^n$ is $\bigl\{ (\rho,y,u) \bigm| (\rho^i,y^i,u) \in \conv(\G^n_i) \; \forall i \in [m] \bigr\}$. The proof is complete if we show that the convex hull of $\G^n_i$ is $\mathcal{R}^n_i$. We prove this by invoking Theorem~\ref{thm:fractionalmain}.  Let $f^i(x) = (1,x,v^i)$, where $v^i_S = \prod_{j \in S}x_j(a_{i0}+a_i^\intercal x)$ for nonempty $S \subseteq [n]$, and let $\F^i = \bigl\{f^i(x) \mid x \in \{0,1\}^n\bigr\}$.  By Theorem 3 in~\cite{sherali1990hierarchy}, the $n^\text{th}$ level RLT of $\X$ describes the convex hull of $U$. Thus, we obtain 
\[
\begin{aligned}
    \conv(\F^i) = & \Bigl\{\bigl(1,x^i,v^i\bigr) \Bigm|  z^i \in \RLT_n(\X),\ x^i_j = z^i_{\{j\}} \text{ for } j\in[n],\\
    &v^i_S = a_{i0}+\sum_{j \in S}\Bigl(a_{ij} z^i_{S\cup \{j\}}\Bigr) \text{ for } S \subseteq [n] \text{ s.t. } 1 \leq \vert S \vert \leq k  \Bigr\}.
\end{aligned}
\]
Let $\alpha_i= (a_{i0},a_i,0)$. Then, $a_{i0}+ a_i^\intercal x = \alpha_i^\intercal f^i(x)$. Defining $y^i = \rho^ix^i$, $u = \rho^iv^i$ and $w^i = \rho^iz^i$, Theorem~\ref{thm:fractionalmain} shows that $\conv(\G^n_i)$ is $\mathcal{R}^n_i$.  \Halmos

%Using $\alpha = (1,a_i,0)$, $y^i = \rho^ix$, $u = \rho^iv^i$ and $w^i = \rho^iz^i$, Theorem~\ref{thm:fractionalmain} shows that $\conv(\G^n_i)$ is $\mathcal{R}^n_i$. \Halmos

\subsection{Proof of Proposition~\ref{prop:QP-MC}}
Since every point in $\mathcal{R}_{\QP}$ satisfies \eqref{eq:normalize} and \eqref{eq:domainPolytope}, it suffices to show that the McCormick constraints used to relax \eqref{eq:ybilinear} are implied in $\mathcal{R}_{\QP}$. Towards this end, we will show a more general result. Assume $\alpha^\intercal x\le \bar{\alpha}$ and $\beta^\intercal x \le \bar{\beta}$ are implied by $\overline{C}x\le \bar{d}$.
We show that $\mathcal{R}_{\QP}$ implies the inequality $\beta^\intercal W^i \alpha -\bar{\alpha}\beta^\intercal y^i - \bar{\beta}(y^i)^\intercal \alpha + \rho^i\bar{\alpha}\bar{\beta}\ge 0$, which is obtained by linearizing $\bigl(\alpha^\intercal x-\bar{\alpha}\bigr)\bigl(\beta^\intercal x -\bar{\beta}\bigr)\ge 0$ and then homogenizing the resulting inequality using $\rho^i$. Since the two inequalities in the product are implied by the linear inequalities, there exists an $\omega\ge 0$ and a $\nu\ge 0$ such that $\overline{C}^\intercal \omega = \alpha$, $\bar{d}^\intercal \omega \le \bar{\alpha}$,  $\nu^\intercal \overline{C} = \beta^\intercal$, and $\nu^\intercal \bar{d} \le \bar{\beta}$. Therefore, it follows that every point in $\mathcal{R}_{\QP}$ satisfies:
\begin{equation}\label{eq:homogenizelinearize}
\begin{aligned}
    &\nu^\intercal \overline{C}W^i\overline{C}^\intercal \omega - \nu^\intercal \bar{d}\Bigl((y^i)^\intercal \overline{C}^\intercal-\rho^i\bar{d}^\intercal\Bigr)\omega - \nu^\intercal \overline{C}y^i\bar{d}^\intercal\omega \ge 0\\    
    &\Rightarrow \beta^\intercal W^i \alpha - \bar{\beta} \Bigl((y^i)^\intercal \overline{C}^\intercal-\rho^i\bar{d}^\intercal\Bigr)\omega -\beta^\intercal y^i\bar{d}^\intercal\omega \ge 0\\ 
    &\Rightarrow \beta^\intercal W^i \alpha - \bar{\beta}(y^i)^\intercal\alpha - \bigl(\beta^\intercal y^i-\rho^i\bar{\beta}\bigr)\bar{d}^\intercal\omega  \ge 0\\
    &\Rightarrow \beta^\intercal W^i \alpha - \bar{\beta}(y^i)^\intercal \alpha - \bigl(\beta^\intercal y^i-\rho^i\bar{\beta}\bigr)\bar{\alpha}  \ge 0,
\end{aligned}
\end{equation}
where the first implication is because $\bigl(\nu^\intercal \bar{d}-\bar{\beta}\bigr)\bigl((y^i)^\intercal \overline{C}^\intercal-\rho^i\bar{d}^\intercal\bigr)\omega\ge 0$, and the third implication is because $\bigl(\beta^\intercal y^i-\rho^i\bar{\beta}\bigr)\bigl(\bar{d}^\intercal\omega - \bar{\alpha}\bigr)\ge 0$. Now, this shows that the homogenization of the linearization of $\bigl(x_j(L)-x_j\bigr)\Bigl(a_{i0}+a_i^\intercal x - \frac{1}{\rho^i(L)}\Bigr)\ge 0$ is implied in $\mathcal{R}_{\QP}$, which in turn implies that $\mathcal{R}_{\QP}$ satisfies
\begin{align*}
    &-e_j^\intercal W^i a_i + x_j(L)(y^i)^\intercal a_i + \bigl(y^i_j-\rho^ix_j(L)\bigr)\Bigl(\frac{1}{\rho^i(L)}-a_{i0}\Bigr)  \ge 0\\
    &\Rightarrow -e_j^\intercal W^i a_i + x_j(L) + \frac{y^i_j}{\rho^i(L)} - \frac{\rho^ix_j(L)}{\rho^i(L)} -a_{i0}y^i_j \ge 0\\
    &\Rightarrow -x_j + x_j(L) + \frac{y^i_j}{\rho^i(L)} - \frac{\rho^ix_j(L)}{\rho^i(L)} \ge 0\\
    &\Rightarrow -\rho^i(L)x_j + \rho^i(L)x_j(L) + y^i_j - \rho^i x_j(L)\ge 0,
\end{align*}
which is one of the McCormick inequalities. The first implication above is because $\rho^ia_{i0}+(y^i)^\intercal a_i = 1$, and the second implication is because $x_j = a_{i0}y^i_j + e_j^\intercal W^ia_i$. The remaining McCormick inequalities follow using a similar argument with other bounds on $x_j$ and $\rho^i$.

\Halmos

\subsection{Example illustrating $\mathcal{R}_{\QP} \subsetneq \mathcal{R}_{\textsf{LEF}}$ }\label{app:strictcontainment}
\revision{We give an example to illustrate that the containment may be strict. Let $m=1$, $n=2$, $a_{0}=a_{1}=a_{2}=1$ and consider the unconstrained BFP, in which case $\mathcal{R}_{\QP}=\conv(\G^1)$ because when $n=2$ the 1st-level RLT relaxation $\widehat{F}$ coincides with the convex hull of $\F$. We verify that $p = (\rho,y_1,y_2,x_1,x_2) = (\frac{1}{2},\frac{1}{4},\frac{1}{4},\frac{1}{4},\frac{1}{4})\in \mathcal{R}_{\textsf{LEF}}\setminus \conv(\G^1)$ when taking $x_{j}(L)=0$, $x_{j}(U)=1$, for $j=1,2$ and $\rho(L)=\frac{1}{3}$, $\rho(U)=1$. To verify that $p \not\in\conv(\G^1)$, observe that $x_1 = \frac{x_1 (1+x_1+x_2)}{1+x_1+x_2} \ge \frac{2x_1}{1+x_1+x_2} = 2y_1$. To see that $p\in \mathcal{R}_{\textsf{LEF}}$, it follows easily that $p$ satisfies \eqref{eq:normalize}. Moreover, 
\begin{equation*}
    \begin{pmatrix}
    \rho\\
    x_j\\
    y_j
    \end{pmatrix}
    = \dfrac{3}{4} \begin{pmatrix} 
    \rho(L)\\
    x_j(L)\\
    x_j(L)\rho(L)
    \end{pmatrix}
    + \dfrac{1}{4} \begin{pmatrix}
        \rho(U)\\
        x_j(U)\\
        x_j(U)\rho(U)
    \end{pmatrix}
\end{equation*}
which implies that $y_j$ satisfies the McCormick inequalities in \eqref{eq:LFP-MC} which relax $y_j = x_j\rho$.}

\subsection{Proof of Proposition~\ref{prop:genMcCormick}}
All such inequalities are obtained as linearizations of products of certain inequalities. Therefore, we will first show a more general result. Assume that the inequality $\gamma_0\rho^i + \gamma^\intercal y^i - \bar{\gamma}\le 0$ is implied by $\overline{C}\bar{y}\le \rho^i \bar{d}$ and $a_{i0}\rho^i + a_i^\intercal y^i = 1$, and $\beta^\intercal x - \bar{\beta}\le 0$ is implied by $\overline{C}x\le \bar{d}$. Then, we will show that the homogenization of 
\begin{align}
    &\bigl(\gamma_0\rho^i + \gamma^\intercal y^i - \bar{\gamma}\bigr)\bigl(\beta^\intercal x -\bar{\beta}\bigr)\ge 0\label{eq:generalnonlinearproduct}\\
    &\text{expressed as } \gamma_0(y^i)^\intercal \beta + \gamma^\intercal W^i\beta - \bar{\gamma}x^\intercal \beta - \bar{\beta}\gamma_0\rho^i - \bar{\beta}\gamma^\intercal y^i + \bar{\gamma}\bar{\beta} \ge 0,\label{eq:generalMcCormick}
\end{align}
is implied by the constraints in $\mathcal{R}_{\QP}$. Since the two inequalities whose product is homogenized are implied by linear inequalities in $\X$, there exists an $\omega \ge 0$, $\theta\in\R$, and a $\nu\ge 0$ such that
$\overline{C}^\intercal \omega + \theta a_i = \gamma$,
$-\bar{d}^\intercal \omega + \theta a_{i0} = \gamma_0$, $\theta \le \bar{\gamma}$, $\nu^\intercal \overline{C} = \beta$, and $\nu^\intercal \bar{d}\le\bar{\beta}$. 
It follows that $0\ge x^\intercal \overline{C}^\intercal \omega -\bar{d}^\intercal\omega = x^\intercal (\gamma-\theta a_i) + (\gamma_0 -\theta a_{i0})$.
Now, points in $\mathcal{R}_{\QP}$ satisfy:
\begin{align}
&(\gamma-\theta a_i)^\intercal W^i \beta - (\theta a_{i0}-\gamma_{0})(y^i)^\intercal \beta - (\gamma-\theta a_i)^\intercal y^i\bar{\beta}+\rho^i(\theta a_{i0}-\gamma_0)\bar{\beta}\notag  \\
&\quad + \theta \Bigl(\bigl(a_{i0}y^i +a_i^\intercal W^i - x^\intercal \bigr) \beta - \bigl(a_{i0}\rho^i + a_i^\intercal y^i - 1\bigr)\bar{\beta}\Bigr)\ge 0\notag\\
&\Rightarrow \gamma^\intercal W^i\beta + \gamma_0(y^i)^\intercal\beta - \gamma^\intercal y^i\bar{\beta} - \rho^i\gamma_0\bar{\beta} -\theta (x^\intercal \beta -\bar{\beta})\ge 0\notag\\
&\Rightarrow \gamma^\intercal W^i\beta + \gamma_0(y^i)^\intercal\beta - \gamma^\intercal y^i\bar{\beta} - \rho^i\gamma_0\bar{\beta} - \bar{\gamma} (x^\intercal \beta -\bar{\beta})\ge 0.\notag%\label{eq:homogenizeyx}
\end{align}
The first inequality follows since $\mathcal{R}_{\QP}$ satisfies $a_{i0}y^i +a_i^\intercal W^i = x^\intercal$, $a_{i0}\rho^i + a_i^\intercal y^i =1$, and \eqref{eq:homogenizelinearize}, where $\alpha = \gamma - \theta a_i$ and $\bar{\alpha}=\theta a_{i0}-\gamma_0$. The modified form of \eqref{eq:homogenizelinearize} holds because $x^\intercal (\gamma-\theta a_i) + (\gamma_0 -\theta a_{i0})\le 0$ and $\beta^\intercal x - \bar{\beta}\le 0$ are a consequence of $\overline{C}x\le \bar{d}$. The second implication follows since $\bigl(\theta-\bar{\gamma}\bigr)\bigl(x^\intercal \beta-\bar{\beta}\bigr)\ge 0$.
Replacing $z_i$ with $b_{i0}\rho^i + b^\intercal y^i$ and $d_i$ with $a_{i0}+a_i^\intercal x$, we reduce $\bigl(z_i-z_i^L\bigr)\bigl(d_i - d_i^L\bigr)\ge 0$ into an inequality of the form \eqref{eq:generalnonlinearproduct}. Then, our earlier argument shows that its linearization, which yields one of the McCormick inequalities, is implied by the inequality corresponding to \eqref{eq:generalMcCormick}, after $b_{i0}\rho^i + b^\intercal y^i$ (resp. $b_{i0}y^\intercal + b^\intercal W^i$) is replaced by $z_i$ (resp. variables representing $z_ix^\intercal$). A similar argument shows that the remaining McCormick inequalities are also implied in $\mathcal{R}_{\QP}$.\Halmos

\subsection{Disaggregation of products}\label{ec:disaggregate}
\begin{proposition}\label{eq:disaggregate}
Consider the equality $z_i(a_{i0} + a_i^\intercal x) = b_{i0}+b_i^\intercal x$, where $z_i$ represents $\frac{b_{i0}+b_i^\intercal x}{a_{i0}+a_i^\intercal x}$. Assume $h_{ij}$ represents  $z_ix_j$. Then, $\mathcal{R}_{\QP}$ implies the following inequalities:
\begin{enumerate}
    \item McCormick inequalities that relax $h_{ij} = z_ix_j$ based on bounds of $z_i$ and $x_j$  obtained over $\overline{C}x\le \bar{d}$. 
    \item The equality $a_{i0}z_i + a_i^\intercal h_i = b_{i0}+b_i^\intercal x$.
\end{enumerate}
\end{proposition}
\begin{proof}{Proof.}
We have that:
\begin{align*}
 b_i^\intercal W^i_{\cdot j} + b_{i0}y^i_j - (b_{i0}\rho^i + b_i^\intercal y^i) x_j(L) - z_i^L (x_j-x_j(L)) \ge 0 \\
 \Rightarrow h_{ij} - z_ix_j(L) - z_i^L(x_j-x_j(L))\ge 0.
\end{align*}
    The first inequality follows since \eqref{eq:generalMcCormick} is implied in $\mathcal{R}_{\QP}$ with $\beta$ as the $j^{\text{th}}$ standard vector, $\bar{\beta} = x_j(L)$, $\gamma= b_i$, $\gamma_0 = b_{i0}$, and $\bar{\gamma} = z_i^L$. This is because $x_j - x_j(L)\ge 0$ and $b_{i0}\rho^i + b_i^\intercal y^i - z_i^L \ge 0$ are implied by $\overline{C}x\le \bar{d}$. The  implication follows since we will represent $b_i^\intercal W^i_{\cdot j} + b_{i0}y^i_j$ as $h_{ij}$ and $b_{i0}\rho^i + b_i^\intercal y^i$ as $z_i$. This shows that one of the McCormick inequalities is implied in $\mathcal{R}_{\QP}$. The proof for the remaining McCormick inequalities is similar. It also follows that:
    \begin{equation*}
        a_{i0}z_i + h_i^\intercal a_i = b_{i0}a_{i0}\rho^i + a_{i0}b_i^\intercal y^i + b_i^\intercal W^ia_i + b_{i0}(y^i)^\intercal a_i = 
        b_{i0} + b_i^\intercal x,
    \end{equation*}
    where the first equality is by the definitions of $z_i$ and $h_i$ given above and the second equality is because $W^i a_i + a_{i0} y^i = x$ and $a_{i0}\rho^i + (y^i)^\intercal a_i = 1$. \Halmos
\end{proof}

\subsection{Proof of Proposition~\ref{prop:SQPconic}}
Clearly, \eqref{eq:BQP-SDPa} implies \eqref{eq:BQP-SDP} since the matrix in the latter constraint is a principal submatrix of the matrix in \eqref{eq:BQP-SDPa}. The following shows the reverse implication:
\begin{equation*}
    \begin{alignedat}{3}
    &\eqref{eq:BQP-SDP}&&\Rightarrow\begin{pmatrix}a_{i0} & a_i^\intercal\\ 1 & 0\\0 & I\end{pmatrix}
    \begin{pmatrix}\rho^i & (y^i)^\intercal \\ y^i & W^i \end{pmatrix}
    \begin{pmatrix}a_{i0} & 1 & 0 \\ a_i & 0 & I\end{pmatrix}
    \succeq 0, ~~~\forall i\in[m]\\
    &&&\Rightarrow\begin{pmatrix}
    a_{i0}^2\rho^i+a_{i0}a_i^\intercal y^i+(a_{i0}(y^i)^\intercal +a_i^\intercal W^i)a_i  & ~~a_{i0}\rho^i+a_i^\intercal y^i & ~~a_{i0}(y^i)^\intercal+a_i^\intercal W^i\\
    a_{i0}\rho^i+a_i^\intercal y^i & \rho^i & (y^i)^\intercal\\
    a_{i0}y^i+W^ia_i & y^i & W^i
    \end{pmatrix}\succeq 0,\forall i\in[m]\\
    &&&\Rightarrow \eqref{eq:BQP-SDPa}.
    \end{alignedat}    
    \end{equation*} 
The last implication follows because $a_{i0}\rho^i+a_i^\intercal y^i=1$ and $a_{i0}(y^i)^\intercal + a_i^\intercal W^i = x^\intercal$.
\Halmos

\bibliographystyle{spmpsci}      % mathematics and physical sciences
\bibliography{reference} 

\begin{comment}

\end{comment}

\end{document}